\documentclass[12pt]{article}

\usepackage[english]{babel}
 \usepackage[T1]{fontenc}
\usepackage{fancyhdr}
\usepackage{indentfirst}
\usepackage{graphicx}
\usepackage{newlfont}
\usepackage{cite}
\usepackage{lipsum}
\usepackage{float}
\usepackage{wrapfig}
\usepackage{array}

\usepackage{longtable}
\usepackage{amsfonts}
\usepackage{amssymb}
\usepackage{amsmath}
\usepackage{latexsym}
\usepackage{amsthm}
\usepackage[margin=1.2in]{geometry}

\theoremstyle{plain}                   
\newtheorem{theorem}{Theorem}[section]
\numberwithin{theorem}{section}
\newtheorem{lem}[theorem]{Lemma}
\newtheorem{cor}[theorem]{Corollary}
\newtheorem{prop}[theorem]{Proposition}
\theoremstyle{remark}
\newtheorem{rem}{Remark}[section]
\numberwithin{rem}{section}
\newtheorem{ex}{Example}[section]

\theoremstyle{definition}
\newtheorem{defn}{Definition}[section]

\begin{document}
\title{\LARGE\textbf{Dimension estimates for Kakeya sets defined in an axiomatic setting} }
\author{Laura Venieri}
\date{}

\maketitle

\begin{abstract}
In this dissertation we define a generalization of Kakeya sets in certain metric spaces. Kakeya sets in Euclidean spaces are sets of zero Lebesgue measure containing a segment of length one in every direction. A famous conjecture, known as Kakeya conjecture, states that the Hausdorff dimension of any Kakeya set should equal the dimension of the space. It was proved only in the plane, whereas in higher dimensions both geometric and arithmetic combinatorial methods were used to obtain partial results.

In the first part of the thesis we define generalized Kakeya sets in metric spaces satisfying certain axioms. These allow us to prove some lower bounds for the Hausdorff dimension of generalized Kakeya sets using two methods introduced in the Euclidean context by Bourgain and Wolff. With this abstract setup we can deal with many special cases in a unified way, recovering some known results and proving new ones. 

In the second part we present various applications. We recover some of the known estimates for the classical Kakeya and Nikodym sets and for curved Kakeya sets. Moreover, we prove lower bounds for the dimension of sets containing a segment in a line through every point of a hyperplane and of an (n-1)-rectifiable set. We then show dimension estimates for Furstenberg type sets (already known in the plane) and for the classical Kakeya sets with respect to a metric that is homogeneous under non-isotropic dilations and in which balls are rectangular boxes with sides parallel to the coordinate axis. Finally, we prove lower bounds for the classical bounded Kakeya sets and a natural modification of them in Carnot groups of step two whose second layer has dimension one, such as the Heisenberg group. On the other hand, if the dimension is bigger than one we show that we cannot use this approach.
\end{abstract}

\tableofcontents

\section{Introduction}

Kakeya sets (also known as Besicovitch sets) in $\mathbb{R}^n$ are sets of zero Lebesgue measure containing a line segment of unit length in every direction. Their study originated from a question of Kakeya, who asked to determine the smallest area in which a unit line segment can be rotated 180 degrees in the plane. Besicovitch \cite{Besicovitch} constructed such a set with arbitrarily small area. Since then, these sets in general $\mathbb{R}^n$ have been studied extensively: in particular, it is conjectured that they should have full Hausdorff dimension but it was proved only in the plane (Davies,\cite{Davies1971}).

Several approaches have been used to get lower bounds for their Hausdorff dimension: Bourgain developed a geometric method \cite{Bourgain1991}, improved then by Wolff \cite{MR1363209}; later on, Bourgain himself \cite{Bourgarithm} introduced an arithmetic combinatorial method, improved by Katz and Tao \cite{Katz&Tao1}.

For a more complete discussion on the results concerning Kakeya sets see \cite{Mattila} (Chapters 11,22,23), where connections to other important questions in modern Fourier analysis are described.

In this thesis we define Kakeya sets in an axiomatic setting in which we can prove estimates for their Hausdorff dimension by suitably modifying Bourgain's and Wolff's geometric arguments.  The idea is to enlighten the geometric aspects of the methods, enclosing them in five axioms that can then be verified in some special cases. Moreover, this approach allows us to deal with many special cases in a unified way. 

The setting is a complete separable metric space $(X,d)$, which is the ambient space, endowed with an upper Ahlfors $Q$-regular measure $\mu$, and another metric space $(Z,d_Z )$ with a compact subset $Y \subset Z$, which is the space of directions ($Z$ is endowed with a measure $\nu$ satisfying \eqref{BS}). We define analogues of Kakeya sets as subsets of $X$ containing certain subsets $F_u(a)$ of $X$ (corresponding to segments in the classical case) associated to every direction $u \in Y$ and some $a \in \mathcal{A}$, which is a space of parameters (see Section 2 for details). Tubes are defined as $\delta$ neighbourhoods of some objects $I_u(a) \supset F_u(a)$. We assume that they satisfy certain axioms that contain the geometric features (such as the $\mu$ measure of the tubes and the way they intersect) required to define a suitable Kakeya maximal function and to use the geometric methods mentioned above to prove certain $L^p$ estimates for it, which imply lower bounds for the Hausdorff dimension of Kakeya sets.

Modifying Bourgain's method we obtain a weak type $L^p$ estimate for $p=\frac{S+2}{2}$ (see Theorem \ref{Bourgain}), which implies a certain lower bound for the Hausdorff dimension. The proof proceeds in the same way as in the classical case, where it yields the lower bound $\frac{n+1}{2}$ for the Hausdorff dimension of Kakeya sets.

Wolff's method requires a more complicated geometric assumption (Axiom 5), which we were not able to obtain from simpler hypothesis. When this is verified we prove another $L^p$ estimate (Theorem \ref{Wolff}), which yields an improvement of Bourgain's bound in the classical case ($\frac{n+2}{2}$) and here only in some cases.

If one can thus show that a certain setting satisfies the axioms, one obtains estimates for the dimension of Kakeya sets in that setting. We show some examples (apart from the classical Kakeya sets in $\mathbb{R}^n$ with the Euclidean metric), recovering some known results and proving new ones. We recover the known dimension estimates ($\frac{n+2}{2}$) for Nikodym sets, which were originally proved by Bourgain and Wolff. Nikodym sets are subsets of $\mathbb{R}^n$ having zero measure and containing a segment of unit length in a line through every point of the space.  We prove the same lower bound  for the Hausdorff dimension of sets containing a segment in a line through every point of a hyperplane (see Theorem \ref{VNik}). Another variant of Nikodym sets are sets  containing a segment in a line through almost every point of an $(n-1)$-rectifiable set with direction not contained in the approximate tangent plane. We first reduce the problem to Lipschitz graphs and then we prove the lower bound $\frac{n+2}{2}$ also for the Hausdorff dimension of these sets  (see Theorem \ref{ENik}), which is to our knowledge a new result.

We also recover the known dimension estimates for curved Kakeya and Nikodym sets, which were originally proved by Bourgain \cite{Bourg_curves} and Wisewell \cite{Wisewell}. Moreover, we consider Kakeya sets with segments in a restricted set of directions. These were considered by various authors before and Bateman \cite{Bateman} and Kroc and Pramanik \cite{KrocPram2} characterized those sets of directions for which the Nikodym maximal function is bounded. Mitsis \cite{Mitsis} proved that sets in the plane containing a segment in every direction of a subset $A$ of the sphere have dimension at least the dimension of $A$ plus one. Here we show in Theorem \ref{AKak} that a subset of $\mathbb{R}^n $, $n \ge 3$, containing a segment in every direction of an Ahlfors $S$-regular subset of the sphere, $S \ge 1$, has dimension greater or equal to $\frac{S+3}{2}$. 

We recover the lower bound proved by Wolff for the dimension of Furstenberg sets in the plane and prove new lower bounds for them in higher dimensions (see Theorem \ref{Furstdim}). Given $0<s\le 1$, an $s$-Furstenberg set is a compact set such that for every direction there is a line whose intersection with the set has dimension at least $s$. Wolff in \cite{Wolff2} proved that in the plane the Hausdorff dimension of these sets is $\ge \max\{2s,s+\frac{1}{2}\}$.  Our result states that in $\mathbb{R}^n$ the Hausdorff dimension of an $s$-Furstenberg set is at least $ \max\{ \frac{(2s-1)n+2}{2},s\frac{ 4n+3}{7}\}$ when $n\le 8$ and at least $ s\frac{4n+3}{7}$ when $n \ge 9$.  Making a stronger assumption, that is considering sets containing in every direction a rotated and translated copy of an Ahlfors $s$-regular compact subset of the real line, we can improve the previous lower bounds in dimension greater or equal to three, proving Theorem \ref{Furstmod}. In this case the lower bounds are $2s+\frac{n-2}{2}$ for $n \le 8$ and $\max\{ 2s+\frac{n-2}{2}, s\frac{4n+3}{7}\}$ for $n \ge 9$. Here we will see that we have only a modified version of Axiom 1 but we can obtain anyway these dimension estimates.

We then consider two applications in non-Euclidean spaces. 
We first  prove dimension estimates for the usual Kakeya sets but considered in $\mathbb{R}^n=\mathbb{R}^{m_1} \times \dots \times \mathbb{R}^{m_s}$ endowed with a metric $d$ homogeneous under non-isotropic dilations and in which balls are rectangular boxes with sides parallel to the coordinate axis. We show that the Hausdorff dimension with respect to $d$ of any Kakeya set is at least $\frac{6}{11}Q+\frac{5}{11}s$, where $Q= \sum_{j=1}^s j m_j$, and in the case when $m_1=n-1$, $m_2= \dots =m_{s-1}=0$ and $m_s=1$ it is at least $\frac{n+2s}{2}$ when $n \le 12$ (see Theorem \ref{dimKd}). To prove these estimates we will also use a modification of the arithmetic method introduced by Bourgain and developed by Katz and Tao. 

One motivation for this last example comes from the idea of studying Kakeya sets in Carnot groups. The author in \cite{Venieri} has proved that $L^p$ estimates for the classical Kakeya maximal function imply lower bounds for the Hausdorff dimension of bounded Besicovitch sets in the Heisenberg group  $\mathbb{H}^n \cong \mathbb{R}^{2n+1}$ with respect to the Kor\'anyi metric (which is bi-Lipschitz equivalent to the Carnot Carath\'eodory metric). By the results of Wolff and of Katz and Tao one then gets the lower bounds $\frac{2n+5}{2}$ for $n \le 3$ and $\frac{8n+14}{7}$ for $n \ge 4$ for the Heisenberg Hausdorff dimension.

In a similar spirit, it would be interesting to obtain some lower bounds for the Hausdorff dimension of Besicovitch sets in a Carnot group with respect to a homogeneous metric. We will show that the axioms hold in a Carnot group of step $2$ whose second layer has dimension 1, thus we can prove the lower bound $\frac{n+4}{2}$ for the dimension of any bounded Kakeya set with respect to any homogenous metric (see Theorem \ref{m21}). Unfortunately this is not the case for other Carnot groups. We conclude with a negative result, showing that in Carnot groups of step 2 whose second layer has dimension $>1$ endowed with the $d_\infty$ metric (see \eqref{dinfty1}, \eqref{dinfty2}) we cannot use this axiomatic approach.

Moreover, we will consider a modification of the classical Kakeya sets in Carnot groups of step 2, namely sets containing a left translation of every segment through the origin with direction close to the $x_n$-axis. We will show the lower bound $\frac{n+3}{2}$ for their Hausdorff dimension with respect to a homogeneous metric in any Carnot group of step $2$ whose second layer has dimension 1 (see Theorem \ref{m21}).

The thesis is organized as follows. In Part \ref{firstpart} (Sections \ref{firstSection}-\ref{SectionWolff}) we define Kakeya sets in certain metric spaces and prove dimension estimates for them. In particular, in Section \ref{firstSection} we introduce the axiomatic setting and in Section \ref{SectionLp} we show that $L^p$ estimates of the Kakeya maximal function imply lower bounds for the Hausdorff dimension of Kakeya sets and how to discretize those $L^p$ estimates. Section \ref{SectionBourg} contains the generalization of Bourgain's method and Section \ref{SectionWolff} of Wolff's method. In Part \ref{secondpart} (Sections \ref{Kaksets}-\ref{Carnots2}) we explain various examples of applications.  

\section{List of Notation}

Since Part 1 is quite heavy in notation, we make here a list of the main symbols that we will use with a reference to where they are defined and a short description.

\begin{longtable}[l]{p{50pt} p{70pt} p{300pt}}
\textbf{Symbol}	& \textbf{Reference} & \textbf{Description} \\
$(X,d)$	 	& Section \ref{firstSection}	 &Ambient space: complete separable metric space\\
$\mu$	 	& \eqref{uppA}	 &  Upper Ahlfors regular measure on $X$ \\
$Q$	 	& \eqref{uppA}	 &  Upper Ahlfors regularity exponent of $\mu$ \\
$B_d(a,r)$    &  Below    \eqref{uppA}       & Closed ball in the metric $d$ \\
$d'$               &  Section \ref{firstSection}& A second metric on $X$ such that $(X,d')$ is separable\\
$(Z,d_Z)$      &  Section \ref{firstSection}& Metric space containing the space of directions $Y$\\
$Y$               &  Section \ref{firstSection}&  Space of directions: compact subset of $Z$  \\
$\nu$        &\eqref{BS} & Borel measure on $Z$ \\
$S$              & \eqref{BS}& Exponent of the radius in the $\nu$ measure of balls centered in $Y$ \\
$\dim_d$ &\eqref{dimd} & Hausdorff dimension with respect to $d$\\
$\mathcal{A}$ & Section \ref{firstSection} &Set of parameters\\
$F_u(a)$ &\eqref{FI} & Subset of $X$ associated to $a \in \mathcal{A}$ and $u  \in Y$\\
$I_u(a)$ &\eqref{FI} &Subset of $X$ containing $F_u(a)$  \\
$\tilde{I}_u(a)$ & \eqref{FI}& Subset of $X$ containing $I_u(a)$ \\
$\mu_{u,a}$ & \eqref{muua}& Measure on $F_u(a)$ \\
$T^\delta_u(a)$ & \eqref{tube1}& Tube with radius $\delta$ \\
$\tilde{T}^{W \delta}_u(a)$ & Below \eqref{tube1} & Tube with radius $W\delta$ \\
$T$ & Axiom 1 & Exponent of $\delta$ in the $\mu$ measure of a tube\\
$\theta$ &Axiom 2  & Exponent of $\delta$ appearing in \eqref{ax2}\\
$W$ & Axiom 4 &Constant appearing in the radius of larger tubes \\
$f^d_\delta$ &\eqref{fddelta} &Kakeya maximal function with width $\delta$\\
$\tilde{f}^d_{W \delta}$ & \eqref{fd2delta} & Kakeya maximal function on tubes with radius $W\delta$\\
$\alpha$ & Axiom 5 & Constant appearing in the exponent of $\gamma$ in \eqref{axiom5def} \\
$\lambda$ & Axiom 5 & Constant appearing in the exponent of $\delta$ in \eqref{axiom5def}\\
\end{longtable}

\part{ Definition and dimension estimates for generalized Kakeya sets}\label{firstpart}
\section{Axiomatic setting and notation}\label{firstSection}

Let $(X,d)$ be a complete separable metric space endowed with a Borel measure $\mu$ that is upper Ahlfors $Q$-regular, $Q > \frac{1}{2}$, that is there exists $0<C_0< \infty$ such that
\begin{equation}\label{uppA}
 \mu (B_d(a,r)) \le C_0 r^Q,
\end{equation}
for every $a \in X$ and every $r < \text{diam}_d (X)$ (we denote by $B_d(a,r)$ the closed ball in the metric $d$ and by $\text{diam}_d (X)$ the diameter of $X$ with respect to $d$). 
Let $d'$ be another metric on $X$ such that $(X,d')$ is separable. Note that in most applications $d$ and $d'$ will be equal whereas they will be different (and not bi-Lipschitz equivalent) in Section \ref{Rnd}, where we consider the classical Kakeya sets in $\mathbb{R}^n$ endowed with a metric $d$ homogeneous under non-isotropic dilations, and in Section \ref{Carnots2}, where we consider Kakeya sets and a modification of them in Carnot groups of step two. In these cases $d'$ will be the Euclidean metric and $d$ the homogenous metric. With this choice the diameter estimate in Axiom 3 below holds, whereas it would not if we used only one metric $d$.

Let $(Z, d_Z)$ be a metric space and let $Y  \subset Z$ be compact. 
Let $\nu$ be a Borel measure on $Z$ such that $0<\nu(Y) \le 1$ and there exist $S$, $1 \le S < 2Q $, and two constants $0 < \tilde{c}_0 \le \tilde{C}_0 < \infty$ such that
\begin{equation}\label{BS}
\tilde{c}_0 r^S \le \nu(B_{d_Z}(u,r)) \le \tilde{C}_0 r^S,
\end{equation}
for every $u \in Y$ and $r< \text{diam}_{d_Z}(Y)$. Note that $Y$ is in general not Ahlfors regular since the measure $\nu$ is not supported on $Y$.

We will denote the $s$-dimensional Hausdorff measure with respect to $d$ by $\mathcal{H}^s_d$, $s\ge0$. We recall that this is defined for any $A \subset X$ by
\begin{equation*}
\mathcal{H}^s_d(A)= \lim_{\delta \rightarrow 0}\mathcal{H}_\delta^s(A),
\end{equation*}
where for $\delta >0$
\begin{equation*}
\mathcal{H}_\delta^s(A) =\inf \left\{ \sum_i \text{diam}_d(E_i)^s : A \subset \bigcup_i E_i, \text{diam}_{d}(E_i) < \delta \right\}.
\end{equation*}
The Hausdorff dimension of a set $A \subset X$ with respect to $d$ is then defined in the usual way as
\begin{equation}\label{dimd}
\dim_d A= \inf \{ s: \mathcal{H}^s_d(A)=0 \} = \sup \{ s: \mathcal{H}^s_d(A)= \infty \}.
\end{equation}
Observe that the Hausdorff dimension of $X$ (with respect to $d$) is $\ge Q$.
We will consider also the Hausdorff measures with respect to the metric $d'$, which we will denote by $\mathcal{H}^s_{d'}$.

The notation $A\lesssim B$ (resp. $A \gtrsim B$) means $A\le C B$ (resp. $A\ge C B$), where $C$ is a constant (depending on $Q$, $S$ and other properties of the spaces $X$ and $Y$); $A \approx B$ means $A \lesssim B$ and $A \gtrsim B$. If $p$ is a given parameter, we denote by $C_p$ a constant depending on $p$.
For $A \subset X$, the characteristic function of $A$ is denoted by $\chi_A$.

Let $\mathcal{A}$ be a set of parameters (we do not need any structure on $\mathcal{A}$).
To every $a \in \mathcal{A}$ and every $u \in Y$ we associate three sets 
\begin{equation}\label{FI}
F_u(a) \subset I_u(a) \subset \tilde{I}_u(a) \subset X
\end{equation}
 such that $c \le \text{
diam}_{d'}(I_u(a))\le c'$ (where $0<c \le c' < \infty$ are constants) and $  \text{diam}_{d'}(\tilde{I}_u(a))\le \bar{c} \  \text{diam}_{d'}(I_u(a))$ for some other constant $1 \le \bar{c}< \infty$. Moreover, there exists a measure $\mu_{u,a}$ on  $F_u(a)$ such that $\mu_{u,a}(F_u(a))=1$ and it satisfies the doubling condition, that is 
\begin{equation}\label{muua}
\mu_{u,a}(F_u(a) \cap B_d(x,2r)) \le C \mu_{u,a}(F_u(a) \cap B_d(x,r))
\end{equation}
for every $a \in \mathcal{A}$, $u \in Y$ and $x \in F_u(a)$. The measures $\mu$ and $\mu_{u,a}$ are not assumed to be related, but they need to satisfy Axiom 2 below. In all applications that we will consider $\mu$ will be the Lebesgue measure on $\mathbb{R}^n$. In most applications $\mu_{u,a}$ will be the $1$-dimensional Euclidean Hausdorff measure on $F_u(a)$, which will be a segment or a piece of curve. Only in the case of Furstenberg type sets (Section \ref{Furstenberg sets}) $\mu_{u,a}$ will be an (upper) Ahlfors $s$-regular measure for some $0<s\le 1$.

Given $0 < \delta < 1$, let $T^\delta_u(a)$ be the $\delta$ neighbourhood of $I_u(a)$ in the metric $d$,
\begin{equation}\label{tube1}
T^\delta_u(a)= \{ x \in X: d(x,I_u(a)) \le \delta \},
\end{equation}
which we will call a tube with radius $\delta$. Moreover, we define tubes $\tilde{T}^{W \delta}_u(a)$ with radius $W \delta$ as $W \delta$ neighbourhoods of $\tilde{I}_u(a) $, where $W$ is the constant such that Axiom 4 below holds.

Note that in the case of the classical Kakeya sets the setting is the following: $X=\mathbb{R}^n$, $d=d'=d_E$ is the Euclidean metric, $\mu=\mathcal{L}^n$ is the Lebesgue measure thus $Q=n$; $Z=Y=S^{n-1}$ is the unit sphere, $d_Z$ is the Euclidean metric on the sphere, $\nu=\sigma^{n-1}$ is the spherical measure thus $S=n-1$. Moreover, $\mathcal{A}=\mathbb{R}^n$ and for every $e \in S^{n-1}$ and $a \in \mathbb{R}^n$ $F_e(a)=I_e(a)$ is the segment with midpoint $a$, direction $e$ and length $1$, whereas $\tilde{I}_e(a)$ is the segment  with midpoint $a$, direction $e$ and length $2$. The measure $\mu_{e,a}$ is the Euclidean $1$-dimensional Hausdorff measure $\mathcal{H}^1_E$ on $I_e(a)$. Then the tubes are Euclidean $\delta$ neighbourhoods of these segments and satisfy the axioms which we assume here (we will see this briefly in Remark \ref{axEucl}).

We assume that the following axioms hold:
\begin{enumerate}
\item[(\textbf{Axiom 1})]The function $u \mapsto \mu(T^\delta_u(a))$ is continuous and there exist $\frac{S}{2}<T<Q$ and two constants $0<c_1\le c_2< \infty$ such that for every $a \in \mathcal{A}$, every $u \in Y$ and $\delta>0$
 \begin{equation*}
c_1 \delta^T \le \mu (T^\delta_u(a)) \le \mu (\tilde{T}^{W \delta}_u(a)) \le c_2 \delta^T, 
\end{equation*}
and if $A \subset \tilde{T}^{W\delta}_u(a)$ then $\mu (A) \le c_2 \text{diam}_{d'} (A) \delta^T$.

\item[(\textbf{Axiom 2})] There exist three constants $0 \le \theta< \frac{2Q-2T+S}{S+2}$, $0 <  K' < \infty$, $1  \le K < \infty$, such that for every $a \in \mathcal{A}$, $u \in Y$, $x \in F_u(a)$,  if $\delta \le r \le 2\delta$ and 
\begin{equation*}
\mu_{u,a}(F_u(a) \cap  B_d(x,r)) = M
\end{equation*}
for some $M >0$, then
\begin{equation}\label{ax2}
\mu (T^\delta_u(a) \cap B_d(x, Kr)) \ge K'M \delta^\theta \mu (T^\delta_u(a)).
\end{equation}
\item[(\textbf{Axiom 3})] There exists a constant $b >0$ such that for every $a,a' \in \mathcal{A}$, every $u,v \in Y$ and $\delta >0$
\begin{equation}\label{ax3}
\text{diam}_{d'}(\tilde{T}^{W\delta}_u(a) \cap \tilde{T}^{W\delta}_v(a')) \le b \frac{\delta}{d_Z(u,v)}.
\end{equation}
\item[(\textbf{Axiom 4})] There exist two constants $0 <W, \bar{N}< \infty$ such that for every $u,v \in Y$ with $u \in B_{d_Z}(v, \delta)$ and for every $a \in \mathcal{A}$, $T^\delta_u(a)$ can be covered by tubes $\tilde{T}^{W\delta}_v(b_k)$, $k=1, \dots, N$, with $N \le \bar{N}$.  

\end{enumerate}

Observe that in the case of the classical Kakeya sets $\theta=0$ and this will hold also in all other applications presented here, except for Furstenberg sets (see Section \ref{Furst}). The bound $\theta< \frac{2Q-2T+S}{S+2}$ ensures that the dimension lower bound proved later in Theorem \ref{Bourgain} is positive.

\begin{defn}\label{Kakeyasets}
We say that a set $B \subset X$ is a \textit{generalized Kakeya} (or \textit{Besicovitch}) \textit{set} if $\mu(B)=0$ and for every $u \in Y$ there exists $a \in \mathcal{A}$ such that $F_u(a) \subset B$.
\end{defn}

Note that the definition might be vacuous in certain contexts since it is possible that generalized Kakeya sets of null measure do not exist. In the applications we will see examples of cases where they exist.

Analogously to the classical Kakeya maximal function, we define for $0<\delta <1$ and $f \in L^1_{loc}(X, \mu)$ the \textit{Kakeya maximal function with width $\delta$ related to $d$} as $f^d_\delta: Y \rightarrow [0, \infty]$,
\begin{equation}\label{fddelta}
f^d_\delta(u) = \sup_{a \in \mathcal{A}} \frac{1}{\mu(T^\delta_u(a))} \int_{T^\delta_u(a)} |f| d \mu.
\end{equation}
Similarly, we define the Kakeya maximal function on tubes with radius $W \delta$ as $\tilde{f}^d_{W\delta}: Y \rightarrow [0, \infty]$,
\begin{equation}\label{fd2delta}
\tilde{f}^d_{W\delta}(u) = \sup_{a \in \mathcal{A}} \frac{1}{\mu(\tilde{T}^{W\delta}_u(a))} \int_{\tilde{T}^{W\delta}_u(a)} |f| d \mu.
\end{equation}

To be able to apply Wolff's method we will need another axiom, which we will introduce in Section 5.

We recall here the $5r$-covering theorem, which we will use several times. For the proof see for example Theorem 1.2 in \cite{Heinonen}.

\begin{theorem}\label{5rcovthm}
Let $(X,d)$ be a metric space and let $\mathcal{B}$ be a family of balls in $X$ such that $\sup \{ diam_d(B): B \in \mathcal{B}\}< \infty$. Then there exists a finite or countable subfamily $\{B_i\}_{i \in \mathcal{I}}$ of $\mathcal{B}$ of pairwise disjoint balls such that
\begin{equation*}
\bigcup_{B \in \mathcal{B}}B \subset \bigcup_{i \in \mathcal{I}} 5B_i,
\end{equation*}
where $5B_i$ denotes the ball $B_d(x_i,5r_i)$ if $B_i=B_d(x_i,r_i)$.
\end{theorem}

\begin{rem} (\textbf{Axioms 1-4 in the classical Euclidean setting for Kakeya sets})\label{axEucl}
As was already mentioned, the classical Kakeya sets correspond to the case when $X=\mathcal{A}$ is $\mathbb{R}^n$, $d=d'$ is the Euclidean metric, $\mu= \mathcal{L}^n$ (Lebesgue measure), $Z=Y=S^{n-1}$ is the unit sphere in $\mathbb{R}^n$, $d_Z$ is the Euclidean metric restricted to $S^{n-1}$, $\nu=\sigma^{n-1}$ is the surface measure on $S^{n-1}$ and $F_e(a)=I_e(a)$ is the segment with midpoint $a$, direction $e \in S^{n-1}$ and length $1$ ($\mu_{e,a}$ is the $1$-dimensional Hausdorff measure on $I_e(a)$). Thus $Q=n$ and $S=n-1$. 

Let us briefly see that in this case the Axioms 1-4 are satisfied and try to understand their geometric meaning. 

Axiom 1 tells us that the volume of a tube is a fixed power of its radius. In this case the tubes are cylinders of radius $\delta$ and height $1$ so we have $\mathcal{L}^n(T^\delta_e(a)) \approx \mathcal{L}^n(\tilde{T}^{2\delta}_e(a)) \approx \delta^{n-1}$. Indeed, we need roughly $1/\delta$ essentially disjoint balls of radius $\delta$ to cover $T^\delta_e(a)$. Moreover, if $A \subset T^\delta_e(a)$ then $\mathcal{L}^n(A) \lesssim \text{diam}_E(A) \delta^{n-1}$, hence Axiom 1 holds with $T=n-1$. 

Axiom 2  holds here with $\theta=0$. It says that if the measure of the intersection of a segment with a ball centred on it is $M$ then the density of the measure of the corresponding tube (with radius essentially the same as the radius of the ball) is at least $M$. Indeed, if $I_e(a)$ is a segment and $x \in I_e(a)$, $\delta \le r \le 2\delta$ then
\begin{equation*}
M= \mathcal{H}^1_E(I_e(a) \cap B_E(x,r)) \approx r \approx \delta.
\end{equation*}
Hence
\begin{equation*}
\mathcal{L}^n(T^\delta_e(a) \cap B_E(x,r)) \gtrsim \delta^n \approx \delta \delta^{n-1} \approx M \mathcal{L}^n(T^\delta_e(a)).
\end{equation*} 

Axiom 3 tells us that the diameter of the intersection of two tubes is at most $\delta$ if the directions of the tubes are sufficiently separated and it can be essentially $1$ if the angle between their directions is $\le \delta$. Here it follows from simple geometric observations. Let $e, f \in S^{n-1}$. Then $|e-f|$ is essentially the angle between any two segments with directions $e$ and $f$. Let $a,a' \in \mathbb{R}^n$ be such that $T^\delta_{e}(a) \cap T^\delta_{f}(a') \neq \emptyset $. Looking at the example in Figure \ref{figtubes} on the left, we see that the diameter of the intersection is essentially $L$. In the thickened right triangle the angle $A$ is essentially $|e-f|$ hence we have $L=2 \delta/\sin A \approx \delta/|e-f|$.  Hence we have
\begin{equation*}
\text{diam}_E(T^\delta_e(a) \cap T^\delta_f(a')) \le b \frac{\delta}{|e-f|}
\end{equation*}
for some constant $b$ depending only on $n$.

\begin{figure}[H]
\includegraphics[scale=0.20]{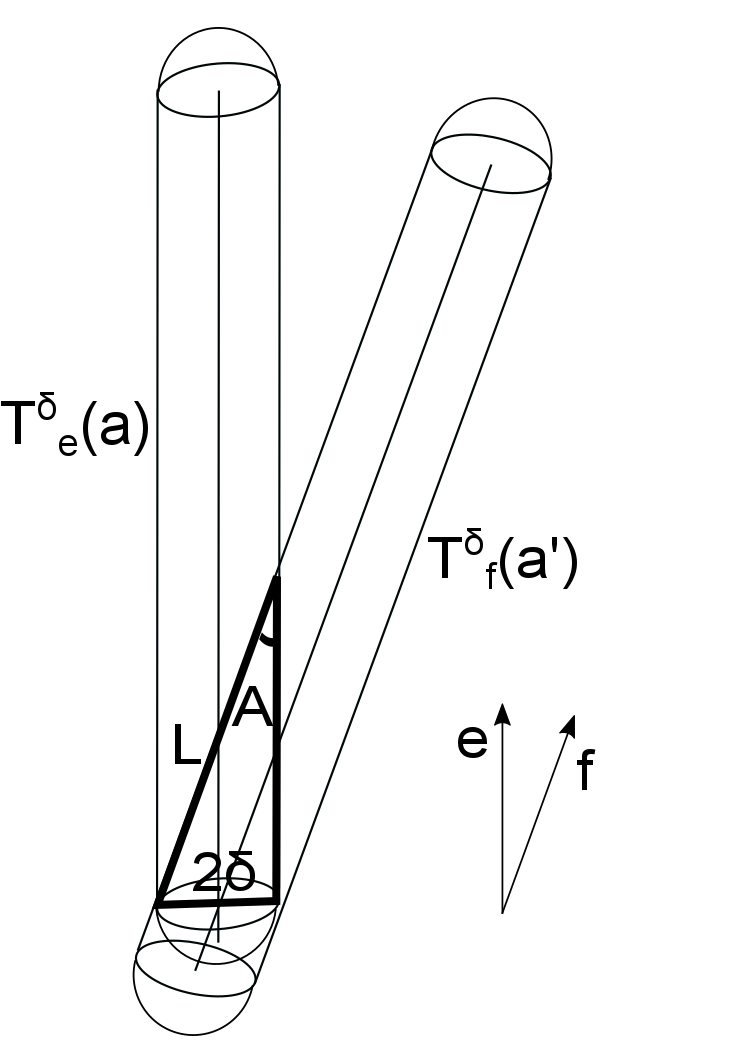}
\hfill
\includegraphics[scale=0.20]{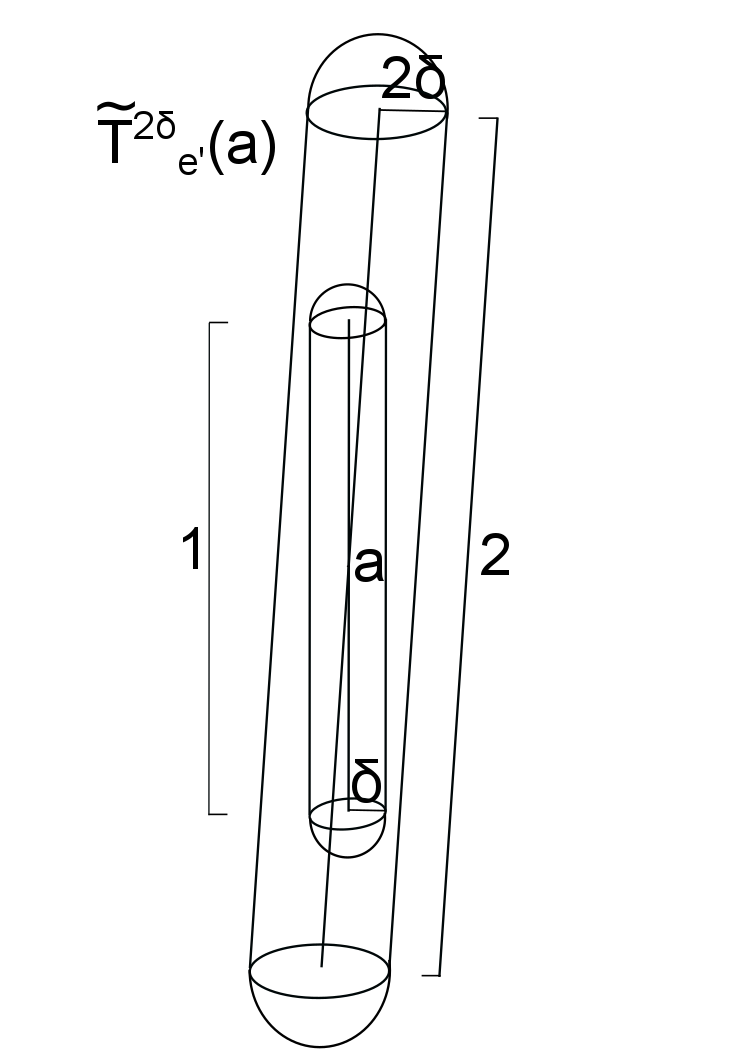}
\caption{Axioms 3 and 4 in the classical Euclidean case (in $\mathbb{R}^3$)}
\label{figtubes}
\end{figure}
For Axiom 4 the intuition is that given two directions $e,e' \in S^{n-1}$ such that $|e-e'| \le \delta$ and given any tube $T^\delta_e(a) $ it can be covered by a fixed number of bigger tubes (with radius $W\delta$) in direction $e'$. We can verify that $T^\delta_e(a) \subset \tilde{T}^{2 \delta}_{e'}(a)$, where $\tilde{T}^{2 \delta}_{e'}(a)$ is the $2 \delta$ neighbourhood of $\tilde{I}_{e'}(a)$, which is the segment with midpoint $a$, direction $e'$ and length $2$ (so $W=2$). Indeed, if $p \in T^\delta_e(a)$ then there exists $q=te+a \in I_e(a)$, $-1/2 \le t \le 1/2$, such that $|p-te-a| \le \delta$. Then we have
\begin{equation*}
|p-te'-a| \le |p-te-a|+|te-te'| \le2 \delta,
\end{equation*}
which means that $p$ is contained in the $2\delta$  neighbourhood of $\tilde{I}_{e'}(a)$, that is $p \in  \tilde{T}^{2 \delta}_{e'}(a)$ (see the right picture in Figure \ref{figtubes}).

\end{rem}

\begin{rem}\label{fduless}
Observe that Axioms 1 and 4 imply that if $u \in B_{d_Z}(v, \delta)$ then $f^d_\delta (u) \lesssim \tilde{f}^d_{W\delta} (v)$. Indeed, for every $a \in \mathcal{A}$ we have $T^\delta_u(a) \subset \cup_{k=1}^N \tilde{T}^{W \delta}_v(b_k)$ with $N \le \bar{N}$. Thus
\begin{align*}
&\frac{1}{\mu(T^\delta_u(a))} \int_{T^\delta_u(a)} |f| d \mu \le \frac{1}{c_1 \delta^T} \sum_{k=1}^N \int_{\tilde{T}^{W \delta}_v(b_k)}|f| d \mu \\
& \le \frac{c_2}{c_1} N \sup_{k=1, \dots, N} \frac{1}{\mu(\tilde{T}^{W  \delta}_v(b_k))} \int_{\tilde{T}^{W \delta}_v(b_k)}|f| d \mu \\
& \le  \frac{c_2}{c_1} \bar{N} \tilde{f}^d_{W \delta}(v),
\end{align*}
which implies $f^d_\delta (u) \lesssim \tilde{f}^d_{W\delta} (v)$.
\end{rem}

\begin{rem}
(\textbf{Measurability of $f^d_\delta$}) The Kakeya maximal function is Borel measurable if the set $\{ f^d_\delta > \alpha \}$ is open for every positive real number $\alpha$. This follows from the fact that $u \mapsto \mu(T^\delta_u(a))$ is continuous (in fact this is assumed only to ensure measurability). Indeed, this implies that if $f^d_\delta(u)> \alpha$ then there exists $a \in \mathcal{A}$ such that 
\begin{equation*}
\frac{1}{\mu(T^\delta_u(a))} \int_{T^\delta_u(a)} |f| d \mu > \alpha.
\end{equation*}
Then we also have $ \frac{1}{\mu(T^\delta_v(a))} \int_{T^\delta_v(a)} |f| d \mu > \alpha$ for $v$ sufficiently close to $u$, which means that $\{ f^d_\delta > \alpha \}$ is open. Thus $f^d_\delta$ is Borel measurable.
\end{rem}

\begin{rem}
In the applications we will consider only objects $I_u(a)$ of dimension $\le 1$ since the validity of Axiom 3 is essential in what we will prove later and it would not be meaningful for example for $2$-dimensional pieces of planes. Indeed, let $G(n,2)$ be the Grassmannian manifold of all $2$-dimensional linear subspaces of $\mathbb{R}^n$. For $P \in G(n,2)$ and $\delta >0$ define $P^\delta$ as a rectangle of dimensions $1 \times 1 \times \delta \times \dots \times \delta$ such that its faces with dimensions $1 \times 1$ are parallel to $P$ (that is $P^\delta$ is the $\delta$ neighbourhood in the Euclidean metric of a square of side length $1$ contained in $P$, so it would correspond to a tube). Given two such rectangles $P^\delta_1$ and $P^\delta_2$, there cannot be a diameter estimate like \eqref{ax3} since $\text{diam} (P^\delta_1 \cap P^\delta_2)$ can be $1$ even if the angle between $P_1$ and $P_2$ is $\pi/2$.
\end{rem}

\begin{rem}(Relation between $T$ and $Q$) 
A priori there is no relation between $T$ and $Q$, but in all applications that we will consider we can express any tube $T^\delta_u(a)$ as a union of essentially disjoint balls $B_d(p_i,\delta)$, $i=1, \dots, M$, and this implies a relation between $T$ and $Q$. The number $M$ will also be some power of $\delta$: as was seen in Remark \ref{axEucl}, $M \approx \delta^{-1}$ in the Euclidean case; in Sections \ref{specialFurst}, \ref{Rnd} and \ref{Carnots2}, it will be a different power of $\delta$. Since
\begin{equation*}
\delta^T \approx \mu(T^\delta_u(a)) \approx \mu(\cup_{i=1}^M B_d(p_i,\delta)) \approx M \delta^Q,
\end{equation*}
we have $T=Q-t$ if $M \approx \delta^{-t}$ for some $t$.
\end{rem}

\begin{rem}\label{axiom2union} (Axiom $2$ with union of balls)
Axiom $2$  implies that for every $a \in \mathcal{A}$, $u \in Y$, $x_j \in F_u(a)$, $j  \in \mathcal{I}$ (a finite set of indices), if $\delta \le r_j \le 2\delta$ for every $j \in \mathcal{I}$ and 
\begin{equation}\label{ax2un}
\mu_{u,a}(F_u(a) \cap \bigcup_{j \in \mathcal{I}} B_d(x_j,r_j)) = M
\end{equation}
for some $M>0$, then
\begin{equation}\label{ax2union}
\mu (T^\delta_u(a) \cap \bigcup_{j \in \mathcal{I}} B_d(x_j, Kr_j)) \gtrsim K' M \delta^\theta \mu (T^\delta_u(a)).
\end{equation}
Indeed, by the $5r$-covering theorem \ref{5rcovthm} applied to the family of balls $B_d(x_j, K r_j)$, $j \in \mathcal{I}$, there exists $\mathcal{I}' \subset \mathcal{I}$ such that 
\begin{equation}\label{5rcov}
\bigcup_{j \in \mathcal{I}} B_d(x_j, K r_j) \subset \bigcup_{i \in \mathcal{I}'} B_d(x_i,5 K r_i)
\end{equation} 
and $B_d(x_i, K r_i)$, $i \in \mathcal{I}'$, are disjoint. Using the doubling condition for $\mu_{u,a}$ and the fact that $\bigcup_{j \in \mathcal{I}} B_d(x_j,  r_j) \subset  \bigcup_{i \in \mathcal{I}'} B_d(x_i,5 K r_i) $ by \eqref{5rcov}, we have
\begin{eqnarray*}
\begin{split}
&\sum_{i \in \mathcal{I}'} \mu_{u,a} (F_u(a) \cap B_d(x_i, r_i)) \gtrsim \sum_{i \in \mathcal{I}'} \mu_{u,a} (F_u(a) \cap B_d(x_i, 5 K r_i))  \\
& \ge  \mu_{u,a}(F_u(a) \cap \bigcup_{i \in \mathcal{I}'} B_d(x_i, 5 K r_i)) \ge  \mu_{u,a} (F_u(a) \cap \bigcup_{j \in \mathcal{I}} B_d(x_j,  r_j)) = M.
\end{split}
\end{eqnarray*}\\
Letting $\mu_{u,a} (F_u(a) \cap B_d(x_i, r_i))= a_i$ for $i \in \mathcal{I}'$, we thus have $\sum_{i \in \mathcal{I}'} a_i \gtrsim M$. By \eqref{ax2} we have
\begin{equation*}
\mu (T^\delta_u(a) \cap  B_d(x_i, Kr_i)) \ge K' a_i \delta^\theta \mu (T^\delta_u(a)),
\end{equation*}
thus (since the balls $B_d(x_i, K r_i)$, $i \in \mathcal{I}'$, are disjoint)
\begin{eqnarray}\label{ax2unend}
\begin{split}
&\mu (T^\delta_u(a) \cap \bigcup_{j \in \mathcal{I}} B_d(x_j, K r_j)) \ge \mu (T^\delta_u(a) \cap \bigcup_{i \in \mathcal{I}'} B_d(x_i, K r_i)) \\ &= \sum_{i \in \mathcal{I}'} \mu (T^\delta_u(a) \cap  B_d(x_i, Kr_i)) \ge \sum_{i \in \mathcal{I}'}  K' a_i \delta^\theta \mu (T^\delta_u(a)) \gtrsim K' M \delta^\theta  \mu (T^\delta_u(a)).
\end{split}
\end{eqnarray}
Hence \eqref{ax2union} holds.
\end{rem}

\begin{rem}\label{muut}
Suppose that $\mu$ is Ahlfors $Q$-regular, that is \eqref{uppA} holds and there exists another constant $c_0 >0$ such that
\begin{equation*}
\mu(B_d(a,r)) \ge c_0 r^Q
\end{equation*}
for every $a \in X$ and every $r < \text{diam}_d(X)$.
If there exists $0<t\le Q-T$ such that 
\begin{equation}\label{mulessrt}
\mu_{u,a}(F_u(a) \cap B_d(x,r)) \le C r^t
\end{equation} 
for every $x \in F_u(a)$ and $r>0$, where $C$ is a constant not depending on $u$ and $a$, then Axiom $2$ holds with $\theta=Q-t-T$ and $K=1$. In fact, we show that it holds for balls with radius $\delta \le r \le 10 \delta$ (we will need this in the following Remark \ref{rem26}). If for some $x \in F_u(a)$, $\delta \le r \le 10 \delta$ and $M>0$ we have
\begin{equation*}
\mu_{u,a}(F_u(a) \cap B_d(x,r))=M
\end{equation*}
then $M \le C r^t \approx \delta^t$. Since $x \in F_u(a) \subset I_u(a)$, we have $B_d(x,\delta) \subset T^\delta_u(a)$ and $B_d(x, \delta) \subset B_d(x,r)$. Thus
\begin{eqnarray}\label{muqtT}
\begin{split}
\mu(T^\delta_u(a) \cap B_d(x,r))& \ge \mu(T^\delta_u(a) \cap B_d(x,\delta))\approx \delta^Q\\
&=\delta^{Q-t-T} \delta^t \delta^T \gtrsim M \delta^{Q-t-T} \mu(T^\delta_u(a)),
\end{split}
\end{eqnarray}
which implies that Axiom 2 holds with $\theta=Q-t-T$.
\end{rem}

\begin{rem}\label{rem26}
(Axiom 2 with union of balls without doubling condition for $\mu_{u,a}$) We can prove that Axiom 2 implies \eqref{ax2union} as in Remark \ref{axiom2union} even if $\mu_{u,a}$ does not satisfy the doubling condition but it satisfies instead condition \eqref{mulessrt} and $\mu$ is Ahlfors $Q$-regular as in the previous Remark \ref{muut}.

Assume that $B_d(x_j,r_j)$, $j \in \mathcal{I}$, is a family of balls such that $x_j \in F_u(a)$, $\delta \le r_j \le 2\delta$ for every $j \in \mathcal{I}$ and 
\begin{equation}\label{a}
\mu_{u,a}(F_u(a) \cap \bigcup_{j \in \mathcal{I}} B_d(x_j,r_j)) = M
\end{equation}
for some $M >0$. Then we want to show that
\begin{equation}\label{c}
\mu(T^\delta_u(a) \cap \bigcup_{j \in \mathcal{I}} B_d(x_j,r_j)) \gtrsim M \delta^{Q-t-T} \mu(T^\delta_u(a)).
\end{equation}
By the $5r$-covering theorem \ref{5rcovthm}, there exists $\mathcal{I}' \subset \mathcal{I}$ such that 
\begin{equation}\label{b}
\bigcup_{j \in \mathcal{I}} B_d(x_j,r_j) \subset \bigcup_{i \in \mathcal{I}'} B_d(x_i,5r_i)
\end{equation}
and the balls $B_d(x_i,r_i)$, $i \in \mathcal{I}'$, are disjoint. Then by \eqref{a} and \eqref{b}
\begin{align*}
M &= \mu_{u,a} (F_u(a) \cap \bigcup_{j \in \mathcal{I}} B_d(x_j,r_j)) \le  \mu_{u,a} (F_u(a) \cap \bigcup_{i \in \mathcal{I}'} B_d(x_i,5 r_i))\\
& \le \sum_{i \in \mathcal{I}'}  \mu_{u,a} (F_u(a) \cap B_d(x_i,5 r_i)).
\end{align*}
Let $a_i =  \mu_{u,a} (F_u(a) \cap B_d(x_i, 5 r_i))$. Then $M \le \sum_{i \in \mathcal{I}'} a_i$ and by \eqref{muqtT} we have $\mu(T^\delta_u(a) \cap B_d(x_i,5r_i)) \gtrsim a_i \delta^{Q-t-T} \mu(T^\delta_u(a))$ since $\delta \le 5r_i \le 10 \delta$. Thus
\begin{align*}
&\mu(T^\delta_u(a) \cap \bigcup_{j \in \mathcal{I}} B_d(x_j,r_j)) \ge \mu ( T^\delta_u(a) \cap \bigcup_{i \in \mathcal{I}'} B_d(x_i,r_i)) \\ & = \sum_{i \in \mathcal{I}'} \mu(T^\delta_u(a) \cap B_d(x_i,r_i)) \approx  \sum_{i \in \mathcal{I}'} \delta^Q \\ & \gtrsim \sum_{i \in \mathcal{I}'} \mu(T^\delta_u(a) \cap B_d(x_i, 5r_i)) \gtrsim \sum_{i \in \mathcal{I}'} a_i \delta^{Q-t-T} \mu(T^\delta_u(a)) \\ & \gtrsim M  \delta^{Q-t-T} \mu(T^\delta_u(a)),
\end{align*}
which proves \eqref{c}.
\end{rem}

\begin{rem}\label{WA}(Wolff's axioms)
In \cite{MR1363209} Wolff used an axiomatic approach to obtain estimates for both the Kakeya and Nikodym maximal functions at the same time. The axioms are different, even if there are some small similarities with the setting considered here. In Wolff's axioms the ambient space is $\mathbb{R}^n$ with the Euclidean metric and the Lebesgue measure. The space of directions is a metric space $(M,d_M)$ endowed with an Ahlfors $m$-regular measure for some $m>0$. To each $\alpha \in M$ is associated a set $F_\alpha$ of lines in $\mathbb{R}^n$ such that the closure of $\cup_\alpha F_\alpha$ is compact and
\begin{equation*}
d_M(\alpha, \beta) \lesssim \inf_{l \in F_\alpha, m \in F_\beta} \mbox{dist}(l,m).
\end{equation*}
Here $\mbox{dist}(l,m) \approx  \measuredangle(l,m)+d_{min}(l,m)$, where $ \measuredangle(l,m)$ is the angle between the directions of $l$ and $m$ and $d_{min}(l,m)=\inf\{|p-q|: p \in l \cap 100D, q \in m \cap 100D\}$, $D$ is a disk intersected by $l$ and $m$ and $100D$ is the disk with the same center as $D$ and radius $100$ times the radius of $D$.

For $f\in L^1_{loc}(\mathbb{R}^n)$ and $0<\delta <1$ the maximal function is defined as
\begin{equation*}
M_\delta f(\alpha)= \sup_{l \in F_\alpha}  \sup_{a \in l} \frac{1}{\mathcal{L}^n(T^\delta_l(a))} \int_{T^\delta_l(a)} |f| d \mathcal{L}^n,
\end{equation*}
where $T^\delta_l(a)$ is the tube with length $1$, radius $\delta$, axis $l$ and center $a$. 
The Kakeya case corresponds to $M=S^{n-1}$ endowed with the Euclidean metric and the spherical measure. For every $e \in S^{n-1}$, $F_e$ is the set of lines with direction $e$.
In Section \ref{Nikodym} we will prove the lower bound $\frac{n+2}{2}$ for the Hausdorff dimension of Nikodym sets, which was originally proved by Wolff in his axiomatic setting. It corresponds to the case when $M$ is the $x_1, \dots, x_{n-1}$-hyperplane and for $\alpha \in M$, $F_\alpha$ is the set of lines passing through $\alpha$.

The other assumption in Wolff's paper (called Property $(*)$) roughly states that there is no $2$-dimensional plane $\Pi$ such that every line contained in $\Pi$ belongs to a different $F_\alpha$.

In Section \ref{NikodLip} we will consider sets containing a segment through almost every point of an $(n-1)$-rectifiable set, which reduces to the case of sets containing a segment through every point of an $(n-1)$-dimensional Lipschitz graph. This case could also be treated using Wolff's original axioms.
\end{rem}

\section{Bounds derived from $L^p$ estimates of the Kakeya maximal function}\label{SectionLp}

As in the Euclidean case, one can show that certain $L^p$ estimates of the Kakeya maximal function yield lower bounds for the Hausdorff dimension of Kakeya sets. We first prove that the bounds follow from a restricted weak type inequality, which we will use when dealing with Bourgain's method.

\begin{theorem}\label{bound}
If for some $1 \le p  < \infty$, $\beta >0$ such that $Q-(\beta+\theta) p >0$ there exists $C=C_{p,\beta}>0$ such that
\begin{equation}\label{weak}
\nu(\{u \in Y: (\chi_E)^d_\delta(u) > \lambda \}) \le C \lambda^{-p} \delta^{- \beta p} \mu(E)
\end{equation}
for every $\mu$ measurable set $E \subset X$and for any $\lambda >0$, $0< \delta <1$, then the Hausdorff dimension of any Kakeya set in $X$ with respect to the metric $d$ is at least $Q -( \beta + \theta) p$.
\end{theorem}

Recall that $\theta$ is the constant appearing in Axiom 2.
The proof is essentially the same as for the Euclidean case, see \cite{Mattila} (Theorems 22.9 and 23.1), where one gets the lower bound $n-\beta p$ for the Hausdorff dimension of Kakeya sets.

\begin{proof}
Given a Kakeya set $B$, consider a covering  $B \subset \cup_j B_d(x_j,r_j)$, $r_j < 1$. We divide the balls into subfamilies of essentially the same radius, by letting for $k=1,2, \dots$
\begin{equation*}
J_k=\{ j: 2^{-k} \le r_j < 2^{1-k} \}.
\end{equation*}
Since $B$ is a Kakeya set, for any $u \in Y$ there exists $a_u \in \mathcal{A}$ such that $F_u(a_u) \subset B$. For $k=1,2,\dots$, let 
\begin{equation*}
Y_k=\{ u \in Y: \mu_{u,a} (F_u(a_u) \cap \bigcup_{j \in J_k} B_d(x_j,r_j)) \ge \frac{1}{2k^2} \}.
\end{equation*}
Then $\cup_k Y_k=Y$. Indeed, if there exists $u \in Y$ such that $u \notin Y_k$ for any $k$, then 
\begin{align*}
1=\mu_{u,a}(F_u(a_u)) \le \sum_k \mu_{u,a} (F_u(a_u) \cap \bigcup_{j \in J_k} B_d(x_j,r_j)) < \sum_k \frac{1}{2k^2} <1,
\end{align*}
which yields a contradiction.

For $u \in Y_k$, if $F_u(a_u) \cap B_d(x_j,r_j)=\emptyset$ then we can discard $B_d(x_j,r_j)$. Otherwise, there exists $y_j \in F_u(a_u) \cap B_d(x_j,r_j)$, thus $B_d(x_j,r_j) \subset B_d(y_j, 2 r_j)$. Since $2^{1-k}\le 2 r_j < 2^{2-k}$ and $\mu_{u,a} (F_u(a_u) \cap \bigcup_{j \in J_k} B_d(y_j, 2r_j)) \ge \frac{1}{2k^2}$, we have by Axiom 2 and Remark \ref{axiom2union}
\begin{equation*}
\mu(T^{2^{1-k}}_u(a_u) \cap F_k) \ge \frac{K'}{2k^2} 2^{(1-k)\theta} \mu(T^{2^{1-k}}_u(a_u)),
\end{equation*}
where $F_k= \cup_{j \in J_k} B_d(y_j,2Kr_j)$.
Letting $f=\chi_{F_k}$, it follows that $f^d_{2^{1-k}}(u) \gtrsim  2^{(1-k)\theta}/k^2$ for every $u \in Y_k$. Using then the assumption and $\mu(F_k) \lesssim \# J_k 2^{(2-k)Q}$, one gets
\begin{equation*}
\nu(Y_k)\lesssim k^{2p } 2^{k \theta p} 2^{k\beta p} \mu(F_k) \lesssim k^{2p} 2^{-k(Q-\beta p-\theta p)} \# J_k.
\end{equation*} 
Hence if $0 < \alpha < Q-(\beta+\theta) p$
\begin{equation*}
\sum_j r_j^\alpha \ge \sum_k \# J_k 2^{-k \alpha} \gtrsim \sum_k \nu(Y_k) \ge  \nu(Y).
\end{equation*}
This implies that $\mathcal{H}^\alpha(B)>0$ for every $0 < \alpha < Q-(\beta+\theta) p$, thus $\dim_d B \ge Q-( \beta+\theta) p$.
\end{proof}

As a corollary, we get the following.

\begin{cor}\label{Lpbound}
If for some $1 \le p  < \infty$, $\beta >0$ such that $Q-(\beta+\theta) p>0$, there exists $C=C_{p,\beta}>0$ such that
\begin{equation}\label{Lpboundeq}
||f^d_\delta||_{L^p(Y, \nu)} \le C \delta^{-\beta} ||f||_{L^p(X,\mu)}
\end{equation}
for every $f \in L^p( X,\mu)$, $0< \delta <1$, then the Hausdorff dimension of any Kakeya set in $X$ with respect to the metric $d$ is at least $Q -( \beta+\theta) p$.
\end{cor}

Indeed \eqref{weak} follows from \eqref{Lpboundeq} by Chebyshev's inequality.

We can discretize the above inequalities as follows.
Given $0 < \delta <1$, we say that $\{u_1, \dots, u_m\}  \subset Y$ is a \textit{$\delta$-separated} subset of $Y$
if $d_Z(u_i,u_j) > \delta$ for every $i \ne j$. Observe that this implies that $m \le C_{Y,S} \delta^{-S}$, where $C_{Y,S}$ is a constant depending on $S$ and $\text{diam}_{d_Z}(Y)$. Indeed, since the balls $B_{d_Z}(u_j, \delta/2)$, $j=1, \dots,m$, are disjoint and $Y \subset B_{d_Z}(v, \text{diam}_{d_Z}(Y))$ for any $v \in Y$, we have $\cup_{j=1}^m B_{d_Z}(u_j, \delta/2) \subset B_{d_Z}(v, \text{diam}_{d_Z}(Y)+1) $ thus
\begin{align*}
&m \tilde{c}_0 2^{-S}\delta^S \le \sum_{j=1}^m \nu\left(B_{d_Z}\left(u_j,\frac{\delta}{2}\right)\right) = \nu \left(\bigcup_{j=1}^m  B_{d_Z}\left(u_j,\frac{\delta}{2}\right)\right) \\& \le \nu(B_{d_Z}(v, \text{diam}_{d_Z}(Y)+1) ) \le \tilde{C}_0 (\text{diam}_{d_Z}(Y)+1)^S.
\end{align*}

We say that $\{u_1, \dots, u_m\}  \subset Y$ is a \textit{maximal $\delta$-separated} subset of $Y$ if it is $\delta$-separated and for every $u \in Y$ there exists $j$ such that $u \in B_{d_Z}(u_j,  \delta)$. We can find this subset for example by taking any $u_1 \in Y$, $u_2 \in Y \setminus B_{d_Z}(u_1, \delta)$, $u_3 \in  Y \setminus (B_{d_Z}(u_1, \delta) \cup B_{d_Z}(u_2, \delta))$ and so on. The process is finite since $Y$ is compact.
Observe that $Y \subset \cup_{j=1}^m B_{d_Z}(u_j,\delta)$, thus $m \gtrsim \nu(Y) \delta^{-S}$, hence $m \approx \delta^{-S}$.

The following lemma can be proved as in the Euclidean case (see \cite{Mattila}, Proposition 22.4).

\begin{lem}\label{Lemmadiscr}
Let $1 < p < \infty $, $q=\frac{p}{p-1}$, $0 < \delta <1$ and $0< M < \infty$. If for all tubes $T_1, \dots, T_m$, $T_j=\tilde{T}^{W\delta}_{u_j}(a_j)$, where $\{u_1,\dots,u_m\}$ is a maximal $\delta$-separated subset of $Y$ and $a_j \in \mathcal{A}$, and all positive numbers $t_1, \dots, t_m$ such that 
\begin{equation*}
\delta^{S} \sum_{j=1}^m t_j^q \le 1,
\end{equation*}
we have 
\begin{equation*}
\left\Vert \sum_{j=1}^m t_j \chi_{T_j} \right\Vert_{L^q(X, \mu)} \le M \delta^{T-S},
\end{equation*}
then for every $f \in L^p(X, \mu)$
\begin{equation*}
||f^d_\delta||_{L^p(Y, \nu)} \lesssim M ||f||_{L^p(X, \mu)}.
\end{equation*}

\end{lem}

Recall that $S$ is the power of the radius appearing in \eqref{BS} in the description of the measure $\nu$, whereas $T$ is the constant appearing in Axiom 1. For completeness, we show the proof. 

\begin{proof}
Let $\{u_1, \dots, u_m\}$ be a maximal $\delta$-separated subset of $Y$. By Remark \ref{fduless}, we have
\begin{eqnarray}
\begin{split}\label{P1}
||& f^d_\delta||_{L^p(Y, \nu)}^p \le \sum_{j=1}^m \int_{B_{d_Z}(u_j,\delta)} (f^d_\delta (u))^p d \nu u  \\
& \lesssim \sum_{j=1}^m \nu(B_{d_Z}(u_j,\delta)) (\tilde{f}^d_{W\delta}(u_j))^p \lesssim \sum_{j=1}^m \delta^{S}  (\tilde{f}^d_{W\delta}(u_j))^p.
\end{split}
\end{eqnarray}
Using then the duality of $l^p$ and $l^q$, one can write
\begin{equation*}
|| f^d_\delta||_{L^p(Y, \nu)} \lesssim \delta^{S} \sum_{j=1}^m t_j \tilde{f}^d_{W\delta}(u_j),
\end{equation*}
where $\delta^{S} \sum_{j=1}^m t_j^q=1$. Thus for some $a_j \in \mathcal{A}$,
\begin{eqnarray*}
\begin{split}
||& f^d_\delta||_{L^p(Y, \nu)} \lesssim \delta^{S} \sum_{j=1}^m t_j \frac{1}{\mu(\tilde{T}^{W\delta}_{u_j}(a_j))} \int_{\tilde{T}^{W\delta}_{u_j}(a_j)} |f| d \mu \\
& \lesssim \delta^{S-T} ||\sum t_j \chi_{\tilde{T}^{W\delta}_{u_j}(a_j)}||_{L^q(X, \mu)} ||f||_{L^p(X, \mu)} \le M ||f||_{L^p(X,\mu)}.
\end{split}
\end{eqnarray*}
by Axiom 1, H\"older's inequality and the assumption.
\end{proof}

As a corollary, we get the following (see \cite{Mattila}, Proposition 22.6).

\begin{prop}\label{PropDiscr}
Let $1 < p < \infty $, $q=\frac{p}{p-1}$, $0 < \delta <1$ and $1\le M < \infty$. If for every $\epsilon >0$ there exists $C_\epsilon $ such that
\begin{equation}\label{LpDiscr}
\left\Vert \sum_{j=1}^m \chi_{T_j}\right\Vert_{L^q(X, \mu)} \le C_\epsilon M \delta^{- \epsilon} (m \delta^{S})^{1/q} \delta^{T-S},
\end{equation}
for all tubes $T_1, \dots , T_m$, where $T_j=\tilde{T}^{W\delta}_{u_j}(a_j)$, $\{u_1,\dots,u_m\}$ is a $\delta$-separated subset of $Y$ and $a_j \in \mathcal{A}$, then
for every $f \in L^p(X, \mu)$
\begin{equation*}
||f^d_\delta||_{L^p(Y, \nu)} \le C_\epsilon M \delta^{-\epsilon}||f||_{L^p(X, \mu)}.
\end{equation*}
\end{prop}

\begin{proof}
Let $T_1, \dots , T_m$ be tubes $T_j=\tilde{T}^{W\delta}_{u^j}(a_j)$, where $\{u_1,\dots,u_m\}$ is a maximal $\delta$-separated subset of $Y$. Let $t_1, \dots, t_m$ be positive numbers such that $\delta^{S} \sum_{j=1}^m t_j^q \le 1$. By Lemma \ref{Lemmadiscr} it is enough to show that 
\begin{equation}\label{suff}
\left\Vert\sum_{j=1}^m t_j \chi_{T_j}\right\Vert_{L^q(X,\mu)} \lesssim M \delta^{- \epsilon} \delta^{T-S}.
\end{equation}
Since $t_j\le \delta^{-S/q}$ and $||\sum_{j=1}^m \delta^T \chi_{T_j}||_{L^q(X,\mu)} \lesssim \delta^{T-S}$, it suffices to sum over $j$ such that $\delta^{T} \le t_j \le \delta^{-S/q}$. Split this sum into $N_\delta \approx \log(1/\delta)$ subsums $I_k= \{j: 2^{k-1} \le t_j < 2^k \}$ and let $m_k$ be the cardinality of $I_k$. Using the assumption \eqref{LpDiscr} with $\epsilon/2$, we have
\begin{eqnarray*}
\left\Vert\sum_{\delta^T \le t_j \le \delta^{-S/q}} t_j \chi_{T_j}\right\Vert_{L^q(X,\mu)} \le \sum_{k=1}^{N_\delta} 2^k \left\Vert \sum_{j \in I_k} \chi_{T_j}\right\Vert_{L^q(X,\mu)} \le C_\epsilon \sum_{k=1}^{N_\delta} 2^k M \delta^{- \epsilon/2} (m_k \delta^S)^{1/q} \delta^{T-S}.
\end{eqnarray*}
Since $ m_k 2^{kq} \le \sum_{j=1}^m (2 t_j)^q \le 2^q \delta^{-S}$, we have $(m_k \delta^S)^{1/q} \le 2^{1-k}$, thus
\begin{equation*}
\left\Vert\sum_{\delta^T \le t_j \le \delta^{-S/q}} t_j \chi_{T_j}\right\Vert_{L^q(X,\mu)}  \le 2 C_\epsilon N_\delta M \delta^{- \epsilon/2} \delta^{T-S}\lesssim M \delta^{- \epsilon} \delta^{T-S},
\end{equation*}
that is \eqref{suff} holds.
\end{proof}

\section{Bourgain's method}\label{SectionBourg}

Bourgain developed a method whose main geometric object is the so called "bush", that is a bunch of tubes intersecting at a common point. Using the same ideas we will show the following.

\begin{theorem}\label{Bourgain}
There exists $C>0$ such that
\begin{equation}\label{weakbourg}
\nu(\{u \in Y: (\chi_E)^d_\delta(u) > \lambda \}) \le C \lambda^{-(S+2)/2} \delta^{S/2-T} \mu(E)
\end{equation}
for every $\mu$ measurable set $E \subset X$ and for any $\lambda >0$, $0< \delta <1$. It follows that the Hausdorff dimension  with respect to $d$ of every Kakeya set in $X$ is at least $\frac{2Q-2T+S}{2}- \theta \frac{S+2}{2}$.
\end{theorem}

The statement about Kakeya sets follows from Theorem \ref{bound}, where $\beta p=T- \frac{S}{2}>0$ and $p=\frac{S+2}{2}$. Note that $\frac{2Q-2T+S}{2}- \theta \frac{S+2}{2}>0$ since $\theta < \frac{2Q-2T+S}{S+2}$.

Observe that interpolating (see Theorem 2.13 in \cite{Mattila}) between this weak type inequality and the trivial inequality $||f^d_\delta||_{L^\infty(Y, \nu)} \lesssim \delta^{-T} ||f||_{L^1(X, \mu)}$, we get for $1 < p < \frac{S+2}{2}$, $q=\frac{pS}{2(p-1)}$,
\begin{equation*}
||f^d_\delta||_{L^q(Y, \nu)} \lesssim C \delta^{1-\frac{T+1}{p}} ||f||_{L^p(X, \mu)}
\end{equation*}
for every $f \in L^p(X, \mu)$. 

We will now prove Theorem \ref{Bourgain} (see Theorem 23.2 in \cite{Mattila} for the Euclidean case).

\begin{proof}

Given a $\mu$ measurable set $E \subset X$ and $\lambda>0$, let
\begin{equation*}
E_\lambda= \{ u \in Y : (\chi_E)^d_\delta(u) > \lambda \}.
\end{equation*}
Let $u_1, \dots, u_N$ be a maximal $\delta$-separated subset of $E_\lambda$, that is $E_\lambda \subset \cup_{j=1}^N B_{d_Z}(u_j, \delta)$ and $d_Z(u_i, u_j) > \delta$ for every $i \ne j$.\\
We have
\begin{equation}\label{B1}
\nu(E_\lambda) \le \sum_{j=1}^N \nu (B_{d_Z}(u_j, \delta)) \lesssim  N \delta^{S},
\end{equation}
hence 
\begin{equation}\label{N}
N \gtrsim \nu(E_\lambda) \delta^{-S}.
\end{equation}
By the definition of $E_\lambda$, we can choose tubes $T_j=T^{\delta}_{u_j}(a_j)$ such that
\begin{equation}\label{EcapTj}
\mu(E \cap T_j ) > \lambda \mu(T_j) \approx \lambda \delta^{T}.
\end{equation} 

To find the bush, consider the smallest integer $M$ such that there exists $x_0 \in E$ that belongs to $M$ tubes $T_j$ and all the other points of $E$ belong to at most $M$ tubes. This means that
\begin{equation*}
\sum_{j=1}^N \chi_{T_j\cap E} \le M,
\end{equation*}
hence integrating over $E$ and using \eqref{EcapTj}
\begin{equation}\label{LnE1}
\mu(E) \ge M^{-1} \sum_{j=1}^N \mu(E \cap T_j) \gtrsim N M^{-1} \lambda \delta^{T}.
\end{equation}
Suppose $x_0 \in T_1 \cap \dots \cap T_M$.

We can show that there exists a constant $c \le b$ (where $b$ is the constant appearing in Axiom 3) such that for every $a \in \mathcal{A}$, $u \in Y$, 
\begin{equation}\label{capx0}
\mu (T^\delta_u(a) \cap B_{d'}(x_0, c \lambda)) \le \frac{\lambda}{2} \mu (T^\delta_u(a)).
\end{equation}
Indeed, $\text{diam}_{d'}(T^\delta_u(a) \cap B_{d'}(x_0, c \lambda)) \le 2 c \lambda$, hence by Axiom 1, $\mu (T^\delta_u(a) \cap B_{d'}(x_0, c \lambda)) \le \frac{c_2}{c_1} 2 c \lambda  \mu (T^\delta_u(a))$, which implies that we can choose $c=\min \{ \frac{c_1}{4c_2},b \}$ so that \eqref{capx0} holds.\\
By \eqref{EcapTj} and \eqref{capx0} for every $j=1, \dots, M$,
\begin{equation}\label{EcapTjminus}
\mu(E \cap T_j \setminus B_{d'}(x_0, c \lambda )) > \frac{\lambda}{2} \mu(T_j).
\end{equation}

Consider the family of balls $\{ B_{d_Z}(u_j, \frac{b \delta}{c \lambda}), j=1, \dots, M \}$, where  $\frac{b\delta}{c \lambda} \ge \delta$ when $\lambda \le 1$ (as we may assume). By the $5r$ covering theorem \ref{5rcovthm} there exists $\{ v_1, \dots,v_L\} \subset \{u_1, \dots, u_M\}$ such that $B_{d_Z}(v_i, \frac{b \delta}{c \lambda})$, $i=1, \dots, L$, are disjoint and 
\begin{equation*}
\bigcup_{j=1}^M B_{d_Z} \left(u_j, \frac{b \delta}{c \lambda}\right) \subset \bigcup_{i=1}^L B_{d_Z}\left(v_i,5 \frac{b \delta}{c \lambda}\right).
\end{equation*}
Thus, since the balls $B_{d_Z}(u_j, \frac{\delta}{2})$, $j=1, \dots,M$, are disjoint,
\begin{equation}\label{B2}
 M \delta^{S} \lesssim \nu \left(  \bigcup_{j=1}^M B_{d_Z}\left(u_j, \frac{\delta}{2}\right)\right) \le \nu \left( \bigcup_{i=1}^L B_{d_Z}\left(v_i,5 \frac{b \delta}{c \lambda}\right)\right) \lesssim L \delta^{S} \lambda^{-S},
\end{equation}
which implies $L \gtrsim M \lambda^S$.

Let $T'_k$ be the tubes corresponding to $v_k$, $k=1, \dots, L$, as chosen above. Since $d_Z(v_i,v_j) > \frac{b \delta}{c \lambda}$ for every $i \ne j \in \{1, \dots, L\}$, it follows by \eqref{ax3} (Axiom 3)
\begin{equation*}
\text{diam}_{d'}(T'_i \cap T'_j) \le c \lambda.
\end{equation*}
Thus the sets $E \cap T'_k \setminus B_{d'}(x_0, c \lambda)$, $k=1, \dots, L$, are disjoint, which implies by \eqref{EcapTjminus}
\begin{equation}\label{LnE2}
\mu(E) \ge \sum_{k=1}^L \mu(E \cap T'_k \setminus B_{d'}(x_0, c \lambda)) \gtrsim L \lambda \delta^{T} \gtrsim M \lambda^{S+1} \delta^{T}.
\end{equation}
It follows by \eqref{LnE1}, \eqref{LnE2} and  \eqref{N}
\begin{eqnarray*}
\mu(E) \gtrsim \max \{ N M^{-1} \lambda \delta^{T}, M \lambda^{S+1} \delta^{T} \} 
\ge \sqrt{ N M^{-1} \lambda \delta^{T}  M \lambda^{S+1} \delta^{T}} = \\
= \sqrt{N} \lambda^{(S+2)/2} \delta^{T} \gtrsim \nu(E_\lambda)^{1/2} \delta^{T-S/2}  \lambda^{(S+2)/2}.
\end{eqnarray*}
Hence, since $\nu(E_\lambda) \le \nu(Y) \le 1$, we get
\begin{equation*}
\nu(E_\lambda) \le \nu(E_\lambda)^{1/2} \lesssim \mu(E)  \delta^{S/2-T} \lambda^{-(S+2)/2},
\end{equation*}
which completes the proof.
\end{proof}

\section{Wolff's method}\label{SectionWolff}

In Wolff's argument the main geometric object is the hairbrush, that is a configuration of tubes intersecting a fixed one. More precisely, we call an $(N,\delta)$-\textit{hairbrush} a collection of tubes $T_1, \dots, T_N$ such that $T_j=\tilde{T}^{W\delta}_{u_j}(a_j)$, $d_Z(u_j,u_k) > \delta$ for every $j \ne k$ and there exists a tube $T=\tilde{T}^{W \delta}_u(a)$ such that $T \cap T_j \ne \emptyset$ for every $j \in \{1, \dots, N \}$.

We will use a simplification of Wolff's proof due to Katz. Here we need to assume also the following, that contains the geometric part of the proof.

(\textbf{Axiom 5})
There exist two constants $\alpha, \lambda$ with
\begin{align*}
&0  \le \alpha \le \min\left\{ \frac{Q}{\theta}-2, S-1 \right\} \quad (\mbox{if }  \theta=0 \mbox{ then only }  0 \le \alpha \le S-1),\\
&\max \{S-\alpha, S-2T+2 \} \le \lambda < 2Q-2T+S+2- 2 \theta (\alpha+2)
\end{align*}
and $0<C'< \infty$ such that the following holds.
Let $0<\delta, \beta, \gamma<1$ and let $T_1, \dots , T_N$ be such that $T_j=\tilde{T}^{W\delta}_{u_j}(a_j)$, $d_Z(u_j,u_k)>\delta$. Let $T=\tilde{T}^{W\delta}_u(a)$ be such that $T\cap T_j \ne \emptyset$ and $d_Z(u_j,u) \ge \beta/8$ for every $j=1, \dots, N$. Then for every $j=1, \dots, N$,
\begin{eqnarray}\label{axiom5def}
\begin{split}
&\# \{ i : d_Z(u_i,u_j) \le \beta, T_i \cap T_j \ne \emptyset, d'(T_i \cap T_j, T_j \cap T) \ge \gamma \}  \\
&\le C' \delta^{-\lambda} \beta \gamma^{- \alpha}.
\end{split}
\end{eqnarray}

Recall that $Q$ is the upper Ahlfors regularity constant of $\mu$, $\theta$ is the constant appearing in Axiom 2, $S $ is the constant related to the $\nu$ measure of balls centered in $Y$ (see \eqref{BS}) and $T$ is the constant in Axiom 1.
Observe that if $\theta >1/2$ then $2Q-2T+S+2- 2 \theta (\alpha+2)>S-\alpha $ only when $\alpha <\frac{2Q-2T+2-4 \theta}{2\theta-1} $ (if $0 \le \theta <1/2$ then $2Q-2T+S+2- 2 \theta (\alpha+2)>S-\alpha $ holds).

 In the applications that we will consider $\lambda$ will always be $1$ except in the case of Kakeya sets in $\mathbb{R}^n$ endowed with a metric homogeneous under non-isotropic dilations (see Section \ref{Rnd}, where we do not show that Axiom 5 holds but only that in general we need to have $\lambda >1$).

\begin{rem}\label{ax5Eucl}(\textbf{Axiom 5 in the Euclidean case for Kakeya sets})
Let us now see why Axiom 5 holds in the Euclidean case with $\lambda=1$ and $\alpha=n-2$ (see also Lemma 23.3 in \cite{Mattila}) to have some geometric intuition. Recall that in this case a tube $T^\delta_e(a)$ is the $\delta$ neighbourhood of the segment $I_e(a)$ with direction $e \in S^{n-1}$ and midpoint $a \in \mathbb{R}^n$.

Let $0<\delta, \beta, \gamma<1$ and let $T_1, \dots , T_N$ be such that $T_j=\tilde{T}^{2\delta}_{e_j}(a_j)$, $|e_j-e_k|>\delta$. Let $T=\tilde{T}^{2\delta}_e(a)$ be such that $T\cap T_j \ne \emptyset$ and $|e_j-e| \ge \beta/8$ for every $j=1, \dots, N$. We want to show that for every $j=1, \dots, N$,
 \begin{eqnarray}\label{ax5Euclidean}
\begin{split}
&\# \{ i : |e_i-e_j| \le \beta, T_i \cap T_j \ne \emptyset, d_E(T_i \cap T_j, T_j \cap T) \ge \gamma \}  \\
&\le C' \delta^{-1} \beta \gamma^{2-n},
\end{split}
\end{eqnarray}
where $C'$ is a constant depending only on $n$.

Fix one $j$.
We can assume that $\delta$ is much smaller than $\gamma$ because if $\delta \gtrsim \gamma$, then since the points $e_k$ are $\delta$-separated we have
\begin{equation*}
\# \{ i: |e_i-e_j| \le \beta \} \lesssim \beta^{n-1} \delta^{1-n} \le \beta \delta^{-1} \delta^{2-n} \lesssim \beta \delta^{-1} \gamma^{2-n}.
\end{equation*}
Thus \eqref{ax5Euclidean} would hold.

The tubes $T$ and $T_j$ intersect thus there exist two intersecting segments $l$ and $l_j$ contained in $T$ and $T_j$ respectively. We can assume that $l \cap l_j$ is the origin and that $l$ and $l_j$ span the $x_1, x_2$-plane. The angle between them is $\gtrsim \beta$. Suppose now that $T_i$ intersects both $T$ and $T_j$ in such a way that the angle between $T$ and $T_i$ is $\gtrsim \beta$ and the angle between $T_i$ and $T_j$ is $\lesssim \beta$. It follows from this and the fact that $d_E(T_i \cap T_j, T_j \cap T) \ge \gamma$ that also $d_E(T_i \cap T_j, T_i \cap T) \gtrsim \gamma$ (see Figure \ref{figax5E}). Thus $T_i$ makes an angle $\lesssim \delta/\gamma$ with the $x_1, x_2$-plane. Since $|e_i - e_j| \le \beta$, we have that $e_i$ is contained in the set 
\begin{equation*}
B_j= B_E(e_j,\beta) \cap \{ x=(x_1, \dots, x_n) \in S^{n-1}: |x_k| \lesssim \delta/\gamma, k=3, \dots, n \}.
\end{equation*}
Since $\sigma^{n-1}(B_j) \lesssim \beta (\delta/\gamma)^{n-2}$, it can contain $\lesssim \beta \delta^{-1} \gamma^{2-n}$ points $e_i$ that are $\delta$-separated. Hence \eqref{ax5Euclidean} holds.
\begin{figure}[H]
\begin{center}
\includegraphics[scale=0.25]{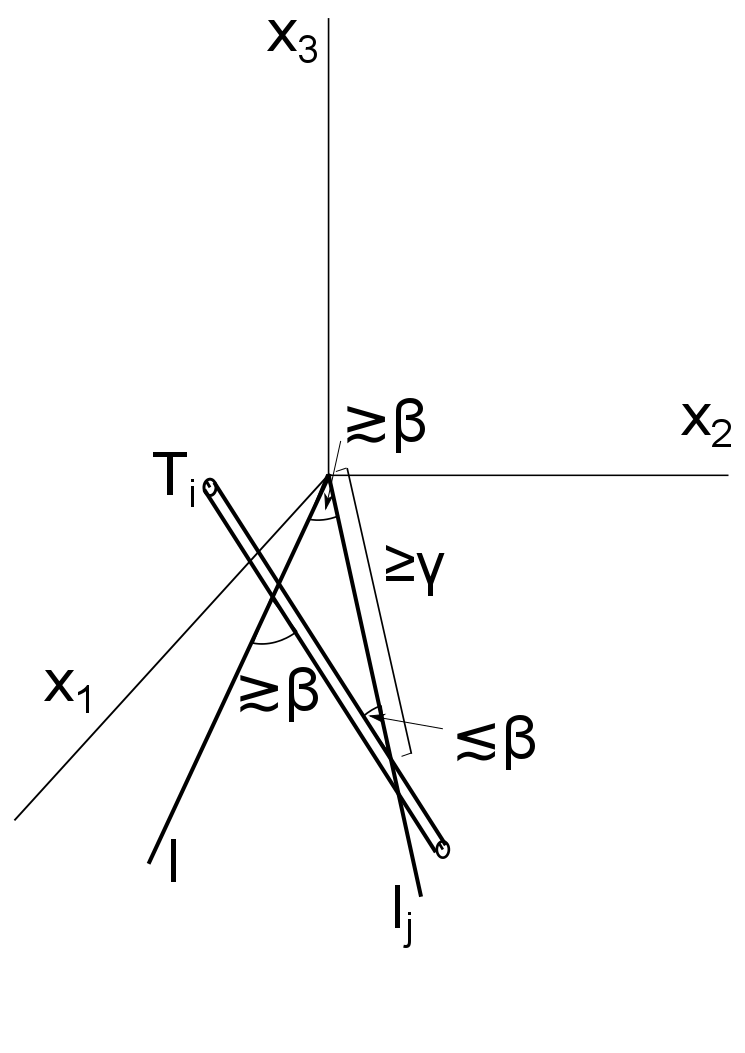}
\caption{Axiom 5 in the classical Euclidean case (in $\mathbb{R}^3$)}
\label{figax5E}\end{center}
\end{figure}

\end{rem}

Following the Euclidean proof (\cite{Mattila}, Lemma 23.4), one gets then this behaviour of the hairbrushes.

\begin{lem}\label{hairbrush}
Let $T_1, \dots, T_N$ be an $(N, \delta)$-hairbrush. Then for every $\epsilon>0$ and $\frac{\alpha+2}{\alpha+1} \le p \le 2$,
\begin{equation}\label{hair1}
\int \left( \sum_{j=1}^N \chi_{T_j} \right)^p d \mu \lesssim C_{p,\epsilon} \delta^{T} N \delta^{\lambda+\alpha+1-p(\lambda+\alpha)-\epsilon}.
\end{equation}
\end{lem}  

\begin{proof}
We may assume that $\text{diam}_{d_Z}(Y) \le 1$.
We partition the set of indices $I=\{1,\dots,N\}$ in several ways. First, for $k=0,1,\dots, $ with $\delta \le 2^{-k} $, we let $I(k)=\{i \in I : 2^{-k-1} <d_Z(u_i,u) \le 2^{-k} \}$. Observe that there is at most one $i$ such that $d_Z(u_i,u) \le \delta/2$, thus we can ignore this and assume that every $i$ belongs to some $I(k)$.

Note also that the second exponent of $\delta$ in \eqref{hair1} is non positive since $\lambda \ge S-\alpha \ge 1$ implies that $p \ge \frac{\alpha+2}{\alpha+1} \ge \frac{\alpha+\lambda+1}{\alpha+\lambda}$, hence $\lambda+\alpha+1-p(\lambda+\alpha) \le 0$.
There are only $\approx \log(1/\delta)$ values of $k$ to consider, thus it is enough to show the estimate summing over $\chi_{T_i}$ with $i \in I(k)$ for a fixed $k$. Since
\begin{equation*}
\int (\sum_{i \in I(k)} \chi_{T_i})^p d \mu= \sum_{j \in I(k)}  \int_{T_j} (\sum_{i \in I(k)} \chi_{T_i})^{p-1} d \mu,
\end{equation*}
this is reduced to show that for each $j \in I(k)$,
\begin{equation*}
\int_{T_j} (\sum_{i \in I(k)} \chi_{T_i})^{p-1} d \mu \lesssim \delta^{T}  \delta^{\lambda+\alpha+1-p(\lambda+\alpha)-\epsilon}.
\end{equation*}
Then for fixed $k$ and $j \in I(k)$, we define for positive integers $l,m$ such that $l \ge k-2$ (otherwise the set is empty) and $\delta/2 \le 2^{-m},2^{-l}<1$,
\begin{eqnarray*}
\begin{split}
I(k,j,l,m)=&\{ i \in I(k): 2^{-l-1}  <d_Z(u_i,u_j) \le 2^{-l}, T_i \cap T_j \neq \emptyset,\\& \delta 2^{m+l-1} < d'(T_i \cap T_j, T_j \cap T) \le \delta 2^{m+l} \},
\end{split}
 \end{eqnarray*}
and for $m=0,$
\begin{eqnarray*}
\begin{split}
I(k,j,l,0)=&\{ i \in I(k): 2^{-l-1}  <d_Z(u_i,u_j) \le 2^{-l} ,T_i \cap T_j \neq \emptyset,\\& d'(T_i \cap T_j, T_j \cap T) \le \delta 2^{l} \}.
\end{split}
\end{eqnarray*} 
Then by Axiom 5,
\begin{equation}\label{cardIkjlm}
\# I(k,j,l,m) \lesssim \delta^{-(\lambda+\alpha)} 2^{-l} (2^{m+l})^{-\alpha}.
\end{equation}
This holds also when $m=0$, since we can trivially estimate 
\begin{equation}\label{W1}
\# I(k,j,l,0)\lesssim 2^{-lS}\delta^{-S} \le 2^{-l(\alpha+1)}\delta^{-(\lambda+\alpha)}.
\end{equation}
Since there are again only logarithmically many values of $l,m$, it suffices to show that for fixed $k,j,l,m$
\begin{equation}\label{toshow}
\int_{T_j} \left( \sum_{i \in I(k,j,l,m)} \chi_{T_i} \right)^{p-1} d \mu \lesssim \delta^{T} \delta^{\lambda+\alpha+1-p(\lambda+\alpha)}.
\end{equation}
For $i \in I(k,j,l,m)$ we have $d_Z(u_i,u_j) \approx 2^{-l}$, which implies by Axiom 3 that $\text{diam}_{d'}(T_i \cap T_j) \lesssim \delta 2^l$. Thus we need only to integrate over $T_j(l,m)=\{x \in T_j: d'(x,T_j \cap T) \lesssim \delta 2^{m+l} \}$. We have that $T_j(l,m) \subset T_j$ and $d_Z(u,u_j) \approx 2^{-k} \gtrsim 2^{-m-l}$, thus by Axiom 3 $\text{diam}_{d'}(T_j \cap T) \lesssim \delta 2^{m+l}$, which implies $\text{diam}_{d'}(T_j(l,m)) \lesssim \delta 2^{m+l}$. It follows by Axiom 1 that $\mu(T_j(l,m)) \lesssim 2^{m+l}\delta^{T+1}$. Using H\"older's inequality we get
\begin{eqnarray}\label{Wolff2}
\begin{split}
\int_{T_j} \left( \sum_{i \in I(k,j,l,m)} \chi_{T_i} \right)^{p-1} d \mu & \le \left( \int_{T_j} \sum_{i \in I(k,j,l,m)} \chi_{T_i} d \mu \right)^{p-1} \mu(T_j(l,m))^{2-p} \\
& \lesssim  \left( \sum_{i \in I(k,j,l,m)} \mu (T_i \cap T_j) \right)^{p-1} (2^{m+l} \delta^{T+1})^{2-p}.
\end{split}
\end{eqnarray}
Since $\text{diam}_{d'}(T_i \cap T_j) \lesssim \delta 2^l$, we have $\mu(T_i \cap T_j) \lesssim 2^l \delta^{T+1}$ by Axiom 1. It follows by \eqref{cardIkjlm}, 
\begin{eqnarray}\label{Wolff3}
\begin{split}
\int_{T_j} \left( \sum_{i \in I(k,j,l,m)} \chi_{T_i} \right)^{p-1} d \mu & \lesssim ( \# I(k,j,l,m) 2^l \delta^{T+1}) ^{p-1} (2^{m+l} \delta^{T+1})^{2-p}  \\
 &\lesssim 2^{(m+l)(\alpha+2-p(\alpha+1))} \delta^{T} \delta^{\lambda+\alpha+1-p(\lambda+\alpha)}  \\ &\le \delta^{T} \delta^{\lambda+\alpha+1-p(\lambda+\alpha)},
\end{split}
\end{eqnarray}
when $p \ge \frac{\alpha+2}{\alpha+1}$. Thus \eqref{toshow} holds.
\end{proof}

Using then Proposition \ref{PropDiscr} and Lemma \ref{hairbrush}, we can prove the following estimate for the Kakeya maximal function.

\begin{theorem}\label{Wolff}
Let $0< \delta<1$. Then for every $f\in L^{\alpha+2}(X,\mu)$ and every $\epsilon >0$,
\begin{equation*}
||f^d_\delta ||_{L^{\alpha+2}(Y, \nu)} \lesssim C_\epsilon \delta^{-\frac{2T-S+\lambda-2}{2(\alpha+2)}-\epsilon} ||f||_{L^{\alpha+2}(X, \mu)}.
\end{equation*}

Hence the Hausdorff dimension with respect to $d$ of every Kakeya set in $X$ is $\ge \frac{2Q-2T+S-\lambda+2}{2}-\theta(\alpha+2)$.
\end{theorem}

The claim about Kakeya sets follows from Corollary \ref{Lpbound}. Note that $\frac{2Q-2T+S-\lambda+2}{2}-\theta(\alpha+2)>0$ since $\lambda < 2Q-2T+S+2-2 \theta(\alpha+2)$. This lower bound improves the estimate $\frac{2Q-2T+S}{2}-\theta \frac{S+2}{2}$, which was found using Bourgain's method, only when $\lambda < 2+\theta(S-2\alpha-2)$. In the other cases it gives a worse (or equal) estimate.
The proof proceeds as in the Euclidean case (\cite{Mattila}, Theorem 23.5), but we show it here for completeness.

\begin{proof}
We may assume that $\text{diam}_{d_Z}(Y) \le 1$.
Let $\{u_1, \dots, u_m\}\subset Y$ be a $\delta$-separated subset.  
By Proposition \ref{PropDiscr} it suffices to show that
\begin{equation*}
\int ( \sum_{j=1}^m \chi_{T_j})^{(\alpha+2)/(\alpha+1)} d \mu \lesssim \delta^{-\frac{2T-S+\lambda-2}{2(\alpha+1)}-\epsilon} m \delta^{S} \delta^{(T-S) \frac{\alpha+2}{\alpha+1}}= m \delta^S \delta^{\frac{2(T-S)(\alpha+1)-S-\lambda+2}{2(\alpha+1)}-\epsilon},
\end{equation*}
where $T_j=\tilde{T}^{W \delta}_{u_j}(a_j) $. This is reduced to prove
\begin{equation*}
\sum_{j=1}^m \int_{T_j} ( \sum_{i=1}^m \chi_{T_i})^{1/(\alpha+1)} d \mu \lesssim  m \delta^{S} \delta^{\frac{2(T-S)(\alpha+1)-S-\lambda+2}{2(\alpha+1)}-\epsilon}.
\end{equation*}
If we subdivide into dyadic scale by letting $I(j,k)=\{i : 2^{-k-1}  < d_Z(u_i,u_j)\le 2^{-k}  \}$ for $k$ such that $\delta \le 2^{-k} $, then there are $N_\delta \approx \log(1/\delta)$ values to consider and we can take the sum in $k$ out of the integral. This follows from the fact that $1/(\alpha+1) \le 1$. Indeed we have
\begin{eqnarray*}
\begin{split}
\sum_{j=1}^m \int_{T_j} ( \sum_{i=1}^m \chi_{T_i})^{1/(\alpha+1)} d \mu &\le \sum_{j=1}^m \int_{T_j} (\sum_{k=1}^{N_\delta} \sum_{i \in I(j,k)} \chi_{T_i})^{1/(\alpha+1)} d \mu  + \sum_{j=1}^m \int_{T_j} \chi_{T_j} d \mu \\
& \le  \sum_{j=1}^m \sum_{k=1}^{N_\delta} \int_{T_j} (\sum_{i \in I(j,k)} \chi_{T_i})^{1/(\alpha+1)} d \mu + \sum_{j=1}^m \mu(T_j) \\
& \lesssim N_\delta \max_k \sum_{j=1}^m  \int_{T_j} (\sum_{i \in I(j,k)} \chi_{T_i})^{1/(\alpha+1)} d \mu + m \delta^T.
\end{split}
\end{eqnarray*} 
Since $m \delta^T \lesssim m \delta^{S} \delta^{\frac{2(T-S)(\alpha+1)-S-\lambda+2}{2(\alpha+1)}-\epsilon}$, we are left with pairs $i,j$ such that $d_Z(u_i,u_j) \approx 2^{-k}$. For a fixed $k$ we can cover $Y$ with balls $B_{d_Z}(v_l,2^{-k})$, $v_l \in Y$, such that the balls $ B_{d_Z}(v_l,2^{1-k})$ have bounded overlap. If $i \in I(j,k)$ then $u_i$ and $u_j$ belong to the same ball $ B_{d_Z}(v_l,2^{1-k})$ for some $l$. Fix one of these balls $B$ of radius $2^{1-k}$ and let $I(B)=\{i: u_i \in B\}$. Then it suffices to show that
\begin{equation}\label{claim}
\sum_{j \in I(B)} \int_{T_j} (\sum_{i \in I(B)} \chi_{T_i})^{1/(\alpha+1)} d \mu \lesssim \# I(B) \delta^{S} \delta^{\frac{2(T-S)(\alpha+1)-S-\lambda+2}{2(\alpha+1)}-\epsilon}.
\end{equation} 
 
The next step consists in finding as many $(N, \delta)$ hairbrushes as possible in the set of tubes indexed by $I(B)$, where $N$ will be chosen later. By doing so, one gets $H_1, \dots, H_M$ hairbrushes and, letting $H= H_1 \cup \dots \cup H_M$, we have that $K=I(B) \setminus H$ does not contain any $(N, \delta)$ hairbrush. 

Since $u_1, \dots, u_m$ are $\delta$-separated, the balls $B_{d_Z}(u_i, \delta/2)$ are disjoint. Moreover, $\delta \le 2^{-k}$ thus $B_{d_Z}(u_i, \delta/2) \subset 2B$, where $2B$ denotes the ball with the same center as $B$ and double radius. Hence
\begin{equation*}
\# I(B) \delta^S \lesssim \sum_{i \in I(B)} \nu(B_{d_Z}(u_i, \delta/2)) \le \nu(2B) \lesssim \nu(B) \lesssim 2^{-k S},
\end{equation*}
which implies $\# I(B) \lesssim 2^{-kS} \delta^{-S}$. Thus
\begin{equation}\label{M}
 M \lesssim 2^{-kS} \delta^{-S}/N.
\end{equation}
We can then split the sum into four parts
\begin{equation*}
\sum_{j \in I(B)} \int_{T_j} ( \sum_{i \in I(B)} \chi_{T_i} )^{1/(\alpha+1)} d \mu \le S(H,H)+ S(K,H)+ S(H,K)+ S(K,K),
\end{equation*}
where 
\begin{equation*}
S(K,H)=\sum_{j \in K} \int_{T_j} (\sum_{i \in H} \chi_{T_i})^{1/(\alpha+1)} d \mu 
\end{equation*}
and similarly for the others. For the first term by Minkowski's inequality and Lemma \ref{hairbrush} we have
\begin{eqnarray}\label{SHH}
\begin{split}
S(H,H)^{(\alpha+1)/(\alpha+2)}&= || \sum_{i \in H} \chi_{T_i}||_{(\alpha+2)/(\alpha+1)} \le  \sum_{l=1}^M ||\sum_{i \in H_l}  \chi_{T_i}||_{(\alpha+2)/(\alpha+1)}  \\
& \lesssim \sum_{l=1}^M (\delta^T \# H_l \delta^{(-\lambda+1)/(\alpha+1)- \epsilon})^{(\alpha+1)/(\alpha+2)},
\end{split}
\end{eqnarray}
thus by H\"older's inequality and \eqref{M} we get
\begin{eqnarray}\label{SHH2}
\begin{split}
S(H,H)& \lesssim M^{1/(\alpha+1)} \delta^{-\epsilon} \#H \delta^{T} \delta^{(-\lambda+1)/(\alpha+1)}  \\ &\lesssim  \delta^{-\epsilon} \# I(B) \delta^{T} \left( \frac{2^{-kS} \delta^{-S-\lambda+1}}{N} \right)^{1/(\alpha+1)}\\
&=\delta^{-\epsilon} \# I(B) \delta^S \left( \frac{2^{-k S} \delta^{-S-\lambda+1+(T-S)(\alpha+1)}}{N}\right)^{1/(\alpha+1)} .
\end{split}
\end{eqnarray}
For the second term, using twice H\"older's inequality we get
\begin{eqnarray*}
\begin{split}
S(K,H) &\le \sum_{j \in K} \left( \int_{T_j} \sum_{i \in H} \chi_{T_i} d \mu \right)^{1/(\alpha+1)} \mu(T_j)^{\alpha/(\alpha+1)}  \\ & \lesssim (\# K)^{\alpha/(\alpha+1)} \left( \sum_{j \in K} \int_{T_j} \sum_{i \in H} \chi_{T_i} d \mu  \right)^{1/(\alpha+1)}  \delta^{T\alpha/(\alpha+1)}.
\end{split}
\end{eqnarray*}
Since  $d_Z(u^i,u^j) \approx 2^{-k}$, it follows by Axioms 1 and 3 that $\mu(T_i \cap T_j)\lesssim 2^k \delta^{T+1}$. Using the fact that $\# \{ i \in K: T_i \cap T \neq \emptyset \} < N$ for any tube $T$, we get
\begin{eqnarray}\label{SKH}
\begin{split}
S(K,H) & \lesssim (\# K)^{\alpha/(\alpha+1)} \left( \sum_{i \in H} \sum_{j \in K, T_i \cap T_j \neq \emptyset} 2^k \delta^{T+1} \right)^{1/(\alpha+1)} \delta^{T\alpha/(\alpha+1)} \\ & \lesssim (\# K)^{\alpha/(\alpha+1)} \left( \sum_{i \in H} N 2^k \delta^{T+1} \right)^{1/(\alpha+1)} \delta^{T\alpha/(\alpha+1)} \\ & \le \# I(B) \delta^{S} (N 2^k \delta^{(T-S)(\alpha+1)+1})^{1/(\alpha+1)}.
\end{split}
\end{eqnarray}
The remaining two terms can be estimated in the same way, obtaining
\begin{equation}\label{SKK}
S(H,K)+ S(K,K) \lesssim \# I(B) \delta^{S} (N 2^k \delta^{(T-S)(\alpha+1)+1})^{1/(\alpha+1)}. 
\end{equation}
Choosing then $N= 2^{-\frac{S+1}{2}k} \delta^{-\frac{S+\lambda}{2}}$, we get \eqref{claim} by \eqref{SHH2}, \eqref{SKH} and \eqref{SKK}.
\end{proof}

\part{ Examples of applications}\label{secondpart}

\section{ Classical Kakeya sets}\label{Kaksets}

We have seen in Remark \ref{axEucl} that in the case of the classical Kakeya sets Axioms 1-4 are satisfied with $T=S=n-1$, $Q=n$ and $\theta=0$. 
We then get by Theorem \ref{Bourgain} the lower bound $\frac{n+1}{2}$ for the Hausdorff dimension of Kakeya sets (which was obtained by Bourgain).

Moreover, Axiom 5 holds with $\lambda=1$ and $\alpha=n-2$ as seen in Remark \ref{ax5Eucl}. Hence we obtain the improved lower bound $\frac{n+2}{2}$ proved originally by Wolff.  
In Figure \ref{fig1} it is shown how Bourgain's bush and Wolff's hairbrush look like in this case.

We now consider some more examples to which the axiomatic setting can be applied and one to which it cannot.
\begin{figure}[H]
\includegraphics[scale=0.20]{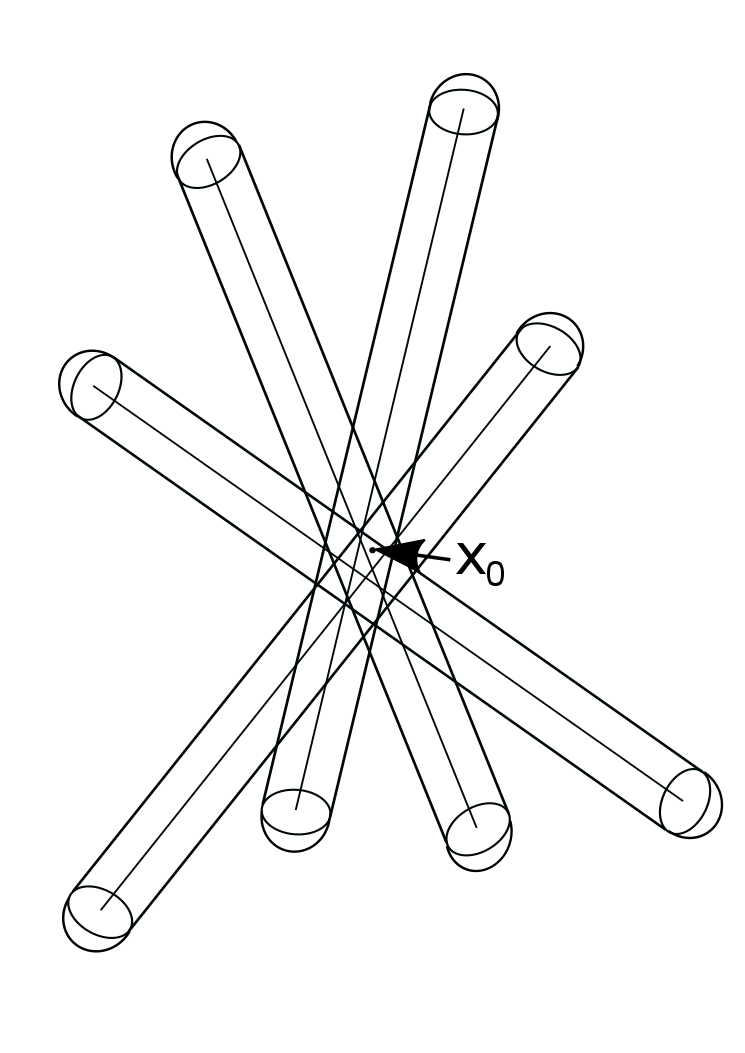}
\hfill
\includegraphics[scale=0.20]{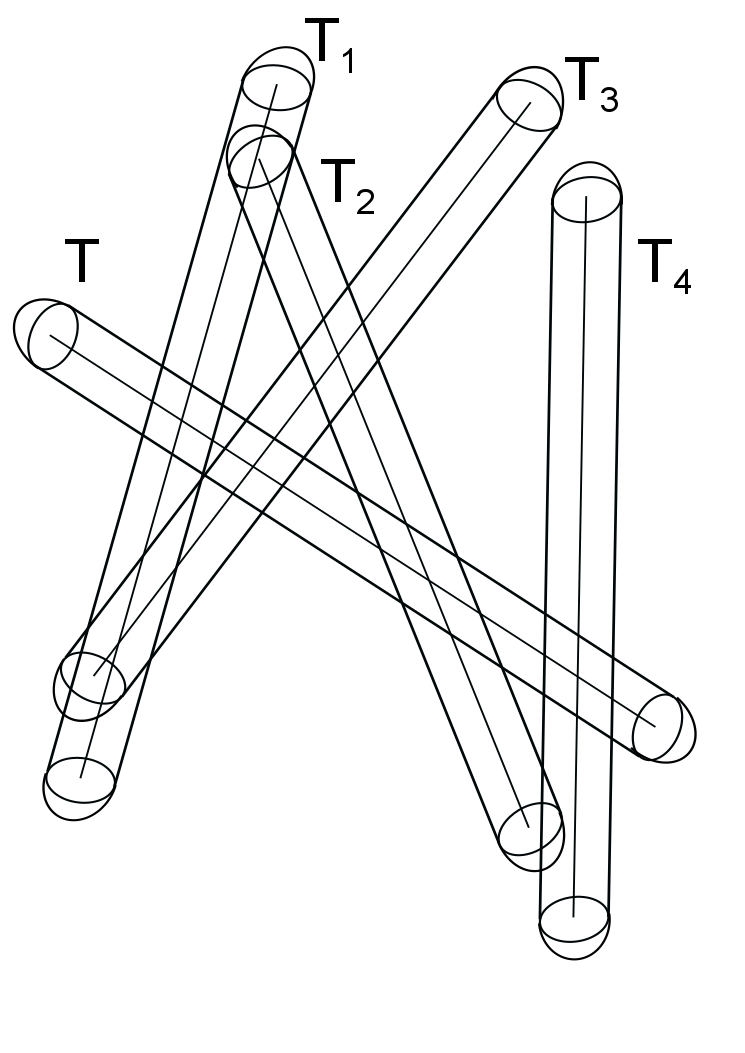}
\caption{Bourgain's bush and Wolff's hairbrush in the classical Euclidean case (in $\mathbb{R}^3$)}
\label{fig1}
\end{figure}

\section{Nikodym sets}\label{Nikodym}

Nikodym sets are closely related to Kakeya sets. A \textit{Nikodym set} $E \subset \mathbb{R}^n$ is such that $\mathcal{L}^n(E)=0$ and for every $x \in \mathbb{R}^n$ there is a line $L$ through $x$ such that $E \cap L$ contains a unit line segment.
The Nikodym conjecture states that every Nikodym set in $\mathbb{R}^n$ has Hausdorff dimension $n$. This is implied by the Kakeya conjecture, see Theorem 11.11 in \cite{Mattila}.

The Nikodym maximal function of $f \in L^1_{loc}(\mathbb{R}^n)$ is defined as $f^{**}_\delta: \mathbb{R}^n \rightarrow [0, \infty]$,
\begin{equation*}
f^{**}_\delta(x) = \sup_{x \in T} \frac{1}{\mathcal{L}^n(T)} \int_T |f| d \mathcal{L}^n,
\end{equation*}
where the supremum is taken over all tubes $T$ of radius $\delta$ and length $1$ containing $x$. There is also a Nikodym maximal conjecture, stating that
\begin{equation*}
||f^{**}_\delta||_{L^n(\mathbb{R}^n)} \le C_{n,\epsilon} \delta^{- \epsilon} ||f||_{L^n(\mathbb{R}^n)}
\end{equation*}
for every $\epsilon >0$, $0 < \delta <1$. This is equivalent to the Kakeya maximal conjecture (see Theorem 22.16 in \cite{Mattila}).

Making some reductions, we can use the axiomatic setting to prove the known estimates $\frac{n+1}{2}$ and $\frac{n+2}{2}$ for the Hausdorff dimension of Nikodym sets (these lower bounds were originally proved by Bourgain and Wolff). 

First we will consider a natural setting in which the roles of $Y$ and $\mathcal{A}$ are basically swapped with respect to the Kakeya case, but in which we can only prove the lower bound $\frac{n+1}{2}$. In Section \ref{Nikmod} we will then consider a different approach, which will yield the lower bound $\frac{n+2}{2}$ for the Hausdorff dimension of Nikodym sets but not the corresponding Nikodym maximal function inequality. It will also give lower bounds for the dimension of sets containing a segment in a line through every point of a hyperplane.

Let $X=\mathbb{R}^n$, $\mu = \mathcal{L}^n$, $d=d'$ be the Euclidean metric, $Q=n$. The set of parameters $\mathcal{A}$ is given by those directions $e \in S^{n-1}$ that make an angle $< \pi/ 100$ with the $x_n$-axis. Let $Z$ be the $x_1, \dots, x_{n-1}$-hyperplane, $\nu=\mathcal{H}^{n-1}_E \big|_Z$, $Y$ be a compact subset of $Z$ such that $0<\mathcal{H}^{n-1}_E(Y) \le 1$, $d_Z$ be the Euclidean metric on $Z$. Then for every $p \in Y$ we have $\mathcal{H}^{n-1}_E(B_E(p,r)) \approx r^{n-1}$, thus $S=n-1$.

For $p \in Y$ and $e \in \mathcal{A}$, we define $F_p(e)=I_p(e)$ as a segment of unit length with direction $e$ starting from $p$ and $\tilde{I}_p(e) \supset I_p(e)$ as a segment of length 2 (starting from $p$). In this case $\mu_{p,e}=\mathcal{H}^1_E \big|_{I_p(e)}$, the $1$-dimensional Euclidean Hausdorff measure restricted to $I_p(e)$. Let $T^\delta_p(e)$ be the $\delta$ neighbourhood of $I_p(e)$ in the Euclidean metric and let $\tilde{T}^{2 \delta}_p(e)$ be the $2  \delta$ neighbourhood of $\tilde{I}_p(e)$. 

\textbf{Axiom 1}: We have $\mathcal{L}^n(\tilde{T}^{2 \delta}_p(e)) \approx \mathcal{L}^n(T^\delta_p(e)) \approx \delta^{n-1}$ and if $A \subset \tilde{T}^{2 \delta}_p$ then $\mathcal{L}^n(A) \lesssim \text{diam}_E(A) \delta^{n-1}$. Thus Axiom 1 holds with $T=S=n-1$. 

\textbf{Axiom 2}: It is easy to see that Axiom 2 holds with $\theta=0$ (since the tubes and the balls are the same as in the classical Kakeya case). 

\textbf{Axiom 3}: Now we show that Axiom 3 is also satisfied, that is there exists $b=b_n$ such that for every $e, \bar{e} \in \mathcal{A}$ and every $p, \bar{p} \in Y$,
\begin{equation}\label{diamNik}
\text{diam}_E(\tilde{T}^{2\delta}_p(e) \cap \tilde{T}^{2\delta}_{\bar{p}}(\bar{e})) \le b \frac{\delta}{|p-\bar{p}|},
\end{equation} 
where $\text{diam}_E$ denotes the diameter with respect to the Euclidean metric. \\
Indeed, if $e= \bar{e}$ or $|e- \bar{e}| \ll \delta$ then the intersection is non-empty only if $|p - \bar{p}| \lesssim 2 \delta$. In this case the left-hand side in \eqref{diamNik} is at most $1$ up to a constant and the right-hand side is $\gtrsim 1$. If $|e - \bar{e}| \gtrsim \delta$ then the intersection is non-empty only if $|p- \bar{p}| \lesssim |e- \bar{e}|$. Thus by the standard diameter estimate,
\begin{equation*}
\text{diam}_E(\tilde{T}^{2\delta}_p(e) \cap \tilde{T}^{2\delta}_{\bar{p}}(\bar{e})) \lesssim \frac{\delta}{|e - \bar{e}|} \lesssim \frac{\delta}{|p-\bar{p}|}.
\end{equation*} 

\textbf{Axiom 4}: Let $p,p' \in Y$ such that $|p-p'| \le \delta$ and let $e \in \mathcal{A}$. We want to show that $T^\delta_p(e) \subset \tilde{T}^{2 \delta}_{p'}(e)$. The segment $I_p(e)$ is the set of points $te+p$ with $0 \le t \le 1$. Let $q \in T^\delta_p(e)$, that is there exists $t \in [0,1]$ such that $|q-te-p| \le \delta$. Then 
\begin{equation}\label{Ax4Nik}
|q-te-p'| \le |q-te-p|+|p-p'| \le 2 \delta,
\end{equation}
thus $q \in \tilde{T}^{2 \delta}_{p'}(e)$, where $\tilde{I}_{p'}(e)=\{ se+p': 0 \le s \le 2 \}$.

Hence all the Axioms $1-4$ are satisfied. Defining, as in \eqref{fddelta}, $f^d_\delta : Y \rightarrow [0, \infty]$,
\begin{equation}\label{fdNik}
f^d_\delta(p)= \sup_{e \in \mathcal{A}} \frac{1}{\mathcal{L}^n(T^\delta_p(e))} \int_{T^\delta_p(e)} |f|d \mathcal{L}^n,
\end{equation}
this satisfies by Bourgain's method a weak type inequality \eqref{weakbourg} with $S=T=n-1$ for all Lebesgue measurable sets.

\begin{rem}\label{Nikest}
Note that any estimate of the form $||f^d_\delta||_{L^p(\mathbb{R}^{n-1})} \le C_\delta ||f||_{L^p(\mathbb{R}^n)}$ valid for any $f \in L^p(\mathbb{R}^n)$ with bounded support, implies the corresponding estimate $||f^{**}_\delta||_{L^p(\mathbb{R}^n)} \le C_n C_\delta ||f||_{L^p(\mathbb{R}^n)}$. Indeed, the assumption means that
\begin{equation}\label{fd100}
\int_{\mathbb{R}^{n-1}} (f^{**}_{\delta,100}(x',0))^p d \mathcal{L}^{n-1} x' \le C_\delta \int_{\{(x',x_n): |x_n| \le 1 \}}|f(x',x_n)|^p d \mathcal{L}^{n-1} x' dx_n,
\end{equation}
where $f^{**}_{\delta, 100}$ is the maximal function as $f^{**}_{\delta}$ but with tubes that make an angle $< \pi/100$ with the $x_n$-axis. Actually there is a small difference between $f^d_\delta$ and $f^{**}_{\delta,100}$: given a point $p$, in $f^d_\delta$ we consider averages over tubes starting from $p$ whereas in $f^{**}_{\delta, 100}$ the tubes just contain $p$. However, any tube $T^\delta_p(e)$ is a tube containing $p$, thus $f^d_\delta(p) \le f^{**}_{\delta, 100}(p)$. On the other hand, if $T$ is any tube containing $p$, say that $T=T^\delta_q(e)$ for some point $q$, then $T \subset T^{2 \delta}_p(e) \cup T^{2\delta}_p(-e)$. Indeed, $|p-te-q| \le \delta$ for some $t \in [0,1]$ and if $q' \in T$ then $|q'-se-q| \le \delta$ for some $s \in [0,1]$. Thus $|q'-(s-t)e-p| \le |q'-se-q|+|q+te-p|\le 2 \delta$ and $s-t \in [-1,1]$. Hence $f^{**}_{\delta, 100}(p) \lesssim f^d_{2\delta}(p)$.

Since in \eqref{fd100} $0$ could be replaced by any $t \in \mathbb{R}$, we have
\begin{equation*}
\int_{\mathbb{R}^{n-1}} (f^{**}_{\delta,100}(x',t))^p d \mathcal{L}^{n-1} x' \le C_\delta \int_{\{(x',x_n): |x_n-t| \le 1 \}}|f(x',x_n)|^p d \mathcal{L}^{n-1} x' dx_n.
\end{equation*}
For any $t$ there exists $k \in \mathbb{Z}$ such that $t \in [k,k+1]$. Thus 
\begin{equation*}
\int_{\mathbb{R}^{n-1}} (f^{**}_{\delta,100}(x',t))^p d \mathcal{L}^{n-1} x' \le C_\delta \int_{\{(x',x_n): |x_n-k| \le 2 \}}|f(x',x_n)|^p d \mathcal{L}^{n-1} x' dx_n.
\end{equation*}
Integrating over $t$ and summing over $k$, we have
\begin{equation*}
\sum_{k \in \mathbb{Z}} \int_{k}^{k+1}\int_{\mathbb{R}^{n-1}} (f^{**}_{\delta,100}(x',t))^p d \mathcal{L}^{n-1} x' dt \le C_\delta \sum_{k \in \mathbb{Z}} \int_{\{(x',x_n): |x_n-k| \le 2 \}}|f(x',x_n)|^p d \mathcal{L}^{n-1} x' dx_n,
\end{equation*}
thus by Fubini's theorem we get the estimate $||f^{**}_{\delta, 100}||_{L^p(\mathbb{R}^n)} \le C_n C_\delta ||f||_{L^p(\mathbb{R}^n)}$.  The restriction on the direction of the tubes can be removed by using finitely many different choices of coordinates.

In the same way we can show that any weak type inequality of the form
\begin{equation*}
\mathcal{L}^{n-1}(\{p=(x',0): x'  \in \mathbb{R}^{n-1}, (\chi_E)^d_\delta(p) > \lambda \}) \le C \lambda^{-p} \delta^{- \beta p} \mathcal{L}^n(E)
\end{equation*}
valid for any Lebesgue measurable set $E \subset \mathbb{R}^n$ and any $\lambda>0$, implies the corresponding estimate
\begin{equation*}
\mathcal{L}^{n}(\{p \in \mathbb{R}^{n}: (\chi_E)^{**}_\delta(p) > \lambda \}) \le C \lambda^{-p} \delta^{- \beta p} \mathcal{L}^n(E).
\end{equation*}
\end{rem}

Hence we have a weak type inequality also for the Nikodym maximal function 
\begin{equation*}
\mathcal{L}^{n}(\{p \in \mathbb{R}^{n}: (\chi_E)^{**}_\delta(p) > \lambda \}) \le C \lambda^{-(n+1)/2} \delta^{- (n-1)/2} \mathcal{L}^n(E),
\end{equation*}
which implies the lower bound $\frac{n+1}{2}$ for the Hausdorff dimension of any Nikodym set in $\mathbb{R}^n$.

Wolff proved the lower bound $\frac{n+2}{2}$ simultaneously for the Hausdorff dimension of Kakeya and Nikodym sets (see Remark \ref{WA}). Unfortunately with our approach it seems that we cannot prove the validity of Axiom 5 in the present setting with $\lambda=1$, $\alpha = n-2$.  The main obstacle is the fact that here if two tubes $\tilde{T}^{2 \delta}_p(e)$ and $\tilde{T}^{2 \delta}_{\bar{p}}(\bar{e})$ intersect then it could happen that $|p-\bar{p}|$ is much smaller than $|e-\bar{e}|$, thus having information about the distance between the starting points of two tubes is not enough to know the angle at which they intersect. The validity of Axiom 5 would give an $L^n$ bound for the Nikodym maximal function, which would imply the lower bound $\frac{n+2}{2}$ for the Hausdorff dimension of Nikodym sets.

We will now use another approach, letting $Y$ and $\mathcal{A}$ be different sets, which will also give dimension estimates for some related sets. 

\subsection{Sets containing a segment in a line through every point of a hyperplane}\label{Nikmod}

Let $V \subset \mathbb{R}^n$ be a hyperplane and let $A \subset V$ be $\mathcal{H}^{n-1}$ measurable and such that $\mathcal{H}^{n-1}(A)>0$. We say that $N \subset \mathbb{R}^n$ is an $(A,V)$-\textit{Nikodym set} if $\mathcal{L}^n(N)=0$ and for every $p\in A$ there is a line $L_p$ through $p$ not contained in $V$ such that $L_p \cap N$ contains a segment of length $1$.
We will obtain the following dimension estimate.

\begin{theorem}\label{VNik}
If $V \subset \mathbb{R}^n$ is a hyperplane, $A \subset V$ is $\mathcal{H}^{n-1}$ measurable, $\mathcal{H}^{n-1}(A)>0$, and $N \subset \mathbb{R}^n$ is an $(A,V)$-Nikodym set then the Hausdorff dimension of $N$ is $\ge \frac{n+2}{2}$.
\end{theorem}

We will prove this by showing that Axioms 1-5 are satisfied and using Wolff's method.
The setting is the following.
Let $X=\mathbb{R}^n$, $d=d'$ be the Euclidean metric and $\mu=\mathcal{L}^n$, thus $Q=n$. Let $Y \subset Z= V$ be compact and such that $0<\mathcal{H}^{n-1}(Y)\le 1$. We let $d_Z$ be the Euclidean metric on $V$ (thus $B_V(p,r)$ is any Euclidean ball in $V$) and $\nu= \mathcal{H}^{n-1} \big|_V$. Thus $S=n-1$.\\
Given $0< \sigma < \pi/2$, let 
\begin{equation*}
\mathcal{A}=\{(e,t): e \in S^{n-1}, \measuredangle(e,V^\bot) < \sigma, t \in \mathbb{R}, \frac{1}{4} <t < C_\sigma \}, 
\end{equation*}
where $\measuredangle(e,V^\bot) $ denotes the angle between a line with direction $e$ and $V^\bot$ and $ C_\sigma$ is a fixed constant depending only on $\sigma$.
For $p \in Y$ and $(e,t) \in \mathcal{A}$ let
\begin{equation}\label{Ipet}
F_p(e,t)=I_p(e,t)=\{p+te+se: s\in [0,1]\},
\end{equation}
thus $I_p(e,t)$ is the segment starting from $p+te$ with direction $e$ and length $1$.\\
Here we consider only $0<\delta<\frac{1}{16(1+\tan \sigma)}$ (this is needed in \eqref{ax31} and \eqref{ax32} and it is not restrictive since we could define tubes and prove all the results of Sections 3, 4, 5 for $0<\delta<c<1$ for any $c$).
Let $T^\delta_p(e,t)$ be the $\delta$ neighbourhood of $I_p(e,t)$ in the Euclidean metric, let $\tilde{I}_p(e,t)= \{p+te+se: s\in [0,2]\}$ and $\tilde{T}^{2 \delta}_p(e,t)$ be its $2\delta$ neighbourhood. 

Here $\lesssim$ (resp. $\gtrsim$) means $\le C_{n,\sigma}$ (resp. $\ge C_{n,\sigma}$) for some constant $C_{n,\sigma}$.

\textbf{Axiom 1}: It holds with $T=n-1$, since the tubes are the usual Euclidean ones.

\textbf{Axiom 2}: Since the tubes and balls are the same as in the Euclidean Kakeya case, Axiom 2 holds with $\theta=0$.

\textbf{Axiom 3}: Let $p \neq \bar{p} \in Y$ and $(e,t), (\bar{e}, \bar{t}) \in \mathcal{A}$, $e \neq \bar{e}$, be such that $\tilde{T}^{2 \delta}_p(e,t) \cap \tilde{T}^{2 \delta}_{\bar{p}}(\bar{e}, \bar{t}) \neq \emptyset$. 
First observe that if $|p-\bar{p}| \le C_{n\sigma} \delta$ then $\delta/|p-\bar{p}| \gtrsim 1 \gtrsim \text{diam}_E(\tilde{T}^{2 \delta}_p(e,t) \cap \tilde{T}^{2 \delta}_{\bar{p}}(\bar{e}, \bar{t}))$, thus Axiom 3 holds. 

Hence we can assume that $|p-\bar{p}| > \frac{10 \delta}{\cos \sigma}$. We can find points $p', \bar{p}'  \in Y$ such that $|p-p'| \le \frac{2 \delta}{\cos \sigma}$, $|\bar{p}-\bar{p}'| \le \frac{2 \delta}{\cos \sigma}$ and the point $u = \{p'+se: s \in \mathbb{R} \} \cap \{\bar{p}'+s \bar{e}: s \in \mathbb{R} \}$ is contained in $\tilde{T}^{2 \delta}_p(e,t) \cap \tilde{T}^{2 \delta}_{\bar{p}}(\bar{e}, \bar{t}) $. Thus 
\begin{equation}\label{ax31}
\frac{1}{8}<\frac{1}{4}-2\delta(1+\tan \sigma) < |u-p'|< C_\sigma+1+2\delta(1+\tan \sigma)< C_\sigma+\frac{9}{8}
\end{equation}
and 
\begin{equation}\label{ax32}
\frac{1}{8}< \frac{1}{4}-2 \delta(1+\tan \sigma)< |u-\bar{p}'|< C_\sigma+1+2 \delta(1+\tan \sigma)<  C_\sigma+\frac{9}{8}. 
\end{equation}
Since  $\measuredangle(e,V^\bot) < \sigma $, the angle between the line $\{p'+se: s \in \mathbb{R} \}$ and $V$ is $> \pi/2-\sigma$. This implies that $|p'-\bar{p}'| \approx |e-\bar{e}|$, which is essentially the angle between the lines $\{p'+se: s \in \mathbb{R} \}$ and $\{\bar{p}'+s \bar{e}: s \in \mathbb{R} \} $.
Since $|p-\bar{p}| \approx |p'- \bar{p}'|$, the classical diameter estimate implies
\begin{equation}\label{diamNm}
\text{diam}_E(\tilde{T}^{2 \delta}_p(e,t) \cap \tilde{T}^{2 \delta}_{\bar{p}}(\bar{e}, \bar{t})) \lesssim \frac{\delta}{|e- \bar{e}|} \lesssim \frac{\delta}{|p-\bar{p}|},
\end{equation}
which proves the validity of Axiom 3. Observe that if $e=\bar{e}$ then $\tilde{T}^{2 \delta}_p(e,t) \cap \tilde{T}^{2 \delta}_{\bar{p}}(\bar{e}, \bar{t}) \neq \emptyset$ only if $|p-\bar{p}|\lesssim \frac{ 2 \delta}{\cos \sigma}$. Also in this case \eqref{diamNm} still holds.

\textbf{Axiom 4}: Let $p, \bar{p} \in Y$ be such that $|p-\bar{p}| \le \delta$. We want to show that $T^\delta_p(e,t) \subset \tilde{T}^{2 \delta}_{\bar{p}}(e,t)$. Indeed, if $w \in T^\delta_p(e,t)$ then $|w-p - te-se| \le \delta$ for some $s \in [0,1]$. Thus
\begin{equation}\label{a4NM}
|w- \bar{p}- te-se| \le |w-p - te-se|+ |p-\bar{p}| \le 2 \delta,
\end{equation}
which implies $w \in \tilde{T}^{2 \delta}_{\bar{p}}(e,t)$. Hence Axiom 4 holds.

\textbf{Axiom 5}: We can prove Axiom 5 as in the classical Kakeya case, see Remark \ref{ax5Eucl} (and  Lemma 23.3 in \cite{Mattila}).

\begin{lem}\label{ax5Nm}
Let $0<\delta, \beta, \gamma < 1$ and let $T_1, \dots, T_N$ be such that $T_j=\tilde{T}^{2\delta}_{p_j}(e_j,t_j)$, $|p_j-p_k|>\delta$ for every $j \ne k$. Let $T=\tilde{T}^{2\delta}_p(e,t)$ be such that $T \cap T_j \ne \emptyset$ and $|p_j -p| \ge \beta/8$ for every $j=1, \dots, N$. Then for all $j=1, \dots, N$,
\begin{eqnarray}\label{ax5V}
\begin{split}
& \# \mathcal{I}_j =\# \{i: |p_i-p_j| \le \beta, T_i \cap T_j \ne \emptyset, d_E(T_i \cap T_j, T_j \cap T) \ge \gamma \}  \\
&\le C_{n,\sigma} \delta^{-1} \beta \gamma^{2-n}.
\end{split}
\end{eqnarray}
\end{lem}

\begin{proof}
We may assume that $\delta$ is much smaller than $\gamma$ and $\beta$. This follows from the fact that since $p_i,p_j$ are $\delta$-separated we have $\# \{i: |p_i-p_j| \le \beta \} \le \beta^{n-1} \delta^{1-n}$. If $\delta \gtrsim \gamma$ we would have $\# \{i: |p_i-p_j| \le \beta \} \lesssim \beta^{n-1} \delta^{-1} \gamma^{2-n} $ and if $\delta \gtrsim \beta$ then $ \# \{i: |p_i-p_j| \le \beta \} \lesssim  \beta \delta^{-1}$, thus \eqref{ax5V} would hold.

Since $|p-p_j| \ge \beta/8 >> \delta$ for every $j=1, \dots, N$, we have seen above (in Axiom 3) that we have $|e-e_j| \approx |p-p_j| \gtrsim \beta$ and there exist $p', p_j'$ such that $|p-p'| \le \frac{2\delta}{\cos \sigma}$, $|p_j-p'_j| \le \frac{2\delta}{\cos \sigma}$ and the lines $\{p'+se_j: s \in \mathbb{R} \}$ and $\{p'_j+se: s \in \mathbb{R} \}$ intersect.

Fix now one of these $j$. Since $|p_i-p_j| > \delta$, we have $ \#\{ i: |p_i-p_j| \le 10 \delta/ \cos \sigma\} \lesssim 1$, hence we can assume that $|p_i-p_j| > 10\delta/\cos \sigma$. Then for $i \in \mathcal{I}_j$, we have $|e_i-e_j| \approx |p_i-p_j| \le \beta$, hence we are essentially in the same situation as for the classical Kakeya case and we can use the same proof, which we summarize here.

Let $P$ be the $2$-dimensional plane spanned by the lines $\{p'+se: s \in \mathbb{R} \}$ and $\{p'_j+s e_j: s \in \mathbb{R} \}$ and let $L$ be the line given by the intersection between $P$ and $V$. Thus $L$ contains $p'$ and $p'_j$.
Observe that the angle between $P$ and $V$ (that is, the angle between their normal vectors) is $\gtrsim \pi/2- \sigma $. 

Since $T_i$ intersects $T$ and $T_j$ in such a way that the angle between $T_i$ and $T_j$ is at most constant times the angle between $T$ and $T_j$, it follows from the fact that $  d_E(T_i \cap T_j, T_j \cap T) \ge \gamma$ that also $d_E(T_i \cap T_j, T_i \cap T) \gtrsim \gamma$. Hence $T_i$ makes an angle $\lesssim \delta/\gamma$ with $P$. Thus the distance from $p_i$ to $L$ is $\lesssim \delta/\gamma$. Moreover, $p_i \in B_V(p_j,\beta)$, hence 
\begin{equation*}
|p_i-p'_j| \le |p_i-p_j|+|p_j-p'_j| \le \beta + \frac{2\delta}{\cos \sigma} \le 2\beta
\end{equation*}
since $\delta <<\beta$.
It follows that
\begin{equation}
p_i \in \left\{q \in V: d_E(q,L) \lesssim \frac{\delta}{\gamma} \right\} \cap B_V(p'_j,2 \beta).
\end{equation}
Since this set has  $\mathcal{H}^{n-1}$ measure $\lesssim \beta \delta^{n-2} \gamma^{2-n}$, it can contain $\lesssim \beta \delta^{-1}  \gamma^{2-n}$  $\delta$-separated points $p_i$. 
\end{proof}

For $f\in L^1_{loc}(\mathbb{R}^n)$ and $0<\delta <1$ we define the maximal function $f^d_\delta: Y \rightarrow [0, \infty]$ as in \eqref{fddelta} by
\begin{equation*}
f^d_\delta(p)= \sup_{(e,t) \in \mathcal{A}} \frac{1}{\mathcal{L}^n(T^\delta_p(e,t))} \int_{T^\delta_p(e,t)} |f| d \mathcal{L}^n.
\end{equation*}
Since all the Axioms $1-5$ are satisfied, Theorem \ref{Wolff} implies the following.

\begin{theorem}\label{fVNik}
There exists a constant $C=C_{n,\sigma}$ such that for every $f \in L^n(\mathbb{R}^n)$ and every $\epsilon>0$,
\begin{equation*}
||f^d_\delta||_{L^n(Y)} \le C C_\epsilon \delta^{\frac{2-n}{2n}-\epsilon} ||f||_{L^n(\mathbb{R}^n)}.
\end{equation*}
\end{theorem}

If $N$ is an $(A,V)$-Nikodym set then there exists $Y \subset V$ and $\mathcal{A}$ as above such that for every $p \in Y$ there exists $(e,t) \in \mathcal{A}$ with $I_p(e,t) \subset N$, which means that $(A,V)$-Nikodym sets are generalized Kakeya sets. Indeed, for every $p \in A$ there exists a half line $L_p=\{se_p+p: s \ge 0\}$ for some $e_p \in S^{n-1}$ such that $L_p \cap N$ contains a segment of length $1/2$, call it $I_p(e_p,t_p)$ (where $t_p$ is such that $p+t_pe_p$ is the starting point of the segment). Since $L_p$ is not contained in $V$ we have $\measuredangle(e_p, V^\bot) < \pi/2$ for every $e_p$ as above. If for some $p$ the segment $I_p(e_p,t_p)$ contains $p$, that is $t_p=0$, then we can redefine $I_p(e_p,0)=I^l_p(e_p,1/4)=\{ p+ se_p: s \in [1/4,1/2]\}= \{ p+1/4e_p+se_p: s \in [0,1/4] \}$. Thus we can assume $t_p>1/4$.

For $R=1,2,\dots$, let $V_R=\{p \in V: t_p < R \}$. Since $V= \cup_{R=1}^\infty V_R$, there exists $R$ such that $\mathcal{H}^{n-1}(V_R)>0$. For $i=1,2,\dots$, let $V_{R,i} =\{ p \in V_R: \measuredangle (e_p,V^\bot) < \pi/2 - 1/i\}$. Then there exists $i$ such that $\mathcal{H}^{n-1}(V_{R,i})>0$. Let $Y\subset V_{R,i}$ such that $0<\mathcal{H}^{n-1}(Y) \le 1$ and $Y$ is compact. Then for every $p \in Y$ there exists $(e_p,t_p) \in \mathcal{A}$ (where $\sigma = \pi/2-1/i$ and $C_\sigma=R$) such that $I_p(e_p,t_p) \subset N$. 

Hence Theorem \ref{fVNik} implies by Corollary \ref{Lpbound} that the Hausdorff dimension of every $(A,V)$-Nikodym set is $\ge \frac{n+2}{2}$, that is Theorem \ref{VNik} is proved.

\begin{rem}
If $M \subset \mathbb{R}^n$ is a Nikodym set, then in particular there is a hyperplane $V \subset \mathbb{R}^n$ such that for every $p \in V$ there exists a line $L_p$ through $p$ with $L_p \cap M$ containing a unit segment and $\mathcal{H}^{n-1}(\{p \in V: L_p \nsubseteq V \})>0$. Thus $M$ is an $(A,V)$-Nikodym set (where $A=\{p \in V: L_p \nsubseteq V \}$), which implies $\dim M \ge \frac{n+2}{2}$. 
\end{rem}

\begin{rem}
We considered here sets such that for every $p \in V$ the line $L_p$ is not contained in $V$. On the other hand, if a set $N \subset \mathbb{R}^n$ is such that for every $p \in V$ there exists a line $L_p \subset V$ with $L_p \cap N$ containing a unit line segment, then $N$ is essentially a Nikodym set in $\mathbb{R}^{n-1}$. Thus in this case we only have the known lower bounds for the dimension of Nikodym sets in $\mathbb{R}^{n-1}$.
\end{rem}

\section{Sets containing a segment in a line through almost every point of an $(n-1)$-rectifiable set}\label{NikodLip}

Instead of the classical Nikodym sets, we now consider sets containing a segment in a line through almost every point of an $(n-1)$-rectifiable set with direction not contained in the approximate tangent plane (this will be defined later).

There are various equivalent definitions of rectifiable sets (see chapters 15-18 in \cite{Pertti} for definitions and properties of rectifiable sets). We recall here two definitions that we will use. Let $E \subset \mathbb{R}^n$ be an $\mathcal{H}^{n-1}$ measurable set with $0 < \mathcal{H}^{n-1}(E) < \infty$. Then 
\begin{enumerate}
\item $E$ is $(n-1)$-rectifiable if and only if there exist $(n-1)$-dimensional Lipschitz graphs $G_1, G_2, \dots$ such that $\mathcal{H}^{n-1} (E \setminus \cup_j G_j)=0$;

\item $E$ is $(n-1)$-rectifiable if and only if there exist $(n-1)$-dimensional $C^1$ submanifolds $M_1, M_2, \dots$ of $\mathbb{R}^n$ such that $\mathcal{H}^{n-1} (E \setminus \cup_j M_j)=0$.
\end{enumerate}

An important property of rectifiable sets is the existence of approximate tangent planes at almost every point. Let us recall here the definition (see Definition 15.17 in \cite{Pertti}). Following the notation in \cite{Pertti}, 15.12, given a hyperplane $V \subset \mathbb{R}^n$, $a \in \mathbb{R}^n$ and $0 < s < 1$ we let
\begin{equation}\label{XaVs}
X(a,V,s)=\{ x \in \mathbb{R}^n: d_E(x-a,V) < s |x-a| \}.
\end{equation}
Given $A \subset \mathbb{R}^n$, we say that $V$ is an approximate tangent hyperplane for $A$ at $a$ if $\Theta^{* n-1}(A,a)>0$ and for all $0<s<1$,
\begin{equation*}
\lim_{r \rightarrow 0} r^{1-n} \mathcal{H}^{n-1}(A \cap B_E(a,r) \setminus X(a,V,s))=0.
\end{equation*}
Here $\Theta^{* n-1}(A,a)$ denotes the upper $(n-1)$-density of $A$ at $a$, defined (see Definition 6.1 in \cite{Pertti}) as $$\limsup_{r \rightarrow 0}(2r)^{1-n} \mathcal{H}^{n-1}(A \cap B_E(a,r)).$$
The set of all approximate tangent hyperplanes of $A$ at $a$ is denoted by $\mbox{ap Tan}(A,a)$. The following holds (see Theorem 15.19 in \cite{Pertti}).

\begin{theorem}
Let $E \subset \mathbb{R}^n$ be an $\mathcal{H}^{n-1}$ measurable set with $\mathcal{H}^{n-1}(E) < \infty$. Then $E$ is $(n-1)$-rectifiable if and only if for $\mathcal{H}^{n-1}$ almost every $p \in E$ there is a unique approximate tangent hyperplane for $E$ at $p$.
\end{theorem}

\begin{defn}
Given an $(n-1)$-rectifiable set $E \subset \mathbb{R}^n$, we say that $K \subset \mathbb{R}^n$ is an \textit{$E $-Nikodym set} if $\mathcal{L}^n(K)=0$ and for $\mathcal{H}^{n-1}$ almost every $p \in E$ there exists $e_p \in S^{n-1}$ such that $e_p \notin \mbox{ap Tan}(E,p)$ and $L_p(e_p) \cap K$ contains a unit segment, where $L_p(e_p)=\{ te_p+p : t \ge 0\}$.
\end{defn}

We will prove the following.

\begin{theorem}\label{ENik}
Let $K \subset \mathbb{R}^n$ be an $E$-Nikodym set. Then $\dim K \ge \frac{n+2}{2}$.
\end{theorem}

To prove the theorem, we will reduce to show the following lemma. Here we let $V$ be an $(n-1)$-dimensional linear subspace of $\mathbb{R}^n$, $g: V \rightarrow V^{\bot}$ be a Lipschitz map with Lipschitz constant $L>0$ (that is, $|g(x)-g(y)| \le L|x-y|$ for every $x,y \in V$) and $G$ be its graph. For $e \in S^{n-1}$ let $\measuredangle (e, V^\bot)$ denote as before the angle between a line in direction $e$ and $V^\bot$.

\begin{lem}\label{NL}
Let $G$ be the graph of a Lipschitz map $g:V \rightarrow V^\bot$ with Lipschitz constant $L$.
Let $N \subset \mathbb{R}^n$ be such that there exists $0 \le \theta_L<\arctan (1/L)$ and for every $p \in A \subset G$, where $\mathcal{H}^{n-1}(A)>0$, there exists $e_p \in S^{n-1}$ such that $\measuredangle (e_p, V^\bot) < \theta_L $ and $L_p(e_p) \cap N$ contains a unit segment. Then $ \dim N \ge \frac{n+2}{2}$.
\end{lem}

Let us first see how the lemma implies Theorem \ref{ENik}.

\begin{proof} (Lemma \ref{NL} $\Rightarrow$ Theorem \ref{ENik})\\
Let  $K \subset \mathbb{R}^n$ be an $E$-Nikodym set. Then for $\mathcal{H}^{n-1}$ almost every $p \in E$ there exists $e_p \in S^{n-1}$ such that $e_p \notin \mbox{ap Tan}(E,p)$ and $L_p(e_p) \cap K$ contains a unit segment. For $j=1, 2, \dots$, let $E_j=\{p \in E: \measuredangle (e_p,( \mbox{ap Tan}(E,p))^\bot) < \pi/2- \frac{1}{j} \}$. Then there exists $k$ such that $\mathcal{H}^{n-1} (E_k) >0$. 

Since $E_k$ is a subset of $E$, it is $(n-1)$-rectifiable. Hence by one of the definitions that we have seen there exists an $(n-1)$-dimensional $C^1$ submanifold $M$ of $\mathbb{R}^n$ such that $\mathcal{H}^{n-1}(E_k \cap M)>0$. It follows from Lemma 15.18 in \cite{Pertti} that for $\mathcal{H}^{n-1}$ almost every $p \in E_k \cap M$ we have $\mbox{ap Tan}(E_k,p)= \mbox{Tan}(M,p)$, where $ \mbox{Tan}(M,p)$ is the tangent hyperplane to $M$ at $p$.

Since $M$ is a $C^1$ manifold, for every point $p \in M \cap E_k$ there are a neighbourhood $U \subset M \cap E_k$ of $p$ and a $C^1$ function $f: W \rightarrow U$ such that $f(W)=U$, where $W \subset \mathbb{R}^{n-1}$ is open. We can assume that $Df$ is uniformly continuous on $W$. Let $0<\epsilon < 1/(2\tan(1/k))$. For every $x \in W$ let $\delta_x>0$ be such that for every $y \in W$ such that $|x-y| < \delta_x$ we have
\begin{align}\label{fxy}
f(y)-f(x)&=Df(x)(y-x)+\epsilon(x)(y-x)\nonumber \\ 
&=Df(x_0)(y-x)+(Df(x)-Df(x_0))(y-x)+\epsilon(x)(y-x),
\end{align}
where $x_0 \in W$ and $|\epsilon(x)(y-x)| < \epsilon |y-x|$.
If we let for $j=1,2, \dots$, $W_j=\{ x \in W: \delta_x>\frac{1}{j} \}$, then there exists $l$ such that $\mathcal{H}^{n-1}(W_l)>0$. Let $0<\delta<\frac{1}{l}$ be so small that $|Df(x)-Df(y)| < \epsilon $ when $|x-y| < \delta$. Fix $x_0 \in W_l$ such that $\mathcal{H}^{n-1}(B_E(x_0, \delta/2) \cap W_l)>0$. 
Let $V= Df(x_0)(\mathbb{R}^{n-1})$ and let $h$ be the isometry such that $h(V)=\mathbb{R}^{n-1}$. Let $g=P_{V^\bot} \circ f \circ h: V \rightarrow V^\bot$, where $P_{V^\bot}$ denotes the orthogonal projection onto $V^\bot$. Then for every $z,w \in h^{-1}(B_E(x_0, \delta/2) \cap W_l) \subset V$, we have $z=h^{-1}(x)$ for some $x \in B_E(x_0, \delta/2) \cap W_l$ and $|h(z)-h(w)|=|z-w| \le \delta$, thus by \eqref{fxy}
\begin{align*}
|g(z)-g(w)|=& |P_{V^\bot}(f(h(z))-f(h(w)))|\\
=& | P_{V^\bot} (Df(x_0) (h(z)-h(w))+(Df(h(z))-Df(x_0))(h(z)-h(w)) \\ &+\epsilon(x)(h(z)-h(w)))|\\
\le & |Df(x)-Df(x_0)||z-w|+ \epsilon |z-w|\\
< & 2 \epsilon |z-w| < \frac{1}{\tan (1/k)} |z-w|,
\end{align*}
where we used $P_{V^\bot}(Df(x_0)(h(z)-h(w)))=0$ since $Df(x_0)(h(z)-h(w)) \in V$. Thus $g$ is Lipschitz with Lipschitz constant $L < 1/\tan(1/k)$. Hence there exists $E'_l \subset E_k \cap M$, $\mathcal{H}^{n-1}(E'_l)>0$, such that $E'_l$ is contained in a Lipschitz graph $G$ with Lipschitz constant $L $.

Then for $p \in E'_l$ we have $  \measuredangle (e_p, (\mbox{ap Tan}(E'_l,p))^\bot) =  \measuredangle (e_p, (\mbox{ Tan}(M,p))^\bot) < \pi/2-\frac{1}{k}$. On the other hand, since $E'_l$ is contained in a Lipschitz graph we have $\measuredangle ((\mbox{ Tan}(M,p))^\bot, V^\bot) \le \pi/2- \arctan(1/L)$. Hence $ \measuredangle (e_p, V^\bot) \le  \arctan(1/L)- 1/k < \arctan(1/L) $.

It follows that the subset $K' \subset K$ containing a segment in $L_p(e_p) \cap K'$ for every $p \in E'_l$ satisfies the assumptions of Lemma \ref{NL}. Hence $\dim
K \ge \dim K' \ge \frac{n+2}{2}$.
\end{proof}

We will now prove Lemma \ref{NL} by showing that in this setting all the Axioms 1-5 are satisfied. Let $X=\mathbb{R}^n$, $d=d'=d_E$, $\mu= \mathcal{L}^n$, $Q=n$. 
We can assume without loss of generality that $V$ is the $x_1, \dots, x_{n-1}$-hyperplane and that $g(x) \ge 0$ for every $x$.

Let $Z=G$ and $Y\subset A \subset G$ be a compact set such that $0 < \mathcal{H}^{n-1}(Y) \le 1$, let $d_Z$ be the Euclidean metric on $G$ and $\nu$ be the $(n-1)$-dimensional Hausdorff measure $\mathcal{H}^{n-1} \big|_G$ restricted to $G$, thus $S=n-1$. Let $$\mathcal{A}= \{(e,t) \in S^{n-1} \times \mathbb{R}: \measuredangle(e,x_n  \mbox{-axis})< \theta_L, \frac{1}{4}< t \le M \},$$ where $0 <\theta_L < \arctan(1/L)$ and $M \in \mathbb{R}$. For $p \in Y$ and $(e,t) \in \mathcal{A}$ let, as in \eqref{Ipet}, $$F_p(e,t)=I_p(e,t)= \{ p+te+se: s \in [0,1]\}.$$ 
Since we will use the diameter estimate \eqref{diamNm}, we consider also here $0<\delta<\frac{1}{16(1+\tan \theta_L)}$. Let $T^\delta_p(e,t)$ be the $\delta$ neighbourhood of $I_p(e,t)$ in the Euclidean metric. Let also $\tilde{T}^{2 \delta}_p(e,t)$ be the $2 \delta$ neighbourhood of $\tilde{I}_p(e,t)=\{p+te+se: s \in [0,2]\}$. 

\textbf{Axiom 1}: Since the tubes are Euclidean, we have $\mathcal{L}^n(T^\delta_p(e,t)) \approx \mathcal{L}^n(\tilde{T}^{2\delta}_p(e,t)) \approx \delta^{n-1}$ and if $A \subset \tilde{T}^{2\delta}_p(e,t)$ then $\mathcal{L}^n(A) \lesssim \text{diam}_E(A) \delta^{n-1}$. Thus Axiom 1 holds with $T=n-1$. 

\textbf{Axiom 2} holds with $\theta=0$ since the tubes and balls are the usual Euclidean ones. 

\textbf{Axiom 3}: We want to show that there exists a constant $b=b_{L,n}$ such that for all $p=(x,g(x))$, $p'=(x',g(x')) \in Y$ and $(e,t)$, $(e',t') \in \mathcal{A}$ we have
\begin{equation}\label{NLdiam}
\text{diam}_E(\tilde{T}^{2\delta}_p(e,t) \cap \tilde{T}^{2\delta}_{p'}(e',t')) \le b \frac{\delta}{|p-p'|}.
\end{equation}
First observe that if $|p-p'| \le C_{L,n} \delta$ then $\delta/|p-p'| \gtrsim 1\gtrsim \text{diam}_E(\tilde{T}^{2 \delta}_p(e,t) \cap \tilde{T}^{2 \delta}_{p'}(e', t'))$, thus \eqref{NLdiam} holds (here the constants depend on $n$ and $L$). 
Hence we can assume that $|p-p'| > \frac{ 10 \delta}{\cos \theta_L}$.
If $g(x)=g(x')$ then we are in the same situation as in Section \ref{Nikmod} (indeed $p$ and $p'$ lie in the same hyperplane parallel to the $x_1, \dots, x_{n-1}$-hyperplane), thus \eqref{NLdiam} follows from \eqref{diamNm}.

\begin{figure}[H]
\centering
\includegraphics[scale=0.25]{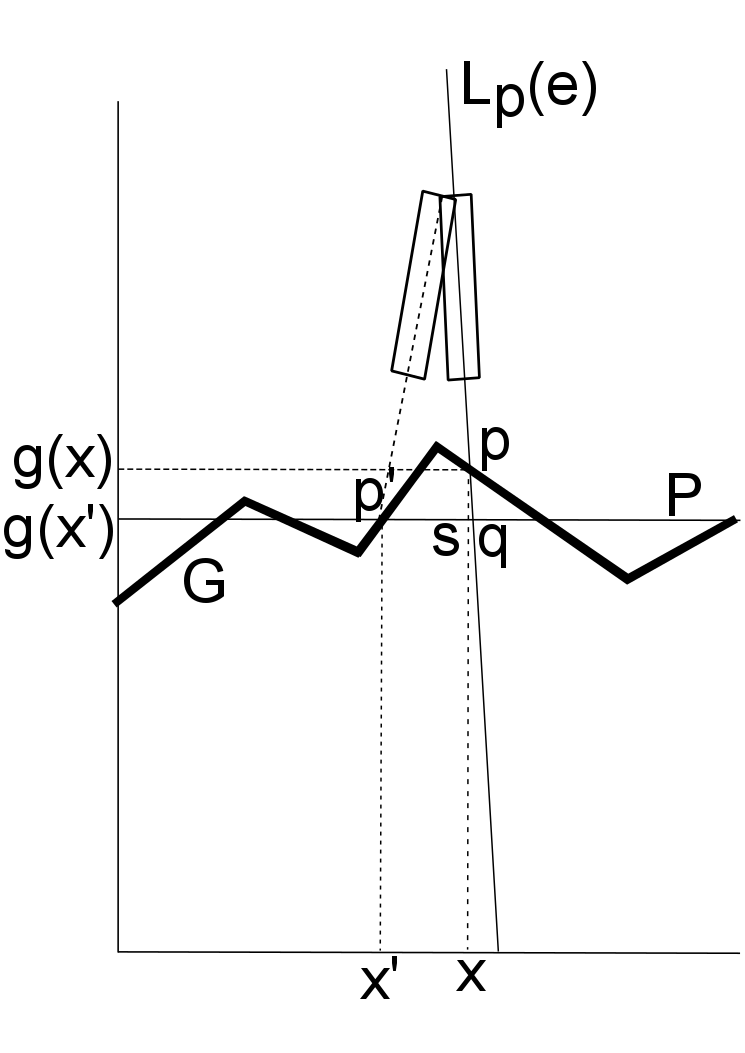}
    \caption{Lipschitz graph in $\mathbb{R}^2$ (proof of Axiom 3)}
    \label{fig2}
\end{figure}
Suppose that $g(x) > g(x')$. Let $P$ be the hyperplane parallel to the $x_1, \dots, x_{n-1}$-hyperplane and passing through $p'$. Let $L_p(e)$ be the line containing $I_p(e,t)$ and let $q= L_p(e) \cap P$ (see Figure \ref{fig2}). Let $s$ be the projection of $p$ onto $P$. Then $|s-p'|= |x-x'|$, $|p-s|=|g(x)-g(x')|$ and $|s-q|=  |p-s| \tan \theta_e$, where $\theta_e = \measuredangle(e,x_n  \mbox{-axis}) < \theta_L$. Thus we have
\begin{align*}
|q-p'|&\ge |s-p'| -|s-q| =  |x-x'|-|p-s| \tan \theta_e \\
& \ge |x-x'| -|g(x)-g(x')| \tan \theta_L \\
& \ge (1-L \tan \theta_L) |x-x'|\\
& \ge \frac{1-L \tan \theta_L}{\sqrt{1+L^2}} |p-p'|,
\end{align*}
where $1-L \tan \theta_L >0$ since $\tan \theta_L < 1/L$.
Moreover,
\begin{align*}
|q-p'|& \le |q-s|+|s-p'| \\
& \le |g(x)-g(x')| \tan \theta_L + |x-x'|\\
& \le (1+L \tan \theta_L) |x-x'|\\
& \le (1+L \tan \theta_L) |p-p'|.
\end{align*}
Hence $|p-p'| \approx |q-p'|$. But we know from Section \ref{Nikmod} (see Axiom 3) that $|q-p'| \approx |e-e'|$, thus
\begin{equation*}
\text{diam}_E(\tilde{T}^{2\delta}_p(e,t) \cap \tilde{T}^{2\delta}_{p'}(e',t')) \le b_n \frac{\delta}{|e-e'|} \le b_{n,L} \frac{\delta}{|p-p'|}.
\end{equation*}

\textbf{Axiom 4}: As in \eqref{a4NM}, we can see that if $p,p' \in Y$ are such that $|p-p'| \le \delta$, then $T^\delta_p(e,t) \subset \tilde{T}^{2 \delta}_{p'}(e,t)$. Hence Axiom 4 holds.

\textbf{Axiom 5}: Let $0<\delta, \beta, \gamma < 1$ and let $T_1, \dots, T_N$ be such that $T_j=\tilde{T}^{2\delta}_{p_j}(e_j,t_j)$, $|p_j-p_k|>\delta$ for every $j \ne k$. Let $T=\tilde{T}^{2\delta}_p(e,t)$ be such that $T \cap T_j \ne \emptyset$ and $|p_j -p| \ge \beta/8$ for every $j=1, \dots, N$. Then for all $j=1, \dots, N$,
\begin{eqnarray}\label{a5NL}
\begin{split}
& \# \mathcal{I}_j =\# \{i: |p_i-p_j| \le \beta, T_i \cap T_j \ne \emptyset, d_E(T_i \cap T_j, T_j \cap T) \ge \gamma \}  \\
&\le C \delta^{-1} \beta \gamma^{2-n}.
\end{split}
\end{eqnarray}

\begin{proof}
As in the proof of Lemma \ref{ax5Nm}, we can assume that $\delta$ is much smaller than $\beta$ and $\gamma$. Let $p=(x,g(x))$, $p_j=(x_j,g(x_j))$. We can also assume that $g(x_j)> g(x)$. 
Since $|p-p_j|\ge \beta/8 >>\delta$, we have seen above in Axiom 3 that we have $|p-p_j| \approx |q-p_j|$, where $q=L_p(e) \cap P$ and $P$ is the hyperplane parallel to the $x_1, \dots, x_{n-1}$-hyperplane and passing through $p_j$. Hence we are in the same situation as in Lemma \ref{ax5Nm} and \eqref{a5NL} follows from \eqref{ax5V}.
\end{proof}

Since all the Axioms 1-5 are satisfied, the maximal function
\begin{equation*}
f^d_\delta(p)= \sup_{(e,t) \in \mathcal{A}} \frac{1}{\mathcal{L}^n(T^\delta_p(e,t))} \int_{T^\delta_p(e,t)} |f| d \mathcal{L}^n
\end{equation*}
satisfies by Theorem \ref{Wolff}
\begin{equation}\label{WNL}
||f^d_\delta||_{L^n(Y)} \le C_\epsilon \delta^{\frac{2-n}{2n}-\epsilon} ||f||_{L^n(\mathbb{R}^n)}
\end{equation}
for every $\epsilon >0$ and every $f  \in L^n(Y)$.

Let $N \subset \mathbb{R}^n$ be such that for every $p \in A$ there exists $e_p \in S^{n-1}$ such that $\measuredangle(e_p, V^\bot) < \theta_L$ and $L_p(e_p) \cap N$ contains a unit segment. Then, in particular, for every $p \in Y \subset A$ we have $I_p(e_p,t_p) \subset N$ for some $t_p \ge 1/2$, where here $I_p(e_p,t_p)$ has length $1/2$.
Hence \eqref{WNL} implies the lower bound $\frac{n+2}{2}$ for the Hausdorff dimension of $N$ and proves Lemma \ref{NL}.

\begin{rem}(Sets containing a segment through every point of a purely unrectifiable set)

If $B \subset \mathbb{R}^n$, $\mathcal{H}^{n-1}(B)>0$, is purely $(n-1)$-unrectifiable and $K \subset \mathbb{R}^n$ is such that for every $p \in B$ there exists a segment $\{p+te_p: 0 \le t \le 1\} \subset K$ for some $e_p \in S^{n-1}$, then we cannot use the axiomatic system to obtain lower bounds for the dimension of $K$.
Indeed, we will show that we cannot find a diameter estimate as in Axiom 3 (where the tubes are Euclidean tubes). This is due to the geometric properties of purely unrectifiable sets, which are rather scattered.

Recall that a set $B \subset \mathbb{R}^n$ is called purely $(n-1)$-unrectifiable if $\mathcal{H}^{n-1}(B \cap F)=0$ for every $(n-1)$-rectifiable set $F$.
Fix some direction $e \in S^{n-1}$ and let $L$ be the line through the origin with direction $e$. We can define as in \eqref{XaVs} for $a \in \mathbb{R}^n$, $0 < s < 1$ and $0 < r < \infty$ 
\begin{equation*}
X(a,L,s)=\{ x \in \mathbb{R}^n: d_E(x-a,L) < s |x-a| \}
\end{equation*}
and
\begin{equation*}
X(a,r,L,s)= X(a,L,s) \cap B_E(a,r).
\end{equation*}
By Lemma 15.13 in \cite{Pertti}, since $B$ is purely $(n-1)$-unrectifiable, for every  $0 < s <1$ there exists $p \in B$ such that
\begin{equation*}
B \cap X(p,1,L,s) \neq \emptyset.
\end{equation*}
Let $s$ be much smaller than $1$ and let $p' \in B \cap X(p,1,L,s)$.
Letting $\delta = s |p-p'|$, we have $p' \in \tilde{T}^{2\delta}_p(e)$. It follows that $\text{diam}_E(\tilde{T}^{2\delta}_p(e) \cap \tilde{T}^{2\delta}_{p'}(e)) \approx 1$.  If we had a diameter estimate of the form of Axiom 3 then we would have $\text{diam}_E(\tilde{T}^{2\delta}_p(e) \cap \tilde{T}^{2\delta}_{p'}(e)) \le b \delta/|p-p'| = b  s$, which is not possible since $s$ is much smaller than $1$.
\end{rem}

\section{Curved Kakeya and Nikodym sets}

Bourgain \cite{Bourg_curves} and Wisewell (\cite{Wisewell}, \cite{Wisewell3}) have studied the case of curved Kakeya and Nikodym sets, that is when $I_u(a)$ is a curved arc from some specific family. We will recall here briefly the setting to see that Wisewell's results follow from Theorems \ref{Bourgain} and \ref{Wolff} since the axioms are satisfied.

The family of curves they consider arises from H\"ormander's conjecture regarding certain oscillatory integral operators. For $x=(x',x_n) \in \mathbb{R}^n$ (with $x' \in \mathbb{R}^{n-1}$) and $y \in \mathbb{R}^{n-1}$, $h(x,y)$ is some smooth cut-off and $\phi(x,y)$ is a smooth function on the support of $h$ that satisfies the following properties:
\begin{itemize}
\item[(i)] the rank of the matrix $\frac{\partial^2 \phi}{\partial x \partial y}(x,y)$ is $n-1$;
\item[(ii)] for all $\theta \in S^{n-1}$ the map $y \mapsto \langle \theta , \frac{\partial \phi }{\partial x} (x,y) \rangle$ has only non degenerate critical points.
\end{itemize}
These imply that $\phi $ can be written as
\begin{equation*}
\phi(x,y)=y^t x'+ x_n y^t Ay + O(|x_n| |y|^3+|x|^2 |y|^2),
\end{equation*}
where $A$ is an invertible $(n-1) \times (n-1)$ matrix. To prove Bourgain's lower bound Wisewell considers functions $\phi$ for which the higher order terms depend only on $x_n$ and not on $x'$. These can be written as
\begin{equation}\label{phiM}
\phi(x,y)= y^t M(x_n)x'+  \tilde{\phi}(x_n,y),
\end{equation}
where $M: \mathbb{R} \rightarrow GL(n-1,\mathbb{R})$ is a matrix-valued function.

Let $X=\mathbb{R}^n$, $d=d'$ be the Euclidean metric, $\mu=\mathcal{L}^n$. Let $Z=Y= \mathcal{A}$ be a certain ball in $\mathbb{R}^{n-1}$ whose radius depends only on $\phi$ (Wisewell in Section 2 in \cite{Wisewell} explains how to find it). We will denote it by $B$. Let $d_Z$ be the Euclidean metric in $\mathbb{R}^{n-1}$ (restricted to $B$) and $\nu= \mathcal{L}^{n-1}$. Thus $Q=n$ and $S=n-1$.\\
Here $F_u(a)=I_u(a)$ is defined for $a,u \in B$ as
\begin{equation*}
I_u(a)= \{ x \in \mathbb{R}^n: \nabla_u \phi(x,u)=a, (x,u) \in supp (h) \},
\end{equation*}
which is a smooth curve by condition (i) and the implicit function theorem. The tube $T^\delta_u(a)$ is defined by
\begin{equation*}
T^\delta_u(a)= \{ x \in \mathbb{R}^n:| \nabla_u \phi(x,u)-a|<\delta, (x,u) \in supp (h) \},
\end{equation*}
thus $\mu(T^\delta_u(a)) \approx \delta^{n-1}$ and if $A \subset T^\delta_u(a)$ then $\mu(A) \lesssim \text{diam}_E(A) \delta^{n-1}$ (\textbf{Axiom 1} holds with $T=n-1$). Here $\mu_{u,a}$ is the $1$-dimensional Hausdorff measure on $I_u(a)$.

In the straight line case the Kakeya and Nikodym problems are equivalent at least at the level of the maximal functions (see Theorem 22.16 in \cite{Mattila}), whereas we will see that in the curved case this is not true.

A \textit{curved Kakeya set} is a set $K \subset \mathbb{R}^n$ such that $\mathcal{L}^n(K)=0$ and for every $u \in B$ there exists $a \in B$ such that $I_u(a) \subset K$. A \textit{curved Nikodym set} is a set $N \subset \mathbb{R}^n$ such that $\mathcal{L}^n(N)=0$ and for every $a \in B$ there exists $u \in B$ such that $I_u(a) \subset N$.\\
Wisewell in \cite{Wisewell2} has proved that there exist such sets of measure zero. 

As in \eqref{fddelta}, we define the curved Kakeya maximal function as
\begin{equation*}
\mathcal{K}_\delta f(u) = \sup_{a \in B} \frac{1}{\mathcal{L}^n(T^\delta_u(a))} \int_{T^\delta_u(a)} |f| d \mathcal{L}^n
\end{equation*}
and the curved Nikodym maximal function as
\begin{equation*}
\mathcal{N}_\delta f(a) = \sup_{u \in B} \frac{1}{\mathcal{L}^n(T^\delta_u(a))} \int_{T^\delta_u(a)} |f| d \mathcal{L}^n.
\end{equation*}

\textbf{Axiom 2}: In the proof of Theorem 11 in \cite{Wisewell3} it is proved that Axiom 2 holds with $\theta=0$ (indeed, Theorem 11 is the corresponding Theorem \ref{bound} for the curved case).

\textbf{Axiom 3}: One can prove the diameter estimate of Axiom 3 (see Lemma 6 in \cite{Wisewell}). This estimate does not hold if the non-degeneracy criterion (ii) is dropped, since in this case two curves can essentially share a tangent, thus the intersection of the corresponding tubes can be larger.

\textbf{Axiom 4}: In \cite{Wisewell} Wisewell observes that since $\phi$ is smooth, if $|u-v|< \delta$ then $T^\delta_u(a) \subset T^{W\delta}_v(b)$, where $|a-b|< \delta$ and $W$ is a constant depending only on $\phi$ (see also Lemma 7 in \cite{Wisewell3}). Thus Axiom 4 holds.

Hence all the Axioms 1-4 are satisfied and Bourgain's method, as shown in \cite{Wisewell} (Theorem 7), gives the lower bound $\frac{n+1}{2}$ for the Hausdorff dimension of curved Kakeya and Nikodym sets for any phase function $\phi$ of the form \eqref{phiM}.

Bourgain in dimension $n=3$ has showed some negative results for certain families of curves, whose associated Kakeya sets cannot have dimension greater than $\frac{n+1}{2}=2$. In particular, these are related to phase functions $\phi $ such that $\frac{\partial^2}{\partial y^2} (\frac{\partial^2 \phi}{\partial x_3^2})$ at $x=y=0$ is not a multiple of $\frac{\partial^2}{\partial y^2} (\frac{\partial \phi}{\partial x_3})$ at $x=y=0$. For example, the curves given by
\begin{equation*}
I_u(a)=\{ (a_1-x_3 u_2-x_3^2 u_1, a_2-x_3 u_1, x_3)\}
\end{equation*}
lie in the surface $x_1=x_2 x_3$ if we choose $a_1=0$, $a_2=-u_2$. Thus if $K$ is a Kakeya set that for every $u \in B$ contains a curve $I_u((0,-u_2))$, then $K$ has Hausdorff dimension $2$.\\
The failure is caused by the presence in $\phi$ of terms non linear in $x$. This is why Wisewell considers only parabolic curves of the form
\begin{equation*}
I_u(a)=\{ (a-tu-t^2Cu,t): t \in [0,1]\},
\end{equation*}
where $C$ is a $(n-1) \times (n-1)$ real matrix, when looking for an improvement of the above lower bound. However, also in this case there are some negative results.
In \cite{Wisewell} (Theorem 10) it is proved that if $C$ is not a multiple of the identity then we cannot have the optimal Kakeya maximal inequality.
This failure does not concern the Nikodym maximal function.

\textbf{Axiom 5}: Wisewell shows that when $C^2=0$ Axiom 5 holds (for the Nikodym case) with $\lambda=1$ and $\alpha=n-2$ (see Claim in the proof of Lemma 13 in \cite{Wisewell}). 

Thus Wolff's argument gives the lower bound $\frac{n+2}{2}$ for the Hausdorff dimension of curved Nikodym sets (for this class of curves).

\section{Restricted Kakeya sets}

Given a subset $A \subset S^{n-1}$ one can study the Kakeya and Nikodym maximal functions restricting to tubes with directions in $A$. In \cite{Cordoba} Cordoba studied the Nikodym maximal function in the plane restricting to tubes whose slopes are in the set $\{1/N,2/N, \dots , 1\}$. Bateman \cite{Bateman} gave a characterization for the set of directions $A$ for which the Nikodym maximal function in the plane is bounded, whereas Kroc and Pramanik \cite{KrocPram2} characterized these sets of directions in all dimensions.

If $A \subset S^{n-1}$ we say that a set $B \subset  \mathbb{R}^n$ is an $A$\textit{-Kakeya set} if $\mathcal{L}^n(B)=0$ and for every $e \in A$ there exists $a \in \mathbb{R}^n$ such that $I_e(a) \subset B$, where $I_e(a)$ is the unit segment with direction $e$ and midpoint $a$.\\
Mitsis \cite{Mitsis} proved that if $B \subset \mathbb{R}^2$ is an $A$-Kakeya set, where $A \subset S^1$, then $\dim B \ge \dim A +1$ (and the estimate is sharp). 

Here we consider $n \ge 3$ and prove lower bounds for $A$-Kakeya sets using the axiomatic setting, where $A$ is an $S$-regular subset of the sphere, $S \ge 1$. This means that there exists a Borel measure $\nu$ on $A$ and two constants $0 < c_1 \le c_2 < \infty$ such that
\begin{equation}\label{Sreg}
c_1 r^S \le \nu(B_E(e,r) \cap A) \le c_2 r^S
\end{equation}
for every $e \in A$ and $0<r \le \text{diam}_E A$. We prove the following.

\begin{theorem}\label{AKak}
Let $n \ge 3$ and let $A \subset S^{n-1}$ be an $S$-regular set, $S \ge 1$. Then the Hausdorff dimension of any $A$-Kakeya set in $\mathbb{R}^n$ is $\ge \frac{S+3}{2}$.
\end{theorem}

Let $X= \mathbb{R}^n$, $d=d'$ be the Euclidean metric, $\mu= \mathcal{L}^n$ thus $Q=n$. Moreover, we let $Z=Y=A$ ($S$-regular subset of $S^{n-1}$), $d_Z$ be the Euclidean metric restricted to $A$, $\nu$ be a measure on $A$ as in \eqref{Sreg}. We let $\mathcal{A}=\mathbb{R}^n$.\\
For $e \in A$ and $a \in \mathbb{R}^n$, $F_e(a)=I_e(a)$ is the segment with direction $e$, midpoint $a$ and Euclidean length 1. The measure $\mu_{e,a}$ is the $1$-dimensional Hausdorff measure restricted to $I_e(a)$.
The tube $T^\delta_{e}(a)$ is the $\delta$ neighbourhood of $I_e(a)$ in the Euclidean metric, thus $T=n-1$. The Axioms 1-4 are satisfied (with $\theta=0 $) since the tubes are the usual Euclidean tubes.

From Bourgain's method we get the lower bound $\frac{2Q-2T+S}{2}= \frac{S+2}{2}$ for the Hausdorff dimension of any $A$-Kakeya set.

\textbf{Axiom 5} holds with $\lambda=1$ and $\alpha=S-1$. This can be proved as in the usual Euclidean case, see Lemma 23.3 in \cite{Mattila}. Indeed, we only need to modify the end of the proof, using the $S$-regularity of $Y$. We explain it briefly here.

\begin{lem}
Let $0<\delta, \beta, \gamma < 1$ and let $T_1, \dots, T_N$ be such that $T_j=\tilde{T}^{2\delta}_{e_j}(a_j)$, $e_j \in Y$, $|e_j-e_k|>\delta$ for every $j \ne k$. Let $T= \tilde{T}^{2\delta}_e(a)$ be such that $T \cap T_j \ne \emptyset$ and $|e_j -e| \ge \beta/8$ for every $j=1, \dots, N$. Then for all $j=1, \dots, N$,
\begin{eqnarray}\label{S5}
\begin{split}
& \# \mathcal{I}_j =\# \{i: |e_i-e_j| \le \beta, T_i \cap T_j \ne \emptyset, d_E(T_i \cap T_j, T_j \cap T) \ge \gamma \}  \\
&\le C \delta^{-1} \beta \gamma^{1-S}.
\end{split}
\end{eqnarray}
\end{lem}

\begin{proof}
We can assume that $\beta > \frac{\delta}{\gamma}$. Indeed, if $\beta \le \frac{\delta}{\gamma}$ then 
\begin{equation*}
\# \{ i: |e_i -e_j| \le \beta \} \lesssim \frac{\beta^S}{\delta^S} \lesssim \beta \left( \frac{\delta}{\gamma} \right)^{S-1} \delta^{-S}= \beta \delta^{-1} \gamma^{1-S},
\end{equation*}
thus \eqref{S5} holds. As in the proof of  Lemma 23.3 in \cite{Mattila} we can show that for $i \in \mathcal{I}_j$
\begin{equation*}
e_i \in B_E(e_j, \beta) \cap \{ x \in A: |x_k| \le c \frac{\delta}{\gamma}, k=3, \dots, n\},
\end{equation*}
where $c$ is a constant depending only on $n$. The number of balls of radius $\delta/\gamma$ needed to cover this set is $\lesssim \beta \gamma \delta^{-1}$, thus this set has $\nu$ measure $\lesssim \beta\gamma \delta^{-1} (\delta/\gamma)^{S}= \beta \delta^{S-1} \gamma^{1-S}$. It follows that it can contain $\lesssim \beta \delta^{-1} \gamma^{1-S}$ points of $Y$ that are $\delta$-separated.
\end{proof}

Thus Wolff's method yields the lower bound $\frac{S+3}{2}$ for the Hausdorff dimension of $A$-Kakeya sets (which proves Theorem \ref{AKak}).

\section{Furstenberg type sets}\label{Furstenberg sets}

\subsection{Furstenberg sets}\label{Furst}

Given $0<s\le1$, an $s$-Furstenberg set is a compact set $F \subset \mathbb{R}^n$ such that for every $e \in S^{n-1}$ there is a line $l_e$ with direction $e$ such that $\dim (F \cap l_e)  \ge s$. Wolff \cite{Wolff2} has proved that when $n=2$ any $s$-Furstenberg set $F$ satisfies $\dim F \ge \max \{2s, s+1/2\}$. Moreover, there is such a set $F$ with $\dim F = 3s/2+1/2$. In \cite{BourgainF} Bourgain has improved the lower bound when $s=1/2$, showing that $\dim F >1+c$, where $c>0$ is some absolute constant.

We can show, using the axiomatic method, the lower bound $2 s$ (this is essentially the same way in which Wolff found it).
Moreover, we can find a lower bound for the Hausdorff dimension of Furstenberg sets in $\mathbb{R}^n$, $n \ge 3$. In $\mathbb{R}^n$ the conjectural lower bound is $\frac{n-1}{2}+s \frac{n+1}{2}$ (see Conjecture 2.6 in \cite{Zhang}, where Zhang considers the Furstenberg problem in higher dimensions and proves a lower bound for the finite field problem when $s=1/2$). More precisely, we prove the following.

\begin{theorem}\label{Furstdim}
Let $F \subset \mathbb{R}^n$ be an $s$-Furstenberg set.
\begin{itemize}
\item[i)] If $n \le 8$, then 
\begin{equation*}
\dim F \ge \max \left \{\frac{(2s-1)n+2}{2}, s\frac{4n+3}{7} \right \}.
\end{equation*}
In particular, it is $\ge s\frac{4n+3}{7} $ for $s  < \frac{7(n-2)}{6(n-1)}$ and $\ge \frac{(2s-1)n+2}{2}$ for $s \ge \frac{7(n-2)}{6(n-1)}$.

\item[ii)] If $n \ge 9$, then $\dim F \ge s\frac{4n+3}{7}$.
\end{itemize}
\end{theorem}

To prove the Theorem, we show that all the Axioms 1-5 are satisfied and apply Wolff's method. Moreover, we use Katz and Tao's estimate for the classical Kakeya maximal function.

Here is the setting. Let $F \subset \mathbb{R}^n$ be an $s$-Furstenberg set.
Let $X= \mathcal{A}= \mathbb{R}^n$, $d=d'$ be the Euclidean metric, $\mu=\mathcal{L}^n$ and $Q=n$. Since $F $ is compact, there exists $R>0$ such that $F \subset B_E(0,R/2)$. For every $e \in Z= S^{n-1}$ and $a \in \mathbb{R}^n$ we define $I_e(a)$ as the segment with $\text{diam}_E(I_e(a)) =R$, direction $e$ and midpoint $a$. Then we let $F_e(a)= I_e(a) \cap F$. 

For every $e \in Z= S^{n-1}$ there exists $a=a_e \in \mathcal{A}$ such that $\dim F_e(a) \ge s$. Thus for every $t <s$ there exists a Borel measure $\mu_{e,a}$ such that $\mu_{e,a}(F_e(a))=1$ and $\mu_{e,a}(F_e(a) \cap B_E(x,r)) \le C_{e,a} r^t$ for every $x \in F_e(a)$ and every $0<r<R$. Observe that since $a=a_e$ depends on $e \in S^{n-1}$, actually $C_{e,a}=C_e$ depends only on $e$. We want to choose $Y \subset S^{n-1}$ such that $C_e \le C$ for every $e \in Y$ for some constant $C$. For $k=1, 2, \dots$, let 
\begin{equation*}
S_k=\{ e \in S^{n-1}: C_e \le k \}.
\end{equation*}
Since $S^{n-1}= \cup_{k=1}^{\infty} S_k$, there exists $k$ such that $\sigma^{n-1}(S_k) >0$. We let then $Y\subset S_k$ be compact and such that $\sigma^{n-1}(Y)>0$, $\nu=\sigma^{n-1}$, $d_Z$ be the Euclidean metric on $S^{n-1}$. Hence $0<\nu(Y) \le 1$ and $S=n-1$.\\
Let $T^\delta_e(a)$ be the $\delta$ neighbourhood of $I_e(a)$ in the Euclidean metric. 

\textbf{Axiom 1} holds with $T=n-1$ since $\mathcal{L}^n(T^\delta_e(a)) \approx \delta^{n-1}$ for every $e \in Y$ and every $a \in \mathbb{R}^n$ and if $A \subset T^\delta_e(a)$ then $\mathcal{L}^n(A) \lesssim \text{diam}_E(A) \delta^{n-1}$.

\textbf{Axiom 2}: By Remark \ref{muut} Axiom 2 holds with $\theta= 1-t$. Indeed, with the choice made above, for every $e \in Y$ and every $x \in F_e(a)$ we have $\mu_{e,a}(F_e(a) \cap B_E(x,r)) \le C_e r^t \le k r^t$, thus Axiom 2 is satisfied with $\theta=Q-t-T=n-t-(n-1)=1-t$.

\textbf{Axioms 3} and \textbf{4} hold since the tubes and balls are Euclidean.

Hence Bourgain's method yields the lower bound $t \frac{n+1}{2}$ for every $t<s$, which implies the lower bound $s \frac{n+1}{2}$ for the Hausdorff dimension of $s$-Furstenberg sets.

\textbf{Axiom 5} holds with $\lambda=1$ and $\alpha=n-2$ since the tubes are the usual Euclidean tubes used for Kakeya sets. 

Thus Wolff's method (Theorem \ref{Wolff}) yields the lower bound $\frac{(2s-1)n+2}{2}$. When $n=2$ this gives $2s$, in general this improves the previous bound $s \frac{n+1}{2}$ when $s > \frac{n-2}{n-1}$.

\begin{rem}\label{Furstb}
Here the maximal function $f^d_\delta$ is the usual Kakeya maximal function since the tubes are the Euclidean tubes and $\mu$ is the Lebesgue measure. Katz and Tao in \cite{Katz&Tao2} proved the estimate 
\begin{equation}\label{fdp}
||f^d_\delta||_{L^{\frac{4n+3}{4}}(S^{n-1})} \lesssim \delta^{\frac{3-3n}{4n+3}-\epsilon} ||f||_{L^{\frac{4n+3}{7}}(\mathbb{R}^n)}
\end{equation}
for every $\epsilon>0$ and every $f \in L^{\frac{4n+3}{7}}(\mathbb{R}^n)$. This implies by Corollary \ref{Lpbound} the lower bound $s\frac{4n+3}{7}$ for the Hausdorff dimension of $s$-Furstenberg sets (in Corollary \ref{Lpbound} we need the same $p$ on both sides of \eqref{fdp} but we can take $p=\frac{4n+3}{7}$ since it is smaller than $\frac{4n+3}{4}$). This improves the lower bound $s \frac{n+1}{2}$ for any value of $n \ge 2$ and $0<s\le 1$. Moreover, it improves the lower bound $\frac{(2s-1)n+2}{2}$ for $n \le 8$ when $s < \frac{7(n-2)}{6(n-1)}$ and for $n \ge 9$ for any value of $s$. This estimate completes the proof of Theorem \ref{Furstdim}.
\end{rem}

\subsection{Sets containing a copy of an $s$-regular set in every direction}\label{specialFurst}

Let us now make a stronger assumption, that is consider $F\subset \mathbb{R}^n$ ($n \ge 2$) that contains in every direction a rotated and translated copy of an Ahlfors $s$-regular compact set $K \subset \mathbb{R}$ (we can assume $K \subset x_1$-axis). More precisely, let as above $X= \mathcal{A}= \mathbb{R}^n$, $d=d'$ be the Euclidean metric, $\mu=\mathcal{L}^n$ and $Q=n$. Let $Z=Y=S^{n-1}$, $d_Z$ be the Euclidean metric on $S^{n-1}$, $\nu=\sigma^{n-1}$, thus $S=n-1$. For $e \in S^{n-1}$ and $a \in \mathbb{R}^n$ let $F_e(a)=I_e(a)$ be a copy of $K$ rotated and translated so that it is contained in the segment with direction $e$, midpoint $a$ (and length $\text{diam}_E K$). Since $F_e(a)$ is $s$-regular, there exists a measure $\mu_{e,a}$ such that $\mu_{e,a}(F_e(a))=1$ and $r^s/C \le \mu_{e,a}(F_e(a) \cap B_E(x,r)) \le C r^s$ for every $x \in F_e(a)$ and $0<r < \text{diam}_E K$ (here $C$ does not depend on $e$ since the $F_e(a)$'s are all copies of the same set). 

We say that a set $F \subset \mathbb{R}^n$ is a $K$-\textit{Furstenberg set} if for every $e \in S^{n-1}$ there exists $a \in \mathbb{R}^n$ such that $F_e(a) \subset F$. We will prove the following.

\begin{theorem}\label{Furstmod}
Let $K \subset \mathbb{R}$ be an $s$-regular compact set and let $F \subset \mathbb{R}^n$ be a $K$-Furstenberg set.
\begin{itemize}
\item[i)] If $n  \le 8$, then $\dim F \ge 2 s+ \frac{n-2}{2}$.

\item[ii)] If $n \ge 9$, then 
\begin{equation*}
\dim F \ge \max \left \{ 2 s+\frac{n-2}{2}, s \frac{4n+3}{7} \right \}.
\end{equation*}
In particular, it is $\ge 2 s+\frac{n-2}{2}$ for $s \le \frac{7(n-2)}{2(4n-11)}$ and it is $\ge  s \frac{4n+3}{7}$ when $ s >\frac{7(n-2)}{2(4n-11)}$.
\end{itemize}
\end{theorem}

To prove the Theorem, we will show that a modified version of Axiom 1 holds and that all the other Axioms are satisfied. Going through the proofs of Lemma \ref{hairbrush} and Theorem \ref{Wolff}, one can see that they can easily be modified to handle this case and we will see what dimension estimate they give.

\textbf{Axiom 1}: If $T^\delta_e(a)$ is the $\delta$ neighbourhood of $F_e(a)$ then $\mathcal{L}^n(T^\delta_e(a))\approx \delta^{n-s}$. Indeed, we need essentially $\delta^{-s}$ balls of radius $\delta$ to cover $T^\delta_e(a)$. On the other hand, if $A \subset T^\delta_e(a)$ then $\mathcal{L}^n (T^\delta_e(a) \cap A) \lesssim (\text{diam}_E(A))^s \delta^{n-s}$. Thus $T=n-s$, even if we do not have exactly Axiom 1 but a modified version of it. 

\textbf{Axiom 2}: By Remark \ref{muut} Axiom 2 holds with $\theta=0$ and $K=1$ ($K'$ is a constant depending only on $n$ and $s$) since $\mu_{e,a}(F_e(a) \cap B_E(x,r)) \approx r^s$ for every $x \in F_e(a)$ and thus $\theta=Q-s-T=n-s-(n-s)=0$.

\textbf{Axiom 3}: Since $T^\delta_e(a)$ is contained in the $\delta$ neighbourhood of the segment containing $F_e(a)$ (that is, in one of the usual Euclidean tubes), Axiom 3 holds.

\textbf{Axiom 4}: To see that Axiom 4 holds, write the points of $F_e(a)$ as $p=\sigma_e(x)+a$, where $x \in K \subset x_1$-axis and  $\sigma_e$ is the rotation that maps $(1,0, \dots, 0)$ to $e$. Suppose that $|e-e'| \le \delta$ and $q \in T^\delta_e(a)$. Then there exists $p=  \sigma_e(x)+a \in F_e(a)$ such that $|q-p| \le \delta$. On the other hand,
\begin{equation*}
|\sigma_e(x)-\sigma_{e'}(x)| \le C \delta,
\end{equation*}
where $C$ is a constant depending on $n$, $\text{diam}_E K$ and the distance from $K$ to the origin. Thus if we let $p'=\sigma_{e'}(x)+a$ we have $p' \in F_{e'}(a)$ and
\begin{equation*}
|q-p'| \le |q-p|+|p-p'| \le (1+C) \delta.
\end{equation*}
It follows that $q \in \tilde{T}^{(1+C)\delta}_{e'}(a)$, which is the $(1+C)\delta$ neighbourhood of $F_{e'}(a)$ (here we do not need to take longer tubes). Thus Axiom 4 holds with $W=1+C$. Observe that Axiom 4 does not necessarily hold if we only know that $F_e(a)$ is $s$-regular but we do not assume that the $F_e(a)$'s are (rotated and translated) copies of the same set.  

\textbf{Axiom 5}: Let us now see that Axiom 5 holds with $\lambda=1$ and $\alpha=n-2$. 

\begin{lem}\label{ax5Fs}
Let $0<\delta, \beta, \gamma <1$ and let $T=\tilde{T}^{(1+C)\delta}_e(a)$, $T_j=\tilde{T}^{(1+C)\delta}_{e_j}(a_j)$, $j=1, \dots, N$, such that $T \cap T_j \neq \emptyset$, $|e_j-e| \ge \beta/8$, $|e_j-e_k| > \delta$ for every $j \neq k$. Then for all $j=1, \dots, N$,
\begin{align}
\nonumber
& \# \mathcal{I}_j=\# \{ i: |e_i-e_j| \le \beta, T_i \cap T_j \neq \emptyset, d_E(T_i \cap T_j, T_j \cap T) \ge \gamma \}\\
\label{ax5s}
& \lesssim \beta \delta^{-1} \gamma^{2-n}. 
\end{align}
\end{lem}

\begin{proof}
We may assume that $\delta$ is much smaller than $\gamma$. Indeed, if $\delta \gtrsim \gamma$ then $\#\{i: |e_i-e_j| \le \beta \} \lesssim \beta^{n-1} \delta^{1-n} \lesssim \beta \delta^{-1} \gamma^{2-n}$ since $e_i$, $e_j$ are $\delta$ separated. Thus \eqref{ax5s} holds.

Observe that each tube $T_j$, $T$ is contained in one (and only one) Euclidean tube (the $(1+C) \delta$ neighbourhood of a segment of length $\text{diam}_E(K)$), thus counting how many tubes $T_i$ we can have will be the same as counting how many such Euclidean tubes there can be.

To see this, let $I^E$ be the segment of length $\text{diam}_E(K)$ (direction $e$ and midpoint $a$) containing $F_e(a)$ and let $T^E$ be its $(1+C)\delta$ neighbourhood. Then $T \subset T^E$. Similarly, let $T_j \subset T^E_j$ and for $i \in \mathcal{I}_j$ let $T_i \subset T^E_i$. 
Then we have the usual Euclidean tubes $T^E$, $T^E_j$, $T^E_i$ that satisfy $|e-e_j| \ge \beta/8$, $|e-e_i | \ge \beta/8$, $\delta < |e_i-e_j| \le \beta$, $T^E_j$ intersects $T^E$ and $T^E_i$ intersects both $T^E$ and $T^E_j$. 
Moreover, $d_E(T^E \cap T^E_j, T^E_j \cap T^E_i) \ge \gamma/2$. Indeed,
\begin{align}\label{diE}
d_E(T^E \cap T^E_j, T^E_j \cap T^E_i) \ge & d_E(T_i \cap T_j, T_j \cap T) \nonumber\\
&- \text{diam}_E(T^E_i \cap T^E_j)- \text{diam}_E(T^E \cap T^E_j).
\end{align}
By the diameter estimate for Euclidean tubes we have
\begin{equation}\label{dij}
\text{diam}_E(T^E_i \cap T^E_j) \le \frac{b \delta}{|e_i-e_j|},
\end{equation}
where $b$ is a constant depending only on $n$.
We can assume that $\beta> C \frac{\delta}{\gamma}$ and $|e_i-e_j| > C  \frac{\delta}{\gamma}$ for some sufficiently big constant $C=C_n$. Indeed, if $ \beta \le C \frac{\delta}{\gamma}$ then since the points are $\delta$-separated we have 
\begin{eqnarray*}
\begin{split}
\# \left\{ i: |e_i-e_j | \le \beta \right\}& \le \frac{\tilde{C}_0}{\tilde{c}_0} \beta^{n-1}\delta^{1-n} = \frac{\tilde{C}_0}{\tilde{c}_0} \beta \beta^{n-2}\delta^{1-n}\\& \le \frac{\tilde{C}_0}{\tilde{c}_0}  \beta C^{n-2} \frac{\delta^{n-2}}{\gamma^{n-2}} \delta^{1-n} = \frac{\tilde{C}_0}{\tilde{c}_0} C^{n-2} \beta \gamma^{2-n} \delta^{-1},
\end{split}
\end{eqnarray*}
hence we would have the desired estimate (the constants $\tilde{C}_0$ and $\tilde{c}_0$ are those appearing in \eqref{BS}). Similarly we have
\begin{eqnarray*}
\begin{split}
\# \left\{ i: |e_i-e_j | \le C \frac{\delta}{\gamma} \right\} &\le \frac{\tilde{C}_0}{\tilde{c}_0} C^{n-1} \delta^{n-1} \gamma^{1-n} \delta^{1-n} = \frac{\tilde{C}_0}{\tilde{c}_0} C^{n-2} C \frac{\delta}{\gamma} \gamma^{2-n} \delta^{-1} \\& <  \frac{\tilde{C}_0}{\tilde{c}_0} C^{n-2} \beta \gamma^{2-n} \delta^{-1},
\end{split}
\end{eqnarray*}
where we used $\beta > C \frac{\delta}{\gamma}$.
Thus we have by \eqref{dij}
\begin{equation*}
\text{diam}_E(T^E_i \cap T^E_j)\le \frac{b \delta}{|e_i-e_j|} < \frac{b}{C}\gamma.
\end{equation*}
Similarly, since $|e-e_j| \ge \frac{\beta}{8} > C \frac{\delta}{ 8\gamma}$, we have
\begin{equation*}
\text{diam}_E(T^E \cap T^E_j) \le \frac{b \delta}{|e-e_j|} < \frac{8 b \gamma}{C}.
\end{equation*}
It follows from \eqref{diE} that $d_E(T^E \cap T^E_j, T^E_j \cap T^E_i) \ge \gamma (1-9b/C)$. If we choose $C> 18b$, then we have $d_E(T^E \cap T^E_j, T^E_j \cap T^E_i) \ge \gamma/2$. Thus by Axiom 5 for the Euclidean tubes we have
\begin{align*}
&  \# \{ i: |e_i-e_j| \le \beta, T^E_i \cap T^E_j \neq \emptyset, d_E(T^E_i \cap T^E_j, T^E_j \cap T^E) \ge \gamma/2 \}\\
& \lesssim \beta \delta^{-1} \gamma^{2-n}. 
\end{align*}
On the other hand, to each tube $T^E_i$ corresponds only one tube $T_i$ thus \eqref{ax5s} follows.
\end{proof}

Since the Axioms 2-5 are satisfied and we have a modified version of Axiom 1, we can get the following result. Here $f^d_\delta$ is defined as in \eqref{fddelta} by
\begin{equation*}
f^d_\delta(e)= \sup_{a \in \mathbb{R}^n} \frac{1}{\mathcal{L}^n(T^\delta_e(a))} \int_{T^\delta_e(a)} |f| d \mathcal{L}^n.
\end{equation*}

\begin{theorem}
Let $0< \delta <1$. Then for every $f \in L^{\frac{n-2+2s}{s}}(\mathbb{R}^n)$ and every $\epsilon>0$,
\begin{equation}\label{WM1}
||f^d_\delta||_{L^{\frac{n-2+2s}{s}}(S^{n-1})} \le C_{n,s,\epsilon} \delta^{-\frac{s(n+2-4s)}{2(n-2+2s)}-\epsilon}||f||_{L^{\frac{n-2+2s}{s}}(\mathbb{R}^n)}.
\end{equation}
\end{theorem}

\begin{proof}
This is obtained by modifying the proofs of Lemma \ref{hairbrush} and Theorem \ref{Wolff}. 

Indeed, in the proof of Lemma \ref{hairbrush} one needs to modify $\mu(T_j(l,m))$ in \eqref{Wolff2} since $\text{diam}_E(T_j(l,m)) \lesssim \delta 2^{m+l}$ implies in this case (by the modified Axiom 1) that $\mathcal{L}^n(T_j(l,m)) \lesssim 2^{(m+l)s} \delta^s \delta^{n-s}=2^{(m+l)s} \delta^n$. Moreover, $\text{diam}_E(T_i \cap T_j) \lesssim \delta 2^l$ implies $\mathcal{L}^n(T_i \cap T_j) \lesssim 2^{ls}  \delta^n$. Hence \eqref{Wolff3} becomes
\begin{align*}
\int_{T_j} \left( \sum_{i \in I(k,j,l,m)} \chi_{T_i} \right)^{p-1} d \mathcal{L}^n & \lesssim (\#I(k,j,l,m) 2^{ls} \delta^n)^{p-1} (2^{(m+l)s} \delta^n)^{2-p}\\
&\lesssim 2^{l(p(1-n)+n+s-1)+m(p(2-n-s)+n+2s-2)} \delta^{p(1-n)+2n-1}\\
& \lesssim \delta^{n-s} \delta^{n-1+s-p(n-1)}
\end{align*}
when $p \ge \frac{n-2+2s}{n-2+s} \ge \frac{n+s-1}{n-1}$ (since in this case the exponent of $2$ is non positive).
Thus Lemma \ref{hairbrush} has the following formulation. Let $T_1, \dots, T_N$ be an $(N, \delta)$-hairbrush. Then for every $\epsilon>0$ and every $\frac{n-2+2s}{n-2+s} \le p \le 2$,
\begin{equation}\label{WM}
\int \left( \sum_{j=1}^N \chi_{T_j} \right)^p d \mathcal{L}^n \lesssim C_{p,\epsilon} \delta^{n-s} N \delta^{n-1+s-p(n-1)-\epsilon}.
\end{equation}

In the proof of Theorem \ref{Wolff}, using \eqref{WM} with $p= \frac{n-2+2s}{n-2+s}$ one gets a different estimate for $S(H,H)$ in \eqref{SHH}, thus \eqref{SHH2} becomes
\begin{equation}\label{SHHs}
S(H,H) \lesssim \delta^{- \epsilon} \# I(B) \delta^{n-1} \left( \frac{2^{-k(n-1)} \delta^{\frac{2n-4+6s-4sn}{2s}}}{N} \right)^{s/(n-2+s)}.
\end{equation}
Moreover, in \eqref{SKH} one needs to change $\mu(T_i \cap T_j)$ since $|e_i-e_j| \approx 2^{-k}$ implies $\text{diam}_E(T_i \cap T_j) \lesssim \delta 2^k$, thus by the modified Axiom 1 $\mathcal{L}^n(T_i \cap T_j) \lesssim 2^{ks} \delta^s \delta^{n-s}=  2^{ks} \delta^n$. Hence
\begin{equation}\label{SKHs}
S(K,H) \lesssim  \# I(B) \delta^{n-1} \left( N 2^{ks} \delta^{\frac{2n-4+6s-2sn}{2s}} \right)^{s/(n-2+s)}.
\end{equation}
Taking $N=2^{-\frac{n}{2}k} \delta^{-\frac{n}{2}}$, we get by \eqref{SHHs} 
\begin{equation}\label{SHHs1}
S(H,H) \lesssim \delta^{- \epsilon} \# I(B) \delta^{n-1} \left(2^{\frac{(2-n)k}{2}} \delta^{\frac{2n-4+6s-3sn}{2s}}\right)^{s/(n-2+s)},
\end{equation}
and by \eqref{SKHs}
\begin{equation}\label{SKHs1}
S(K,H) \lesssim  \# I(B) \delta^{n-1} \left( 2^\frac{(2s-n)k}{2} \delta^{\frac{2n-4+6s-3sn}{2s}}\right)^{s/(n-2+s)}.
\end{equation}
Since $n \ge 2\ge 2s$, the exponents of $2$ in \eqref{SHHs1} and \eqref{SKHs1} are non positive, hence we obtain
\begin{align*}
\int \left( \sum_{j=1}^m \chi_{T_j} \right)^{(n-2+2s)/(n-2+s)} \lesssim m \delta^{n-1} \delta^{\frac{2n-4+6s-3sn}{2(n-2+s)}},
\end{align*}
which implies \eqref{WM1} by Proposition \ref{PropDiscr} (this and Lemma \ref{Lemmadiscr} do not need to be modified since they do not use the second part of Axiom 1).
\end{proof}

It follows from Corollary \ref{Lpbound} (this holds as usual since it does not use the second part of Axiom 1) that the Hausdorff dimension of a $K$-Furstenberg set is $\ge 2 s+\frac{n-2}{2}$, which improves the lower bound $\frac{(2s-1)n+2}{2}$ obtained in Section \ref{Furst} for general Furstenberg sets when $n >2$. For $n=2$ we get $2s$, which is the known lower bound for general Furstenberg sets in the plane. Moreover, the lower bound $2s+\frac{n-2}{2}$ improves $s \frac{4n+3}{7}$, the best one found for general Furstenberg sets when $n \ge 9$ (see Remark \ref{Furstb}), when $s < \frac{7(n-2)}{2(4n-11)}$. Hence we proved Theorem \ref{Furstmod}.

\begin{rem}
In the plane for these special Furstenberg sets considered here the dimension can be $s+1$. Indeed, let $K \subset x_1$-axis be a Cantor set of Hausdorff dimension $s$ and let $F=K \times [0,1]$. Then $\dim F = \dim C+\dim
[0,1]=s+1$ and $F$ contains essentially a copy of $K$ in every line which makes an angle between $0$ and $\pi/4$ with the $x_1$-axis. 
\end{rem}

\section{Kakeya sets in $\mathbb{R}^n$ endowed with a metric homogeneous under non-isotropic dilations}\label{Rnd}

We will now consider the usual Kakeya sets in $\mathbb{R}^n$ and find dimension estimates for them with respect to the metric $d$ defined in \eqref{metric} (in which balls look like rectangular boxes as in Figure \ref{fig3}).

Let $X=\mathbb{R}^n=\mathbb{R}^{m_1} \times \dots \times \mathbb{R}^{m_s}$, where $n \ge2$, $s > 1$, $m_j\ge 0$, $j=1, \dots, s-1$, $m_s \ge 1$ and $m_j$ are integers, and let $Q=\sum_{j=1}^s j m_j$. Observe that when $m_j=0$, $\mathbb{R}^{m_j}=\{0\}$ and it could be removed but it will be convenient for the notation to keep it. Denote the points by $p=(x_1, \dots, x_n)=[p^1, \dots, p^s]$, where $x_i \in \mathbb{R}$, $i=1, \dots,n$, and $p^j \in \mathbb{R}^{m_j}$ for every $j=1, \dots, s$.
Consider the following metric on $\mathbb{R}^n$:
\begin{equation}\label{metric}
d(p,q)= \max \{ |x_i-y_i|^{1/j}: \sum_{k=1}^{j-1} m_k +1 \le i \le \sum_{k=1}^{j} m_k, j=1, \dots, s \},
\end{equation}
where $p=(x_1, \dots, x_n)$, $q=(y_1, \dots, y_n)$. We assumed $s>1$ because for $s=1$ the metric $d$ is essentially the usual Euclidean metric in $\mathbb{R}^n$.

\begin{figure}[H]
\centering
\includegraphics[scale=0.20]{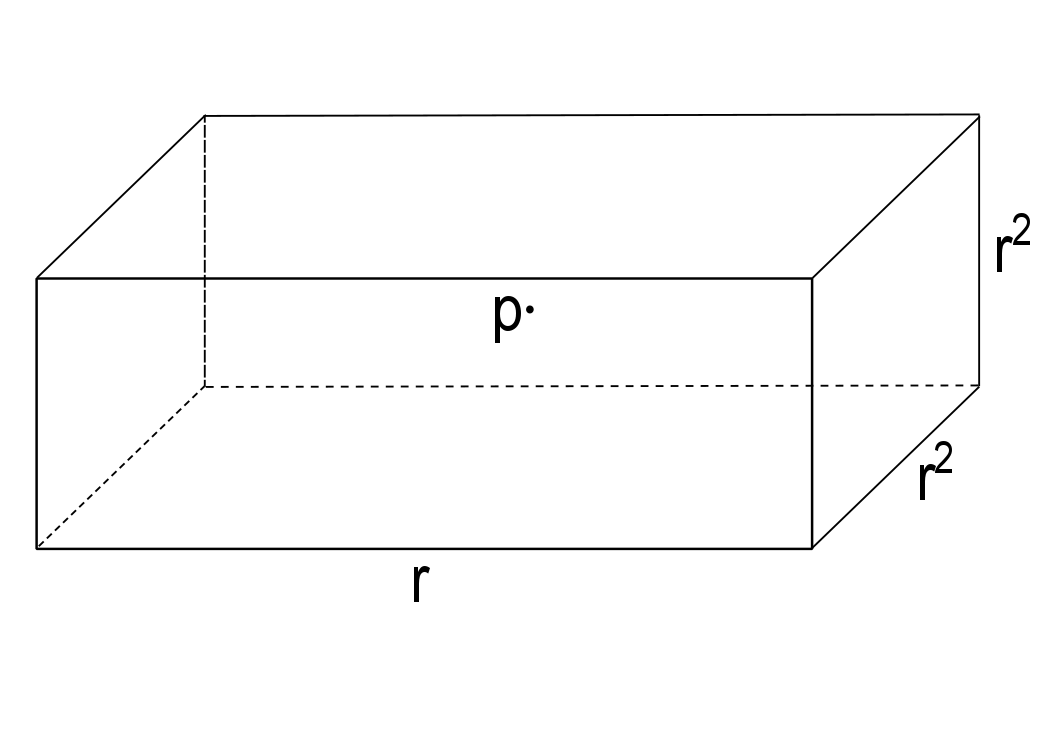}
\caption{A ball $B_d(p,r)$ with $r<1$ in $\mathbb{R}^3=\mathbb{R}\times \mathbb{R}^2$}
\label{fig3}
\end{figure}
The metric $d$ is homogeneous under the non-isotropic dilations
\begin{equation*}
\delta_\lambda(p)=[ \lambda p^1, \lambda^2 p^2, \dots, \lambda^s p^s],
\end{equation*}
where $\lambda>0$. Indeed,
\begin{equation*}
d( \delta_\lambda (p), \delta_\lambda (q))=\lambda d(p,q).
\end{equation*}
Balls centred at the origin $B_d(0,r)$ are rectangular boxes of the form
\begin{equation*}
[-r,r]^{m_1} \times [-r^2,r^2]^{m_2} \times \dots \times [-r^s,r^s]^{m_s}
\end{equation*}
and balls $B_d(p,r)$ are obtained by translating by $p$ the above boxes. Letting $\mu = \mathcal{L}^n$, we have $\mathcal{L}^n(B_d(p,r))= (2r)^Q$.

We will prove the following.

\begin{theorem}\label{dimKd}
Let $\mathbb{R}^n=\mathbb{R}^{m_1} \times \dots \times \mathbb{R}^{m_s}$ be endowed with the metric $d$ defined in \eqref{metric}. Let $K \subset \mathbb{R}^n$ be a standard Kakeya set. We have the following:
\begin{itemize}
\item[(a)] If $m_1=n-1$, $m_2=\dots=m_{s-1}=0$ and $m_s=1$, then 
\begin{eqnarray}\label{dimKd1}
\begin{split}
&\dim_d K \ge \frac{n+2s}{2} \quad \mbox{for} \quad n \le 12, \\
&\dim_d K \ge \frac{6}{11}n-\frac{6}{11}+s \quad \mbox{for} \quad n \ge 13.
\end{split}
\end{eqnarray}
\item[(b)]Otherwise,
\begin{equation*}
\dim_d K \ge \frac{6}{11}Q+\frac{5}{11}s.
\end{equation*}
\end{itemize}
\end{theorem}

We will prove the Theorem by showing that the Axioms 1-4 are satisfied and in the case $(a)$ also Axiom 5, thus we can use Wolff's method to obtain the first lower bound in  \eqref{dimKd1}. Then in Remark \ref{KTd} we will explain how to modify Katz and Tao's arithmetic argument to obtain the other lower bounds. 

Let $d'=d_E$ be the Euclidean metric.
We will denote a Euclidean ball in $\mathbb{R}^n$ by $B_n(a,r)$ and a Euclidean ball in the $x_1, \dots, x_{n-1}$-hyperplane by $B_{n-1}(u,r)$. Let $Z=Y=B_{n-1}(0,\bar{r}) $ be such that $\mathcal{L}^{n-1}( B_{n-1}(0,\bar{r})) \le 1$ and let $d_Z=d_{n-1}$ be the restriction of the metric $d$ to $Y$, that is
for $u=(u_1, \dots, u_{n-1})$, $v=(v_1, \dots, v_{n-1}) \in Y$  
\begin{eqnarray*}
d_{n-1}(u,v)=\max \{ \{ |u_i-v_i|^{1/j}, \sum_{k=1}^{j-1} m_k +1 \le i \le \sum_{k=1}^{j} m_k, j=1, \dots, s-1 \},\\ |u_{m_1+ \dots+ m_{s-1}+1}- v_{m_1+ \dots+ m_{s-1}+1}|^{1/s}, \dots, |u_{m_1+ \dots+m_s-1}-v_{m_1+ \dots+m_s-1}|^{1/s} \}.
\end{eqnarray*}
Note that if $m_s=1$ we do not have the terms with power $1/s$. Letting $\nu=\mathcal{L}^{n-1}$, we have $\nu(B_{d_{n-1}}(u,r)) \approx r^{Q-s}$, thus $S=Q-s$. For $p, q \in \mathbb{R}^n$, we have
\begin{equation*}
d(p,q)= \max \{ d_{n-1}((x_1, \dots, x_{n-1}),(y_1, \dots, y_{n-1})), |x_n-y_n|^{1/s}  \}.
\end{equation*}
For $ u =(u_1,\dots, u_{n-1}) \in B_{n-1}(0,\bar{r})$ we consider the unit segment
\begin{equation*}
F_u=I_u = \{ t(u,1): 0 \le t \le \frac{1}{\sqrt{|u|^2+1}}\},
\end{equation*}
where $| \cdot |$ denotes the Euclidean norm, and for $0 < \delta <1$ we define a tube with central segment $I_u$ and radius $\delta$ with respect to $d_{n-1}$,
\begin{align}\label{tube}
\begin{split}
T_u^\delta = \{ &p=(x_1, \dots, x_n): 0\le x_n \le \frac{1}{\sqrt{|u|^2+1}} ,\\ &d_{n-1}((x_1, \dots, x_{n-1}),x_n u) \le \delta \}.
\end{split}
\end{align}
For $a \in \mathcal{A}= \mathbb{R}^n$ denote by $F_u(a)=I_u(a)$ the unit segment starting from $a$
\begin{equation}\label{IuaRnd}
I_u(a) = \{ t(u,1)+a: 0 \le t \le \frac{1}{\sqrt{|u|^2+1}}\}
\end{equation}
and by $T_u^\delta(a)$ the corresponding tube
\begin{eqnarray}\label{tubea}
\begin{split}
T_u^\delta (a) = \{ p=(x_1, \dots, x_n): 0\le x_n-a_n \le \frac{1}{\sqrt{|u|^2+1}} ,\\ 
 d_{n-1}((x_1, \dots, x_{n-1}),(x_n-a_n) u+ (a_1, \dots, a_{n-1})) \le \delta \}.
\end{split}
\end{eqnarray}
Then $T^\delta_u(a)= T^\delta_u + a$. We define $\tilde{I}_u(a)$ as 
\begin{equation*}
\tilde{I}_u(a) = \{ t(u,1)+a: 0 \le t \le \frac{2}{\sqrt{|u|^2+1}}\},
\end{equation*}
which is a segment of (Euclidean) length $2$ containing $I_u(a)$, and $\tilde{T}^{2 \delta}_u(a)$ as 
\begin{eqnarray}\label{2tubea}
\begin{split}
\tilde{T}_u^{2\delta} (a) = \{ p=(x_1, \dots, x_n): 0\le x_n-a_n \le \frac{2}{\sqrt{|u|^2+1}} , \\ 
 d_{n-1}((x_1, \dots, x_{n-1}),(x_n-a_n) u+ (a_1, \dots, a_{n-1})) \le 2\delta \}.
\end{split}
\end{eqnarray}

Recall that a Kakeya set in $\mathbb{R}^n$ is a set $K \subset \mathbb{R}^n$ such that $\mathcal{L}^n(K)=0$ and for every  $e \in S^{n-1}$ there exists $b \in \mathbb{R}^n$ such that the unit segment $\{te+b, 0 \le t \le 1 \}$ is contained in $K$. In particular, for every $u \in B_{n-1}(0,\bar{r})$ there exists $a \in \mathbb{R}^n$ such that $I_u(a) \subset K$. 

\begin{rem}
Observe that $T^\delta_u$ is not exactly the $\delta$ neighbourhood of $I_u$ as in \eqref{tube1}. Indeed, it is contained in it since if $p \in T^\delta_u$ then $d(p, I_u) \le \delta$ but not vice versa. Nevertheless, the $\delta$ neighbourhood of $I_u$ is contained in $T^{l,2 \delta}_u$, where
\begin{eqnarray*}
T_u^{l,2\delta } = \{ p=(x_1, \dots, x_n): - \delta^s \le x_n\le \frac{1}{\sqrt{|u|^2+1}}+ \delta^s ,
 d_{n-1}((x_1, \dots, x_{n-1}),x_n u) \le 2 \delta \}.
\end{eqnarray*}
Indeed, if $p=(x_1, \dots, x_n)$ is in the $\delta$ neighbourhood of $I_u$ then
\begin{equation*}
\inf_{t \in \left[0, \frac{1}{\sqrt{1+|u|^2}} \right]} \max \{ d_{n-1}((x_1, \dots, x_{n-1}), tu), |x_n-t|^{1/s} \} < \delta.
\end{equation*}
Thus $-\delta^s \le x_n \le \frac{1}{\sqrt{1+|u|^2}}+ \delta^s$. Moreover, for every $t \in [0, \frac{1}{\sqrt{1+|u|^2}}]$,
\begin{eqnarray*}
\begin{split}
d_{n-1}((x_1, \dots, x_{n-1} ), x_n u)& \le d_{n-1}((x_1, \dots, x_{n-1} ), tu)+ d_{n-1}(tu, x_n u) \\
&\le d_{n-1}((x_1, \dots, x_{n-1} ), tu)+ |x_n-t|^{1/s}.
\end{split}
\end{eqnarray*}
Taking the infimum over $t$, we have $d_{n-1}((x_1, \dots, x_{n-1} ), x_n u) \le 2 \delta$, that is $ p \in T_u^{l,2\delta }$. The tube $ T_u^{l,2\delta }$ is not contained in $\tilde{T}^{2\delta}_u(a)$ since the points $p=(x_1, \dots, x_n)$ with $- \delta^s \le x_n <0$ are not in $\tilde{T}^{2\delta}_u(a)$. See Figure \ref{figTubed} for a visual example of the comparison between a tube $T^\delta_u(a)$ and $T^{l,2\delta}_u(a)$. We could work with the tubes $ T_u^{l,2\delta }$ instead of $\tilde{T}^{2\delta}_u$ since the same results would hold, but we will use the latter ones for simplicity.

\begin{figure}[H]
\centering
\includegraphics[scale=0.25]{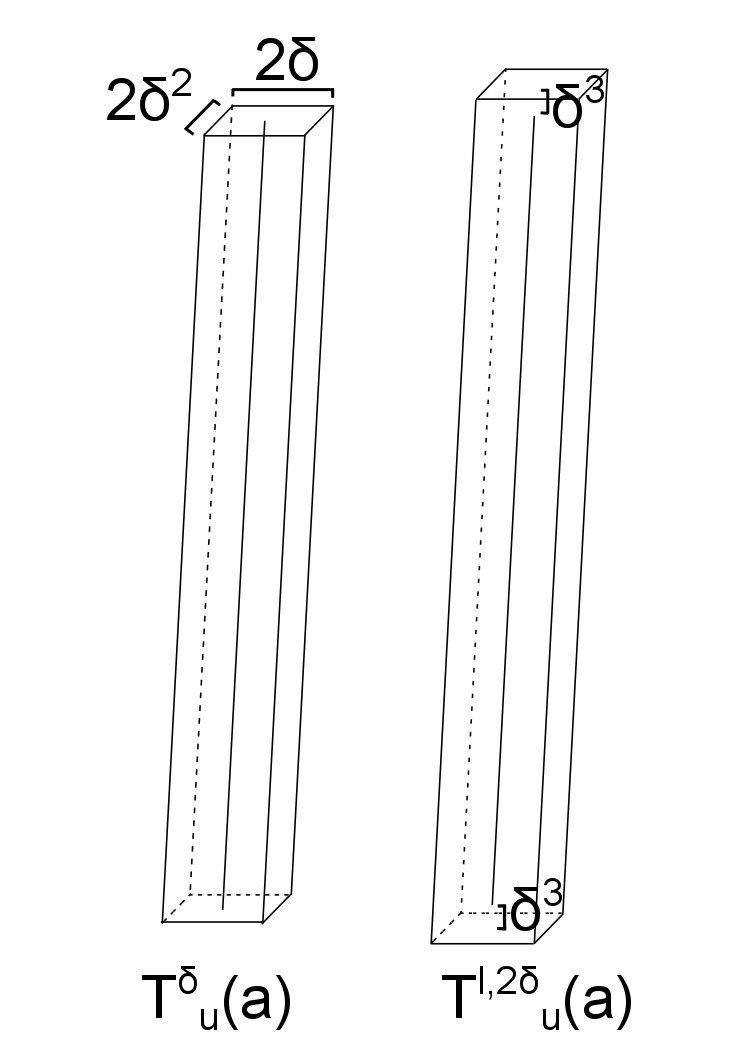}
\caption{Tubes $T^\delta_u(a)$ and $T^{l,2\delta}_u(a)$ in $\mathbb{R}^3=\mathbb{R}\times \mathbb{R} \times \mathbb{R}$}
\label{figTubed}
\end{figure}
\end{rem}

\textbf{Axiom 1}: Observe that we consider only segments $I_u(a)$ that make an angle $\ge \pi/4$ with the $x_1, \dots, x_{n-1}$-hyperplane thus the volume of the tubes $T^\delta_u(a)$ and $\tilde{T}^{2 \delta}_u(a)$  is essentially $ \delta^{Q-s}$. Moreover, if $A \subset \tilde{T}^{2 \delta}_u(a)$ then $\mathcal{L}^n(A) \lesssim \text{diam}_E(A) \delta^{Q-s}$. Thus Axiom 1 is satisfied with $T=S=Q-s$.

\textbf{Axiom 2}: We now show that Axiom 2 is satisfied with $\theta=0$. Let $a,x \in \mathbb{R}^n$, $u \in B_{n-1}(0, \bar{r})$, $x \in I_u(a)$, $\delta \le r \le 2 \delta$ be such that 
\begin{equation*}
\mathcal{H}^1_E(I_u(a) \cap B_d(x,r))= M.
\end{equation*}
Then for $p \in I_u(a) \cap B_d(x,r)$ any segment starting from $p$, parallel to the $x_1, \dots, x_{n-1}$-hyperplane and contained in $T^{\delta}_u(a)$ is also contained in $ B_d(x,2r)$. Thus for any segment $I$ parallel to $I_u(a)$ and contained in $T^{\delta}_u(a)$ we have $\mathcal{H}^1_E(I \cap B_d(x, 2r)) \ge M$. It follows by Fubini's theorem that
\begin{equation*}
\mathcal{L}^n(T^\delta_u(a) \cap  B_d(x,2 r)) \gtrsim M \mathcal{L}^n(T^\delta_u(a)).
\end{equation*}

\textbf{Axiom 3}: The following lemma shows that Axiom 3 holds.

\begin{lem}
For any $u$, $v \in Y$ and any $a,a' \in \mathbb{R}^n$
\begin{equation}\label{diamint}
\text{diam}_E (\tilde{T}_u^{2\delta}(a) \cap \tilde{T}_v^{2\delta}(a')) \le b \frac{\delta}{d_{n-1}(u,v)},
\end{equation}
where $b>0$ is a constant depending only on $n$.
\end{lem}

\begin{proof}

It is enough to estimate $\text{diam}_E(\tilde{T}^{2\delta}_u \cap \tilde{T}^{2\delta}_v)$ (that is, we can assume $a=a'=0$).
Let $p, q \in \tilde{T}_u^{2\delta} \cap \tilde{T}_v^{2\delta}$. Since we want to estimate $\sup |p-q|$, we can assume that $p=(x_1, \dots, x_{n-1},0)$ and $q=(y_1, \dots,y_n)$ with $y_n >0$. We have by \eqref{tube}
\begin{equation*}
d_{n-1}(y_n u, y_n v) \le d_{n-1}(y_n u, (y_1, \dots, y_{n-1}))+d_{n-1}((y_1, \dots, y_{n-1}), y_n v) \le 4 \delta.
\end{equation*}
On the other hand since $y_n \le 2$ (thus $y_n/2 \le 1$),
\begin{eqnarray*}
d_{n-1}(y_n u, y_n v) \ge d_{n-1} \left(\frac{y_n}{2}u, \frac{y_n}{2}v \right) \ge  \frac{y_n}{2} d_{n-1}(u,v),
\end{eqnarray*}
hence 
\begin{equation}\label{t'}
y_n \le \frac{8 \delta}{d_{n-1}(u,v)}.
\end{equation}
Thus
\begin{eqnarray*}
\begin{split}
&|p-q|^2= (x_1-y_1)^2+\dots + (x_{m_1}-y_{m_1})^2+(x_{m_1+1}-y_{m_1+1})^2+\dots +y_n^2=\\
&= (x_1-y_n u_1+y_n u_1-y_1)^2 +\dots + (x_{m_1}-y_nu_{m_1}+y_n u_{m_1}-y_{m_1})^2 +\\ &+(x_{m_1+1}-y_n u_{m_1+1}+y_n u_{m_1+1}-y_{m_1+1})^2+ \dots + y_n^2 \le \\
& \le 3(x_1^2 + y_n^2 u_1^2 + (y_n u_1-y_1)^2 )+\dots + 3(x_{m_1}^2+y_n^2 u_{m_1}^2+(y_nu_{m_1}-y_{m_1})^2)+ \\ & +3(x_{m_1+1}^2 + y_n^2 u_{m_1+1}^2 + (y_n u_{m_1+1}-y_{m_1+1})^2 )+ \dots + y_n^2 \le \\
& \le 3( 8 \delta^2  + y_n^2)  +\dots +3(8 \delta^2 +y_n^2) +3(32 \delta^4  + y_n^2) + \dots+ y_n^2,
\end{split}
\end{eqnarray*}
where we used \eqref{tube} ($p,q \in \tilde{T}_u^{2\delta}$, $x_n=0$) and $|u| \le 1$. Then by \eqref{t'} we get \eqref{diamint}.

\end{proof}

\textbf{Axiom 4}: To see that Axiom 4 is satisfied, let $u,v \in Y$, $u \in B_{d_{n-1}}(v, \delta)$. We want to show that $T^\delta_u \subset \tilde{T}^{2 \delta}_v$. Let $p =(x_1, \dots, x_n) \in T^\delta_u$. Then $0 \le x_n \le \frac{1}{\sqrt{|u|^2+1}} \le 1 \le \frac{2}{\sqrt{|v|^2+1}}$. Moreover, since $d_{n-1}((x_1, \dots, x_{n-1}), x_n u) \le \delta$, we have by triangle inequality
\begin{eqnarray*}
\begin{split}
d_{n-1}((x_1, \dots, x_{n-1}), x_n v)& \le d_{n-1}((x_1, \dots, x_{n-1}), x_n u) + d_{n-1}(x_n u, x_n v) \\
& \le \delta + d_{n-1}(u,v) \le 2 \delta.
\end{split}
\end{eqnarray*}
Thus $p \in \tilde{T}^{2 \delta}_v$ by \eqref{2tubea}.

Hence all the axioms $1-4$ are satisfied. We obtain by Theorem \ref{Bourgain} (Bourgain's method) the lower bound $\frac{2Q-S}{2}= \frac{Q+s}{2}$ for the Hausdorff dimension with respect to $d$ of any Kakeya set in $\mathbb{R}^n$. Moreover, we have a weak type inequality
\begin{equation*}
\mathcal{L}^{n-1}(\{ u \in Y: (\chi_E)^d_\delta(u)>\lambda\}) \le C \lambda^{-(Q-s+2)/2} \delta^{-(Q-s)/2} \mathcal{L}^n(E)
\end{equation*}
for every Lebesgue measurable set $E \subset \mathbb{R}^n$ and for any $\lambda >0$, $0<\delta <1$. As we have seen, this implies an $L^p \rightarrow L^q$ inequality for any $1 < p < \frac{Q-s+2}{2}$ and $q=\frac{p(Q-s)}{2(p-1)}$. 

\textbf{Axiom 5}: Observe that when $\theta=0$, $\lambda$ is the only one of $\lambda$, $\alpha$, $\beta$ that appears in the dimension estimate $ \frac{2Q-S-\lambda+2}{2}$ implied by Theorem \ref{Wolff}. We will see in the following example that in general we cannot have anything better than $\lambda= Q-s-n+2$.

\begin{ex}
For any $n \ge 3$, let $s\ge2$, $m_1=n-2$, $m_2= \dots =m_{s-1}=0$, $m_s=2$, thus $Q=n-2+2s$ and $\mathbb{R}^n=\mathbb{R}^{n-2} \times \{0\} \times \dots \times \{0\} \times \mathbb{R}^2$. Then $Q-s-n+2=s$.\\
Fix some $0 < \beta, \gamma, \delta < 1$, $\delta \le \beta$, and suppose that $u=[0,0]$, $T=\tilde{T}^{2\delta}_0$ and $u^1=[0, \beta^s/8^s]$ (thus $d_{n-1}(u,u^1) = \beta/8$), $T_1= \tilde{T}^{2\delta}_{u^1}$. \\
Let also $u^2, \dots, u^L, u^{L+1}, \dots, u^N$ be such that for $i=2, \dots , L+1$, $u^i= [0, u^i_{n-1}]$, $\beta^s/8^s \le u^i_{n-1} \le (8^s+1) \beta^s/8^s$ (so that $d_{n-1}(u^i, u) \ge \beta/8$ and $d_{n-1} (u^i, u^1) \le \beta$) and $|u^i_{n-1} - u^l_{n-1}| > \delta^s$ for $l \ne i \in \{1, \dots, L+1 \}$ (that is, $d_{n-1}(u^i,u^l) > \delta$), where $L \approx \beta^s \delta^{-s}$. \\
Moreover, fix some point $p \in \tilde{I}_{u^1}$ such that $|p| > \gamma + 4 \delta$ and assume that for $ i=2, \dots, L+1$ the points $a_i$ are given so that $ \tilde{I}_{u^1} \cap  \tilde{I}_{u^i}(a_i)= \{ p \}$ and $\tilde{I}_{u^i}(a_i) \cap T \neq \emptyset$.\\
Then
\begin{eqnarray*}
\begin{split}
&\# \{ i : d_{n-1}(u^i,u^1) \le \beta, T_i \cap T_1 \ne \emptyset, d_E(T_i \cap T_1, T_1 \cap T) \ge \gamma \}  \\
&\ge \# \{ i: u^i=[0,u^i_{n-1}],\beta^2/2^s \le u^i_{n-1} \le (1+2^s) \beta^s/2^s , \tilde{I}_{u^i}(a_i) \cap \tilde{I}_{u^1}= \{p\} \} \\
&=L \approx \beta^s \delta^{-s}.
\end{split}
\end{eqnarray*}
Hence in this case the exponent $\lambda$ of $\delta$ in Axiom 5 is at least $s$.
\end{ex}

Thus the best lower bound that we could get by Theorem \ref{Wolff} (Wolff's argument) is $\frac{n+2s}{2}$, which improves the estimate $\frac{Q+s}{2}$ found using Bourgain's method only when $m_1=n-1$, $m_2=m_3= \dots =m_{s-1}=0$ and $m_s=1$, which is case $(a)$ in Theorem \ref{dimKd}. In this case the tubes are essentially Euclidean since $d_{n-1}$ is equivalent to the Euclidean metric in $\mathbb{R}^{n-1}$. Thus Axiom 5 holds with $\lambda=1=Q-s-n+2$ and $\alpha=n-2$, giving the lower bound $\frac{n+2s}{2}$.

\begin{rem}\label{KTd}(Completing the proof of Theorem \ref{dimKd})
The arithmetic method introduced by Bourgain in \cite{Bourgarithm} and developed by Katz and Tao in \cite{Katz&Tao1} can be modified almost straightforwardly to this case to obtain the lower bound $ \frac{6}{11}Q+\frac{5}{11} s$ for the Hausdorff dimension of Kakeya sets with respect to $d$. This will complete the proof of Theorem  \ref{dimKd} since it improves the estimate $\frac{Q+s}{2}$ found with Bourgain's method and in case $(a)$ of the theorem the lower bound $\frac{6}{11}n-\frac{6}{11}+s$ improves the one found with Wolff's method in high dimension ($n \ge 13$). 

We will show here only the proof  for the lower Minkowski dimension (see Theorem 23.7 in \cite{Mattila} for the Euclidean case). This can be extended to the Hausdorff dimension using a deep number theoretic result and we will just mention what we would need to modify to adapt the Euclidean proof to our situation. We first recall the definition of lower Minkowski dimension with respect to the metric $d$. Let $A\subset \mathbb{R}^n$ be bounded and let $A(\delta)_d=\{ p \in \mathbb{R}^n: d(p,A) < \delta\}$ be its $\delta$ neighbourhood with respect to $d$. Then we can define
\begin{equation*}
\underline{\dim}_{M,d} A= \inf\{t>0: \liminf_{\delta \rightarrow 0} \delta^{t-Q} \mathcal{L}^n(A(\delta)_d)=0\}.
\end{equation*}
Recall that $\underline{\dim}_{M,d} K \ge \dim_d K$.

 We will prove the Minkowksi dimension lower bound following the Euclidean proof with some minor changes.

\begin{theorem}
The lower Minkowski dimension with respect to $d$ of any bounded Kakeya set $K \subset \mathbb{R}^n$ is $\ge \frac{6}{11}Q+ \frac{5}{11}s$.
\end{theorem}

\begin{proof}
We can reduce to have a set $K \subset \mathbb{R}^n$ such that for every $u \in B_{n-1}(0,\bar{r})$ there exists $a \in [0,1]^{n-1}  \subset x_1, \dots, x_{n-1}$-hyperplane such that
\begin{equation*}
I_u(a)= \{ t(u,1)+a: 0 \le t \le 1\} \subset K.
\end{equation*}
We consider here for convenience these longer segments than those considered in \eqref{IuaRnd} even if we could use also those. Suppose by contradiction that $\underline{\dim}_{M,d}  K < cQ+(1-c)s$ for some $c < 6/11$. By definition of Minkowski dimension this means that for some arbitrarily small $\delta$ we have
\begin{equation*}
\mathcal{L}^n(K(2\delta)_d) \le \delta^{(1-c)(Q-s)}.
\end{equation*}
Using this we have by Chebyshev's inequality and Fubini's theorem 
\begin{align*}
&\mathcal{L}^1( \{ t \in [0,1]: \mathcal{L}^{n-1}(K(2\delta)_d \cap \{ x_n=t\}) > 100 \delta^{(1-c)(Q-s)} \}) \\
& \le \frac{\int_0^1  \mathcal{L}^{n-1}(K(2\delta)_d \cap \{ x_n=t\}) dt }{100 \delta^{(1-c)(Q-s)}} = \frac{\mathcal{L}^n(K(2\delta)_d)}{100 \delta^{(1-c)(Q-s)}} \le \frac{1}{100}.
\end{align*}
Thus the measure of the complement of this set is $>99/100$, which implies that we can find
\begin{equation*}
\bar{t}, \bar{t}+d, \bar{t}+2d \in \{ t \in [0,1]: \mathcal{L}^{n-1}(K(2\delta)_d \cap \{ x_n=t\}) \le 100 \delta^{(1-c)(Q-s)} \}
\end{equation*}
for some $d \le 1/2$. We can assume that $\bar{t}=0$ and $d=1/2$ so the numbers are $0$, $1/2$ and $1$. Letting now for $t \in [0,1]$,
\begin{equation}\label{Kt}
K[t]=\{ i \in \delta  \mathbb{Z}^{m_1} \times \dots \times \delta^s \mathbb{Z}^{m_s-1}: (i,t) \in K(\delta)_d\},
\end{equation}
we have that the balls $B_{d}((i,t), \delta/3)$, $i \in K[t]$, are disjoint and contained in $K(2 \delta)_d$. Thus we have
\begin{align*}
\mathcal{L}^{n-1}\left( \bigcup_{i \in K[t]} B_d\left((i,t),\frac{\delta}{3}\right) \cap \{ x_n=t \} \right) \le \mathcal{L}^{n-1} (K(2\delta)_d \cap \{ x_n=t\}),
\end{align*}
which implies for $t=0,1/2$ and $1$,
\begin{equation*}
\# K[t] \delta^{Q-s} \lesssim \delta^{(1-c)(Q-s)}.
\end{equation*}
Hence
\begin{equation*}
\# K[0], \# K[1/2], \# K[1] \lesssim \delta^{c(s-Q)}.
\end{equation*}
Let now
\begin{equation*}
G=\{(p,q) \in K[0] \times K[1]: (p,0),(q,1) \in T^\delta_u(a)\subset K(\delta)_d \ \mbox{for some}\ u,a \}.
\end{equation*}
For $(p,q) \in G$, we have that $(p,0) ,(q,1) \in T^\delta_u(a)$, which implies that $((p+q)/2,1/2) \in  T^\delta_u(a)\subset K(\delta)_d $. Since $(p+q)/2 \in (\delta  \mathbb{Z}^{m_1} \times \dots \times \delta^s \mathbb{Z}^{m_s-1})/2$,  it follows that the cardinality of $\{p+q:(p,q) \in G \}$ is bounded by that of $ K[1/2]$. Hence
\begin{equation*}
\# \{p+q \in G: (p,q) \in G \} \lesssim \delta^{c(s-Q)}.
\end{equation*}
On the other hand, there are essentially $\delta^{s-Q}$ tubes $T^\delta_u(a)$ in $K(\delta)_d$ that are $\delta$-separated, each of them contains points $(p,0)$ and $(q,1)$ for some $(p,q) \in G$ and $p-q$ give their directions, thus
\begin{equation*}
\# \{p-q \in G: (p,q) \in G \} \gtrsim \delta^{s-Q}.
\end{equation*}
This gives a contradiction with the following combinatorial proposition (see Proposition 23.8 in \cite{Mattila}).
\begin{prop}
Let $A$ and $B$ be finite subsets of a free Abelian group such that $\# A \le N$ and $\#B \le N$. Suppose that $G \subset A \times B$ is such that
\begin{equation*}
\# \{ x+y \in G: (x,y) \in G \} \le N.
\end{equation*}
then
\begin{equation*}
\# \{ x-y \in G: (x,y) \in G \} \le N^{11/6}.
\end{equation*}
\end{prop}
\end{proof}

As can be seen, this proof relies on the arithmetic structure of $\mathbb{R}^n$, thus it cannot be generalized to a metric space $X$ as considered in our axiomatic setting. 

The proof of the lower bound for the Hausdorff dimension relies on a deep result in number theory proved by Heath-Brown, which gives a sufficient condition for a set of natural numbers to contain an arithmetic progression of length $3$ (see Proposition 23.10 in \cite{Mattila}). 

In the proof for the Hausdorff dimension of the classical Kakeya sets (see Theorem 23.11 in \cite{Mattila}) instead of looking at intersections with hyperplanes $\{x_n=t\}$ one considers the following sets. For certain small enough $0< \eta <1$ and $\delta$ one defines
\begin{equation*}
A_{j,i}=\{ x \in \mathbb{R}^n: j \delta + i N \delta \le x_n < j \delta + i N \delta + \delta \}
\end{equation*}
for $i=0, \dots, M$ and $j=0, \dots, N-1$, where $N,M$ are integers such that $N \le \delta^{\eta-1}<N+1$, $MN\delta>1$ and $M \approx \delta^{-\eta}$. Then one finds two points $a_e, b_e$ in $\delta \mathbb{Z}^{n}$ belonging to $A_{j,i}$ for two different $i$ and to the $n \delta$ neighbourhood of a segment (contained in the Kakeya set) such that also $(a_e+b_e)/2$ belongs to the $n \delta$ neighbourhood of the same segment. In the case of $(\mathbb{R}^n,d)$ we need to replace $A_{j,i}$ by
\begin{equation*}
\{ x \in \mathbb{R}^n: j \delta^s + i N \delta^s \le x_n < j \delta^s + i N \delta^s + \delta^s \}
\end{equation*}
with $M,N$ such that $N \le \delta^{\eta-s}<N+1$, $MN\delta^s>1$ and $M \approx \delta^{-\eta}$. Moreover $\delta \mathbb{Z}^n$ is replaced by $\delta \mathbb{Z}^{m_1} \times \dots \times \delta^s \mathbb{Z}^{m_s}$, similarly to what was done in \eqref{Kt}.
\end{rem}

The proof of Theorem \ref{dimKd} can then be concluded arguing as in Theorem 23.11 in \cite{Mattila}.

\section{Bounded Kakeya sets and a modification of them in Carnot groups of step 2}\label{Carnots2}

We now consider both bounded Euclidean Kakeya sets in a Carnot group of step 2 and a bit modified Kakeya sets (which we will call LT-Kakeya sets, where LT stands for left translation), that is sets containing a left translation of every segment through the origin with direction close to the $x_n$-axis. We will see that when the second layer of the group has dimension $1$ we can obtain lower bounds for their Hausdorff dimension with respect to a left invariant and one-homogeneous metric using the axiomatic setting, whereas if the dimension is $>1$ we cannot.

We recall here briefly some facts about Carnot groups that will be useful later (see for example \cite{BLU} for more information).
Let $\mathbb{G}$ be a Carnot group of step $s$ and homogeneous dimension $Q= \sum_{j=1}^s j m_j$, where $m_j=\dim V_j$ and $g= V_1 \oplus \dots \oplus V_s$ is a stratification of the Lie algebra $g$ of $\mathbb{G}$ such that $[V_1,V_i]=V_{i+1}$, $V_s \ne \{0\}$ and $V_j=\{0\}$ if $j>s$ (here  $[V_1,V_i]$ is the subspace of $g$ generated by the commutators $[X,Y]$ with $X \in V_1$, $Y \in V_i$). Via exponential coordinates, we can identify $\mathbb{G}$ with $\mathbb{R}^n$, $n=m_1+ \dots + m_s$, and denote the points by $p=(x_1, \dots , x_n)=[p^1, \dots, p^s]$ with $p^i \in \mathbb{R}^{m_i}$. There are two important families of automorphisms of $\mathbb{G}$, which are the left translations 
\begin{equation*}
\tau_p(q)=p \cdot q,
\end{equation*}
where $\cdot$ denotes the group product, and the dilations defined for $\lambda>0$ as
\begin{equation*}
\delta_\lambda (p)=[\lambda p^1, \lambda^2 p^2, \dots, \lambda^s p^s].
\end{equation*}

One can define the Carnot-Carath\'eodory distance $d_{CC}$ on $\mathbb{G}$ as follows. Fix a left invariant inner product $\langle \cdot, \cdot \rangle_\mathbb{G}$ on $V_1$ and let $X_1, \dots, X_{m_1}$ be an orthonormal basis for $V_1$. An absolutely continuous curve $\gamma: [0,1] \rightarrow \mathbb{G}$ is called horizontal if its derivative $\gamma'$ lies in $V_1$ almost everywhere. The horizontal length of $\gamma$ is
\begin{equation*}
\mbox{Length}_{CC}(\gamma)= \int_0^1  \langle \gamma'(t),\gamma'(t) \rangle_\mathbb{G}^{1/2} dt. 
\end{equation*}
The Carnot-Carath\'eodory distance is defined for any $p, q \in \mathbb{G}$ as
\begin{equation*}
d_{CC}(p,q)=\inf \{\mbox{Length}_{CC}(\gamma): \gamma \  \mbox{horizontal curve}, \gamma(0)=p, \gamma(1)=q \}.
\end{equation*}
It is left invariant, that is
\begin{equation*}
d_{CC}(p \cdot q, p \cdot q')=d_{CC}(q,q'),
\end{equation*}
and one-homogeneous with respect to the dilations:
\begin{equation*}
d_{CC}(\delta_\lambda(p), \delta_\lambda(q))= \lambda d_{CC}(p,q).
\end{equation*}
If $B_{CC}(p,r)$ denotes a ball in $\mathbb{G}$ with respect to $d_{CC}$, then by the Ball-Box theorem (see \cite{Montgomery}, Theorem 2.10) there exists a constant $C>0$ such that
\begin{equation}\label{BallBox}
\mbox{Box}_{CC}(p, r/C) \subset B_{CC}(p,r) \subset \mbox{Box}_{CC}(p,Cr)
\end{equation}
for every $r>0$, where $ \mbox{Box}_{CC}(0,r)=B_d(0,r)=[-r,r]^{m_1} \times \dots \times [-r^s,r^s]^{m_s}$ and $\mbox{Box}_{CC}(p, r)= \tau_p( \mbox{Box}_{CC}(0,r))= p \cdot \mbox{Box}_{CC}(0,r)$ (here $d$ is the metric defined earlier in \eqref{metric}). 

We call a metric homogeneous if it is left invariant and one-homogeneous under the dilations of the group. Such a metric is equivalent to the Carnot-Carath\'eodory one (see Corollary 5.1.5 in \cite{BLU}, where it is shown that any two homogeneous metrics are equivalent).

We now define LT-Kakeya sets in Carnot groups of step two and see if we can find lower bounds for their Hausdorff dimension with respect to a homogeneous metric. We will then consider the classical bounded Kakeya sets.

Let $\mathbb{G}$ be a Carnot group of step $s=2$. We can identify $\mathbb{G}$ with $\mathbb{R}^n=\mathbb{R}^{m_1} \times \mathbb{R}^{m_2}$, denoting the points by $p=[p^1,p^2]=(x_1, \dots, x_n)$ with $p^i \in \mathbb{R}^{m_i}$. The group product has the form (see Proposition 2.1 in \cite{FSSC} or Lemma 1.7.2 in \cite{Monti})
\begin{equation}\label{operation}
p \cdot q= [p^1+q^1, p^2+q^2+P(p,q)],
\end{equation}
where $P=(P_{m_1+1}, \dots, P_n)$ and each $P_j$ is a homogeneous polynomial of degree $2$ with respect to the dilations of $\mathbb{G}$, that is
\begin{equation*}
P_j(\delta_\lambda p, \delta_\lambda q)=\lambda^2 P_j(p,q)
\end{equation*}
for every $p,q \in \mathbb{G}$. Moreover for every $j=m_1+1, \dots, n$ and every $p,q \in \mathbb{G}$,
\begin{equation*}\label{Ppp}
P_j(p,0)=P_j(0,q)=0, \ P_j(p,-p)=P_j(p,p)=0
\end{equation*}
and 
\begin{equation*}
P_j(p,q)=P_j(p^1,q^1).
\end{equation*}
It follows that each $P_j$ has the form (see Lemma 1.7.2 in \cite{Monti})
\begin{equation*}
P_j(p,q)=\sum_{1\le l<i\le m_1} b^j_{l,i} (x_l y_i - x_i y_l),
\end{equation*}
for some $b^j_{l,i} \in \mathbb{R}$, where $p=(x_1, \dots, x_n)$, $q=(y_1, \dots, y_n)$. We will work with the following metric, which is equivalent to the Carnot-Carath\'eodory metric:
\begin{equation}\label{dinfty1}
d_\infty (p,q)=d_\infty (q^{-1} \cdot p, 0),
\end{equation}
where
\begin{equation}\label{dinfty2}
d_\infty (p,0)=\max \{ |p^1|_{\mathbb{R}^{m_1}}, \epsilon |p^2|^{1/2}_{\mathbb{R}^{m_2}} \}.
\end{equation}
Here $\epsilon \in (0,1)$ is a suitable constant depending on the group structure (see Section 2.1 and Theorem 5.1 in \cite{FSSC}).
The metric $d_\infty$ is left invariant and one-homogeneous with respect to the dilations. Moreover, the following relations between the metric $d_\infty$ and the Euclidean metric $d_E$ holds, since they hold for any homogeneous metric on $\mathbb{G}$ (see Proposition 5.15.1 in \cite{BLU}).Here and in the following $B_n(0,R)$ denotes the Euclidean ball in $\mathbb{R}^n$ with center $0$ and radius $R$.
 
\begin{lem}\label{inftyECR}
Let $R >0$. Then there exists a constant $C_R >0$ such that for every $p, q \in B_n(0,R)$
\begin{equation*}
\frac{1}{C_R} d_E(p,q) \le d_\infty(p,q) \le C_R d_E(p,q)^{1/2}.
\end{equation*}
\end{lem}

The Lebesgue measure $\mathcal{L}^n$ is the Haar measure of the group $\mathbb{G}$. Letting $X=\mathbb{G}$, $d=d_\infty$ and $\mu =  \mathcal{L}^n$, we have $\mu (B_\infty(p,r)) \approx r^Q$, where $Q=m_1+2 m_2$ is the homogeneous dimension and $B_\infty(p,r)=B_{d_\infty}(p,r)$.
 Observe that the Hausdorff dimension of $\mathbb{G}$ with respect to $d_\infty$ (and with respect to any homogeneous metric) is $Q$. Moreover, for any set $A \subset \mathbb{G}$ we have $\dim_\mathbb{G}A \ge \dim_E A$, where $\dim_\mathbb{G}$ denotes the Hausdorff dimension with respect to any homogeneous metric and $\dim_E$ the Euclidean Hausdorff dimension. We let $d'=d_E$ be the Euclidean metric on $\mathbb{G}$.
  
As in the previous section, we denote by $B_{n-1}(0,1)$ the Euclidean unit ball in the $x_1, \dots, x_{n-1}$-hyperplane. For $ u =(u_1,\dots, u_{n-1}) \in B_{n-1}(0,1)$ we consider the segment (of unit length with respect to $d'$) starting from the origin
\begin{equation}\label{Iu}
I_u = \{ t(u,1): 0 \le t \le \frac{1}{\sqrt{|u|^2+1}}\}
\end{equation}
and for $0 < \delta <1$ we define a tube with central segment $I_u$ and radius $\delta$ with respect to $d_\infty$ by
\begin{eqnarray}\label{tubedinfty}
T_u^\delta = \{ p: d_\infty(p,I_u) \le \delta \}.
\end{eqnarray}
Let $R>0$ and for $a=[a^1,a^2]\in B_n(0,R)$ (which will be our set of parameters $\mathcal{A}$) let
\begin{align*}
F_u(a)=I_u(a) = \tau_a(I_u)
= \{[ a^1+tu^1, a^2+t(u^2,1)+t P(a^1,u^1)]:  0 \le t \le \frac{1}{\sqrt{|u|^2+1}}\},
\end{align*}
where $u=[u^1,u^2]$, $u^2 \in \mathbb{R}^{m_2-1}$, and the corresponding tube
\begin{equation}\label{tubeTau}
T^\delta_u(a)=\{p: d_\infty(p,I_u(a)) \le \delta \}= \{ a \cdot p: p \in T^\delta_u \}= \tau_a(T^\delta_u).
\end{equation}
We let $\mu_{u,a}=\mathcal{H}^1_E \big|_{I_u(a)}$. Let $Z=Y=B_{n-1}(0,r_R)\subset B_{n-1}(0,1)$, where
\begin{equation}\label{rR}
r_R< \min \left \{ 1, \frac{1}{2 C_{R,n}} \right \},
\end{equation}
and
\begin{equation}\label{CRn}
C_{R,n}=R \sqrt{m_2 m_1 (m_1-1) \max_{m_1+1\le j \le n} \max_{1\le l<i \le m_1} (b^j_{l,i})^2}.
\end{equation}
Observe that $C_{R,n}$ depends also on the group structure, that is in this case on the coefficients $b^j_{l,i}$ of the polynomials $P_j$. Here and in the following we do not write for brevity the dependence of constants on the group structure, that is also on $\epsilon$ that appears in the definition \eqref{dinfty2} of the metric $d_\infty$.

Let $\nu=\mathcal{H}^{n-1}_E \big|_{Y}$. We will see that for Axiom 3 to hold we need to take $d_Z$ to be essentially the Euclidean metric on $Y$.

\begin{rem}
Observe that $\mathcal{H}^1_E(I_u)=1$, whereas the Euclidean length of $I_u(a)$ varies with $a$. Anyway, $\sqrt{1-2C_{R,n} r_R} \le \mathcal{H}^1_E(I_u(a))\le \sqrt{2+2C_{R,n}^2 r_R^2}$, where $1-2C_{R,n} r_R>0$ since $r_R < 1/(2 C_{R,n})$. Indeed, $I_u(a)$ is a segment with starting point $a$ and endpoint 
\begin{equation*}
q=\left[a^1+\frac{u^1}{\sqrt{|u|^2+1}}, a^2+ \frac{1}{\sqrt{|u|^2+1}}((u^2,1)+P(a^1,u^1))\right].
\end{equation*}
Thus
\begin{align*}
|a-q|^2=\frac{1}{|u|^2+1}(|u^1|^2+|(u^2,1)+P(a^1,u^1)|^2).
\end{align*}
We have
\begin{align}\label{Pau}
|P(a^1,u^1)|^2&= \sum_{j=m_1+1}^n P_j(a^1,u^1)^2 \nonumber \\
&=\sum_{j=m_1+1}^n \left( \sum_{1\le l<i\le m_1} b^j_{l,i} (a_lu_i-a_i u_l) \right)^2 \nonumber \\
&= \sum_{j=m_1+1}^n \left( \sum_{1\le l<i\le m_1} b^j_{l,i} \langle (a_l,-a_i), (u_i,u_l) \rangle \right)^2\\
&\le m_2 m_1 (m_1-1) \max_{m_1+1\le j \le n} \max_{1\le l<i \le m_1} (b^j_{l,i})^2 |a|^2 |u|^2 \nonumber \\
&\le C_{R,n}^2 |u|^2, \nonumber
\end{align}
where $\langle \cdot, \cdot \rangle$ denotes the scalar product and we used the Cauchy-Schwarz inequality.\\
Hence
\begin{align*}
|a-q|^2 &\le \frac{2}{|u|^2+1}(|u^1|^2+|u^2|^2+1+|P(a^1,u^1)|^2)\\
& \le 2+2 \frac{|P(a^1,u^1)|^2}{|u|^2+1}\\
& \le 2+2 C_{R,n}^2 |u|^2 \le 2+ 2C_{R,n}^2 r_R^2
\end{align*}
and by Cauchy-Schwarz inequality
\begin{align*}
|a-q|^2 & = \frac{1}{|u|^2+1}(|u^1|^2+|u^2|^2+1 +|P(a^1,u^1)|^2 +2 \langle (u^2,1), P(a^1,u^1) \rangle)\\
& \ge \frac{1}{|u|^2+1} (|u^1|^2+|u^2|^2+1 +|P(a^1,u^1)|^2 -2|(u^2,1)| |P(a^1,u^1)|)\\
& \ge \frac{1}{|u|^2+1}(|u^1|^1+|u^2|^2+1 - 2  \sqrt{|u|^2+1} |P(a^1,u^1)|)\\
& \ge 1- \frac{2|P(a^1,u^1)|}{\sqrt{|u|^2+1}}\\
& \ge 1-2C_{R,n}|u| \ge 1-2 C_{R,n} r_R.
\end{align*}
\end{rem}

We are now ready to define LT-Kakeya sets.

\begin{defn}
We say that a set $K \subset \mathbb{G}$ is a (bounded) \textit{LT-Kakeya set} if for every $ u \in B_{n-1}(0,r_R)$ there exists $a \in B_n(0,R)$ such that $I_u(a) \subset K$.
\end{defn}

\begin{rem}(\textbf{LT-Kakeya sets can have measure zero})\\
LT-Kakeya sets are a natural variant of Kakeya sets in Carnot groups. Indeed, if $K \subset \mathbb{R}^n$ is a classical bounded Kakeya set (say $K \subset B_n(0,R)$) then in particular for every $u \in B_{n-1}(0,r_R)$ there exists $a \in B_n(0,R)$ such that $I_u +a \subset K$. Hence Kakeya sets contain a Euclidean translated of every segment $I_u$. On the other hand, LT-Kakeya sets contain a left translated $a \cdot I_u=I_u(a)$ of every $I_u$.
Note that all segments $I_u+a$ have the same direction as $I_u$, whereas the direction of $I_u(a)$ depends also on $a$.

LT-Kakeya sets can have Lebesgue measure zero. Indeed, they can be obtained as cartesian products $K \times\mathbb{R}^{m_2}$, where $K$ is a bounded Kakeya set in $\mathbb{R}^{m_1}$. If $K \subset x_1, \dots, x_{m_1}$-plane is a bounded Kakeya set, then $\mathcal{L}^{m_1}(K)=0$ and for every $e \in S^{m_1-1}$ there exists $p \in B_{m_1}(0,R)$ such that $\{se+p: 0 \le s \le 1\} \subset K$, where $B_{m_1}(0,R)$ denotes the ball with center $0$ and radius $R$ in the $ x_1, \dots, x_{m_1}$-plane. Let $B= K \times \mathbb{R}^{m_2}$. Then $\mathcal{L}^n(B)=0$. Moreover, for every $u \in B_{n-1}(0,r_R)$, $u^1 \neq 0$, there is $e \in S^{m_1-1}$ such that $u^1=|u^1|e$, hence there exists $p \in B_{m_1}(0,R)$ such that $\{ tu^1+p : 0 \le t \le 1/|u^1| \} \times \mathbb{R}^{m_2} \subset B$. In particular, for every $q \in \mathbb{R}^{m_2}$, $|q| \le R$, we have $\{ [tu^1+p, q+ t(u^2,1) +tP(p,u)]: 0 \le t \le 1/|u^1| \} \subset B$. This means that for every $u \in B_{n-1}(0,r_R)$ there exists $[p,q] \in B_n(0,R)$ such that $I_u([p,q]) \subset B$. Thus $B$ is an LT-Kakeya set.
\end{rem}

We will now consider separately two cases: I) $m_2>1$, in which case in general the axioms are not satisfied; II) $m_2=1$, in which case they are. We will see that the same holds also for the usual Kakeya sets, even if we will find for their Hausdorff dimension different lower bounds in case II.

\subsection{Case I: $m_2>1$}

We will see that in this case Axiom 1 holds (or more precisely, we show only the first part of Axiom 1). For Axiom 3 to hold we cannot take anything better than $d_Z$ to be essentially the Euclidean metric on $Y$, but then Axiom 4 is not satisfied.
We will not actually need the result that Axiom 1 holds since we cannot use this axiomatic method to prove any dimension estimate. On the other hand, the fact that the measure of tubes is a fixed power of the radius is a very natural condition so we show anyway that it is satisfied. 

\textbf{Axiom 1}:
We show that with the choice of $r_R$ made in \eqref{rR} for every $u \in B_{n-1}(0, r_R)$ and every $a \in B_n(0,R)$ we have $\mathcal{L}^n(T^\delta_u(a)) \approx \delta^{Q-2}$. 
Since $\mathcal{L}^n$ is the Haar measure of the group, we have $\mathcal{L}^n(T^\delta_u(a))= \mathcal{L}^n (\tau_a(T^\delta_u))=\mathcal{L}^n(T^\delta_u)$ for every $a \in B_n(0,R)$ so it is enough to show that $\mathcal{L}^n(T^\delta_u) \approx \delta^{Q-2}$. This follows from the following lemma.

\begin{lem}\label{Ax1m2>1}
There exist two constants $0<c<1<C$ (depending only on $n$, $R$) such that for every $u \in B_{n-1}(0,r_R)$ there exist $N$ points $p_j$, $j=1, \dots, N$, with $\delta^{-2}/\sqrt{2} \le N \le 2 \delta^{-2}$, such that 
\begin{equation*}
\bigcup_{j=1}^N B_\infty(p_j, c \delta) \subset T^\delta_u \subset \bigcup_{j=1}^N B_\infty (p_j, C \delta)
\end{equation*}
and the balls $B_\infty( p_j, c \delta)$ are pairwise disjoint.
\end{lem}

Observe that this implies that $\mathcal{L}^n(T^\delta_u) \approx \delta^{Q-2}$ since
\begin{align*}
\mathcal{L}^n(T^\delta_u)  \le \sum_{j=1}^N \mathcal{L}^n(B_\infty(p_j, C \delta)) \approx N (C\delta)^Q \approx \delta^{Q-2}
\end{align*}
and
\begin{align*}
\mathcal{L}^n(T^\delta_u) \ge   \sum_{j=1}^N \mathcal{L}^n(B_\infty(p_j, c \delta)) \approx N (c \delta)^Q \approx \delta^{Q-2}.
\end{align*}

\begin{proof}
Let $t_1=0$, $t_2= \delta^2$, $\dots$, $t_j=t_{j-1}+\delta^2=(j-1) \delta^2$, where $j=1, \dots N$ and $(N-1) \delta^2 \le \frac{1}{\sqrt{1+|u|^2}} \le N \delta^2$ (then $\delta^{-2}/\sqrt{2}\le N \le 2 \delta^{-2}$). Let $p_j=t_j(u,1) \in I_u$. 

First observe that for any $0<c<1$ we have $B_\infty (p_j, c \delta) \subset T^\delta_u$ since $p_j \in I_u$.
Now we want to find $c$ small enough such that the balls $B_\infty( p_j, c \delta)$ are pairwise disjoint. It suffices to show that $B_\infty (p_j , c \delta) \cap B_\infty(p_{j+1}, c \delta) = \emptyset$ for every $j=1, \dots, N-1$.
Let $q=[q^1,q^2] \in B_\infty (t_j(u,1), c \delta)$, which means that
\begin{align}\label{qintj}
\max  \{ |q^1-t_j u^1|, \epsilon |q^2-t_j(u^2,1)-P( t_j u^1, q^1)|^{1/2} \} \le c \delta.
\end{align}
Observe that
\begin{align*}
|q^2&-t_{j+1}(u^2,1)-P( t_{j+1} u^1, q^1)|\\
=&|q^2-(t_j+\delta^2)(u^2,1)-P( (t_j + \delta^2) u^1, q^1)|\\
 \ge &|(t_j+\delta^2)(u^2,1)+P( (t_j + \delta^2) u^1, q^1)- t_j(u^2,1)-P(t_j u^1, q^1)|\\ &- |q^2-t_j(u^2,1)-P(t_j u^1, q^1)|\\
=& |\delta^2(u^2,1)+\delta^2 P(u^1,q^1)|- |q^2-t_j(u^2,1)-P(t_j u^1, q^1)|.
\end{align*}
We have $|q^2-t_j(u^2,1)-P( t_j u^1, q^1)| \le c^2 \delta^2/ \epsilon^2$ by \eqref{qintj} and
\begin{align*}
&\delta^2 |(u^2,1)+P(u^1,q^1)| \ge \delta^2(|(u^2,1)|-|P(u^1, q^1)|),\\
&|(u^2,1)| \ge 1,\\
&|P(u^1,q^1)|= \sqrt{\sum_{k=m_1+1}^n P_k(u^1,q^1)^2} \le C_{R,n} |u| \le C_{R,n} r_R
\end{align*}
by the same calculation as in \eqref{Pau}.
Hence
\begin{align*}
&|q^2-t_{j+1}(u^2,1)-P(t_{j+1}u^1,q^1)| \\
& \ge \delta^2(1-C_{R,n} r_R) - \frac{c^2 \delta^2}{\epsilon^2}.
\end{align*}
We can choose $0<c<\epsilon \sqrt{\frac{1-C_{R,n}r_R}{2}}$ (where $1-C_{R,n}r_R >0$ since $r_R<1/(2C_{R,n}) $). Then we have
\begin{equation*}
\epsilon |q^2-t_{j+1}(u^2,1)-P(t_{j+1} u^1,q^1)|^{1/2} > c \delta,
\end{equation*}
which means 
\begin{equation*}
d_\infty(q,t_{j+1}(u,1))=\max \{|q^1-t_{j+1}u^1|, \epsilon |q^2-t_{j+1}(u^2,1)-P(t_{j+1}u^1,q^1)|^{1/2} \} >c \delta,
\end{equation*}
thus $q \notin B_\infty(t_{j+1}(u,1), c \delta)$. Hence $B_\infty(p_j,c \delta) \cap B_\infty(p_{j+1},c \delta)=\emptyset$.

Next we want to find $C>1$ such that
\begin{equation*}
T^\delta_u \subset \bigcup_{j=1}^N B_\infty (t_j(u,1), C \delta).
\end{equation*}
Let $q \in T^\delta_u$. There exists $0 \le t \le \frac{1}{\sqrt{1+|u|^2}}$ such that $d_\infty(q,t(u,1)) \le \delta$. Since $|t_{j+1}-t_j|=\delta^2$, there is $j \in \{1, \dots, N \}$ such that $|t-t_j|\le \delta^2$. We have
\begin{align*}
d_\infty(t(u,1), t_j(u,1))= \max \{ |t-t_j| |u^1|, \epsilon |t(u^2,1)-t_j(u^2,1)- P(t_j u^1, t u^1)|^{1/2} \}.
\end{align*}
Since $|t-t_j| |u^1| \le \delta^2 r_R$, $P(t_j u^1, t u^1)=0$ and $\epsilon |t-t_j|^{1/2} |(u^2,1)|^{1/2} \le \epsilon \delta (r_R^2+1)^{1/4}$, it follows that
\begin{equation*}
d_\infty(t(u,1), t_j(u,1)) \le \max\{ \delta^2 r_R, \epsilon \delta (1+r_R^2)^{1/4} \}. 
\end{equation*} 
Thus
\begin{align*}
d_\infty(q,t_j(u,1)) \le d_\infty(q,t(u,1))+d_\infty(t(u,1),t_j(u,1)) \le \delta+\delta \epsilon (1+r_R^2)^{1/4}.
\end{align*}
Choosing $C=1+\epsilon (1+r_R^2)^{1/4}$, we have $d_\infty(q,t_j(u,1)) \le C \delta$, hence $q \in B_\infty(p_j,C \delta)$.
\end{proof}

\textbf{Axiom 3}:
Let us look now at what diameter estimate we can prove (to see that we need to take $d_Z$ to be essentially the Euclidean metric on $Y$).

Suppose that we can find $u=[u^1,u^2], v=[v^1,v^2] \in Y$ (with $u^2,v^2 \in \mathbb{R}^{m_2-1}$) such that $|u^1-v^1| \approx \delta \approx |u^2-v^2|$ and $ \frac{1}{2} P(u^1,v^1)=u^2-v^2$. 
Then
\begin{eqnarray*}
\begin{split}
d_\infty \left(\frac{1}{2} (u,1), \frac{1}{2} (v,1)\right)& = \max \left\{ \frac{1}{2} \left| u^1-v^1 \right|, \epsilon \left|\frac{u^2}{2}- \frac{v^2}{2} - \frac{1}{4} P(u^1,v^1)\right|^{1/2} \right\} \\
&= \max \left\{ \frac{1}{2} |u^1-v^1|, 0 \right\} \approx \delta.
\end{split}
\end{eqnarray*}
Thus $0, \frac{1}{2} (u,1), \frac{1}{2} (v,1) \in {T}^{K\delta}_u \cap {T}^{K\delta}_v$ for some $K>0$, where ${T}^{K\delta}_u$ is the $K \delta$ neighbourhood of $I_u$. This implies that $\text{diam}_E({T}^{K\delta}_u \cap {T}^{K\delta}_v) \approx 1$.
Hence we cannot get anything better than
\begin{equation}\label{Ax3m2>1}
\text{diam}_E( {T}^{K\delta}_u \cap {T}^{K\delta}_v) \lesssim \frac{\delta}{|u-v|}.
\end{equation}
Thus in order for Axiom 3 to hold we need to take $d_Z$ to be essentially the Euclidean metric in $\mathbb{R}^{n-1}$, that is $S=n-1$. 

To find $u$ and $v$ as above, we need to have $u=(u_1, \dots, u_{n-1})$, $v=(v_1, \dots, v_{n-1})$ such that $|u^1-v^1| \approx \delta \approx |u^2-v^2|$ and the following is satisfied:
\begin{eqnarray}\label{system}
\begin{split}
& \sum_{1 \le l < i \le m_1} b^j_{l,i} (u_l v_i-u_i v_l)= 2(u_j-v_j) \quad \mbox{for all} \ j=m_1+1, \dots, n-1,\\
& \sum_{1 \le l < i \le m_1} b^n_{l,i} (u_l v_i-u_i v_l)=0.
\end{split}
\end{eqnarray} 
We do not know if we can find such $u$, $v$ in any Carnot group of step $2$ with $m_2>1$, but at least it is possible in some cases as we will see now.

Suppose that there exist $1 \le k < h \le m_1$ and $m_1+1\le J \le n-1$ such that 
\begin{equation}\label{condition}
b^n_{k,h}=0 \quad \mbox{and} \quad b^J_{k,h} \neq 0. 
\end{equation}
Let $u_k \neq 0$ and $u_i=0$ for every $ i \in \{1, \dots, m_1 \} \setminus \{k\}$. Let $v_k=u_k+\delta$, $v_h=\delta$ and $v_i=0$ for every $i \in \{1, \dots, m_1 \} \setminus \{k,h\} $. Then $|u^1-v^1| \approx \delta$ and
\begin{equation*}
\sum_{1 \le l < i \le m_1} b^n_{l,i} (u_l v_i-u_i v_l)= b^n_{k,h} u_k v_h =0
\end{equation*}
thus the last equation in \eqref{system} holds. For $j=m_1+1, \dots, n-1$, let $v_j=0$ and
\begin{equation*}
u_j= \frac{1}{2} \sum_{1 \le l < i \le m_1} b^j_{l,i} (u_l v_i-u_i v_l) = \frac{1}{2} b^j_{k,h} u_k v_h = \frac{1}{2} b^j_{k,h} u_k \delta ,
\end{equation*}
which means that all the equations in \eqref{system} are satisfied.
We can choose $u_k$ small enough so that $u, v \in Y$. Since $b^J_{k,h} \neq 0$, we have at least $u_J \neq 0$, hence $|u^2-v^2| \approx \delta$. Hence in this case we found $u,v$ as desired, which implies that in general we cannot take $d_Z$ to be better than the Euclidean metric on $Y$.

Observe that condition \eqref{condition} is satisfied for example in free Carnot groups of step $2$ (see Section 3.3 in \cite{BLU}). To define these, it is convenient to use the following notation. For a point $p=[p^1,p^2]=(x_1, \dots, x_n)$, we denote the coordinates of $p^2=(x_{m_1+1}, \dots, x_n)$ by $p_{l,i}$, where $(l,i) \in \mathcal{I}$ and
\begin{equation}
\mathcal{I}=\{ (l,i): 1\le l < i \le m_1\}.
\end{equation}
Then $\# \mathcal{I}= m_1(m_1-1)/2$ and we set this to be $m_2$. The composition law is given by $p \cdot q$, where
\begin{align*}
&(p \cdot q)_k=x_k+y_k, \quad k=1, \dots, m_1,\\
&(p \cdot q)_{l,i}= p_{l,i}+q_{l,i} + \frac{1}{2} (x_i y_l-x_l y_i), \quad (l,i) \in \mathcal{I}.
\end{align*}
Thus in this case each polynomial $P_j$, where $j=(l,i)$ for some $(l,i) \in \mathcal{I}$, has coefficients all zero except $b^j_{l,i}= \frac{1}{2}$.

Another example of a Carnot group satisfying condition \eqref{condition} is the quaternionic Heisenberg group (see for example Section 2.1 in \cite{Vas}). This is identified with $\mathbb{R}^n=\mathbb{R}^{4N+3}$ equipped with the group operation
\begin{align*}
[p^1,p^2] \cdot [q^1,q^2]= [p^1+q^1, p^2+q^2+P(p^1,q^1)],
\end{align*}
where $p^1,q^1 \in \mathbb{R}^{4N}$, $p^2,q^2 \in \mathbb{R}^{3}$ and $P=(P_{4N+1},P_{4N+2},P_{4N+3})$,
\begin{align*}
&P_{4N+1}(p^1,q^1)= 2 \left( \sum_{i=1}^N (q_i p_{N+i} - p_i q_{N+i})+ \sum_{i=1}^N (q_{3N+i} p_{2N+i}-p_{3N+i}q_{2N+i} ) \right),\\
&P_{4N+2}(p^1,q^1)=2 \left( \sum_{i=1}^N(q_i p_{2N+i}-q_{2N+i} p_i) + \sum_{i=1}^N(q_{N+i}p_{3N+i}-q_{3N+i}p_{N+i}) \right),\\
&P_{4N+3}(p^1,q^1)=2 \left( \sum_{i=1}^N(q_i p_{3N+i}-p_i q_{3N+i})+\sum_{i=1}^n(q_{2N+i}p_{N+i}-p_{2N+i}q_{N+i})  \right).
\end{align*}
Hence $m_1=4N$, $m_2=3$, $Q=4N+6$. Here for example $b^{4N+3}_{1,N+1}=0$ and $b^{4N+1}_{1,N+1}=2$, which means that \eqref{condition} holds with $k=1$, $h=N+1$, $J=4N+1$.

It is more tedious to verify that condition \eqref{condition} is satisfied also in another Iwasawa group, the first octonionic Heisenberg group (see Section 2.1 in \cite{Vas}). This is modeled as $\mathbb{O} \times \mbox{Im} (\mathbb{O}) \cong \mathbb{R}^8 \times \mathbb{R}^7$, where $\mathbb{O}$ denotes the octonions. These form the eight-dimensional real vector space spanned by the indeterminates $e_j$, $j=0, \dots, 7$ and equipped with a certain product rule that is explained in the table in Section 2.1.3 in  \cite{Vas} ($e_0=1$ is the identity element). An element in $\mathbb{O}$ can be expressed as $z=z_0+\sum_{j=1}^7 z_j e_j$ and $\mbox{Im}(z) =\sum_{j=1}^7 z_j e_j$. The first octonionic Heisenberg group is $\mathbb{O} \times \mbox{Im} (\mathbb{O})$ equipped with the group product
\begin{equation*}
(z, \tau) \cdot (z',\tau')= (z+z',\tau+\tau'+2 \mbox{Im}(\bar{z'}z)),
\end{equation*}
where $\bar {z'}=z'_0-\mbox{Im}(z')$. It can be verified that for example $b^9_{1,2}=2$ and $b^{15}_{1,2}=0$ hence \eqref{condition} holds with $J=9=m_1+1$ and $k=1$, $h=2$. 

Other examples of Carnot groups of step $2$ satisfying \eqref{condition} can be found in Chapter 3 in \cite{BLU}: in Remark 3.6.6 there are two examples of Carnot groups of Heisenberg type on $\mathbb{R}^6= \mathbb{R}^4 \times \mathbb{R}^2$ and on $\mathbb{R}^7=\mathbb{R}^4 \times \mathbb{R}^3$; in Exercise 6 (page 179) there is one example of a Kolmogorov type group on $\mathbb{R}^5=\mathbb{R}^3 \times \mathbb{R}^2$.

\textbf{Axiom 4}: Now we want to show that Axiom $4$ does not hold with $d_Z$ equal to the Euclidean metric on $Y$.
For Axiom $4$ to hold, we would need to be able to cover $T^\delta_u$ with $N$ tubes $T^{K\delta}_v(b_k)$, where $K>0$ and $N$ are independent of $\delta$, when $|u-v| \le \delta$. We will see that this cannot happen in general.

Let $u=[0,0]$ and $v=[0,(\delta,0, \dots, 0)]$, thus $|u-v|=|v|=\delta$. We want to show that we need $\gtrsim \delta^{-1/2} $ tubes ${T}^{K\delta}_v(b_k)$ to cover $T ^\delta_0$. Consider a point $p=(0, \dots,0, a) \in I_0$, $\delta^{1/2} \le a \le 1$. We have
\begin{eqnarray*}
\begin{split}
d_\infty(p,I_v)&=\inf_{t \in [0,1/\sqrt{\delta^2+1}]} d_\infty (p , t(v,1))\\ &= \inf_{t \in [0,1/\sqrt{\delta^2+1}]} \max \{ 0, \epsilon ((t\delta)^2+(t-a)^2)^{1/4} \} \\ &=\inf_{t \in [0,1/\sqrt{\delta^2+1}]}  \epsilon ((t\delta)^2+(t-a)^2)^{1/4}.
\end{split}
\end{eqnarray*}  
The minimum is attained when $t=\frac{a}{\delta^2+1}$, thus
\begin{equation}\label{Ax4m2>1}
d_\infty(p,I_v)= \epsilon \left( \frac{a^2 \delta^{2}}{(\delta^2+1)^2}+ \frac{a^2 \delta^4}{(\delta^2+1)^2} \right)^{1/4} \gtrsim a^{1/2} \delta^{1/2} \ge \delta^{3/4} > \delta.
\end{equation}
Moreover, we cannot find $K$ independent of $\delta$ such that $\delta^{3/4} \le K \delta$.
Hence $p \notin {T}^{K\delta}_v$ for any $K$ (for small $\delta>0$). On the other hand, $0 \in {T}^{K\delta}_v$. Thus for every $p \in I_0$ such that $|p| \ge\delta^{1/2}$ we have that $p \notin {T}^{K\delta}_v$, which implies that ${T}^{K\delta}_v$ covers only a piece of $I_0$ of Euclidean length $< \delta^{1/2}$. To cover $T^\delta_0$ we need to cover at least $I_0$ and we will now see that we need at least $\gtrsim \delta^{-1/2}$ tubes ${T}^\delta_v(b_j)$ to do so.

If $p^1=q^1$ for any two points $p,q \in \mathbb{G}$, then for any $y=[y^1,y^2] \in \mathbb{G}$ we have
\begin{align*}
p \cdot y &= [p^1+y^1, p^2+y^2+P(p^1,y^1)]\\
&= [q^1+y^1, q^2+y^2+P(q^1,y^1)] + [0,p^2-q^2]\\
&= q \cdot y+ [0,p^2-q^2].
\end{align*}
If we let $b_k=[b_k^1,b_k^2], b_i=[b_i^1,b_i^2] \in I_0$, we have $b_k^1=b_i^1=0$ hence
\begin{align*}
{T}^{K\delta}_v(b_i)= \tau_{b_i}({T}^{K\delta}_v)=\tau_{b_k}({T}^{K\delta}_v) + [0, b_i^2-b_k^2]= {T}^{K\delta}_v(b_k)+ [0, b_i^2-b_k^2].
\end{align*} 
Thus the tube ${T}^{K\delta}_v(b_k)$ is obtained by translating ${T}^{K\delta}_v(b_i)$ in the Euclidean sense.  It follows that each ${T}^{K\delta}_v(b_k)$ with $b_k\in I_0$ covers a piece of $I_0$ of Euclidean length $< \delta^{1/2}$. Thus we need $> \delta^{-1/2}$ of such tubes to cover $I_0$, which means that Axiom 4 does not hold. Observe that it would not help to take longer tubes $\tilde{T}^{K \delta}_v(b_k)$ ($K\delta$ neighbourhoods of $\tilde{I}_v(b_k)$, which are segments containing $I_v(b_k)$ and having double length).

Hence when $m_2>1$ we cannot use this axiomatic setting to obtain estimates for the Hausdorff dimension of LT-Kakeya sets with respect to $d_\infty$.

\begin{rem}\label{classicKak}(\textbf{Classical Kakeya sets})\\
Almost in the same way we can see that we cannot obtain estimates for the Hausdorff dimension (with respect to $d_\infty$) of the classical Kakeya sets when $m_2>1$.

In this case we can define $I_u$ as in \eqref{Iu} and $T^\delta_u$ as in \eqref{tubedinfty}, whereas we let for $a \in B_n(0,R)$ and $u \in B_{n-1}(0,r_R)$,
\begin{equation*}
I_u(a)=I_u+a= \{[tu+a^1,t+a_n]: 0 \le t \le \frac{1}{\sqrt{1+|u|^2}} \}
\end{equation*}
and
\begin{equation*}
T^\delta_u(a)=\{p: d_\infty(p,I_u(a)) \le \delta\}.
\end{equation*}

A bounded Kakeya set $K \subset B_n(0,R) \subset \mathbb{R}^n$ contains a unit segment in every direction, thus in particular for every $u \in B_{n-1}(0,r_R)$ there exists $a \in B_n(0,R)$ such that $I_u(a) \subset K$.

Let us now see that also in this case the measure of the tubes is $\approx \delta^{Q-2}$, Axiom 3 holds if $d_Z$ is essentially the Euclidean metric and Axiom 4 does not hold.

\textbf{Axiom 1}: For every $u \in B_{n-1}(0,r_R)$ and every $a \in B_n(0,R)$ we have $\mathcal{L}^n(T^\delta_u(a)) \approx \delta^{Q-2}$. Indeed, in this case we can prove the following version of Lemma \ref{Ax1m2>1}.

\begin{lem}
There exist two constants $0<c<1<C$ (depending only on $n$, $R$) such that for every $a \in B_n(0,R)$ and every $u \in B_{n-1}(0,r_R)$ there exist $N$ points $p_j$, $j=1, \dots, N$, with $\delta^{-2}/\sqrt{2} \le N \le 2 \delta^{-2}$, such that 
\begin{equation*}
\bigcup_{j=1}^N B_\infty(p_j, c \delta) \subset T^\delta_u(a) \subset \bigcup_{j=1}^N B_\infty (p_j, C \delta)
\end{equation*}
and the balls $B_\infty( p_j, c \delta)$ are pairwise disjoint.
\end{lem}

\begin{proof}
The proof is the same as that of Lemma \ref{Ax1m2>1} if we define $p_j=t_j(u,1)+a \in I_u(a)$, where $t_j=(j-1) \delta^2$, $j=1, \dots, N$.
\end{proof}

\textbf{Axiom 3}: The tubes starting from the origin $T^\delta_u$ are the same as those defined in \eqref{tubedinfty}, thus we can use the same argument used to show \eqref{Ax3m2>1} to see that in general we cannot take $d_Z$ to be anything better than essentially the Euclidean metric.

\textbf{Axiom 4}: Also here we can show that Axiom 4 is not satisfied with the same example used previously, that is $u=[0,0]$, $v=[0, (\delta,0,\dots, 0)]$. We can see that we need at least $\delta^{-1/2}$ tubes $\tilde{T}^{K\delta}_v(b_k)$ (which is the $K\delta$ neighbourhood of $I_v(b_k)=\{t(v,1)+b_k: 0 \le t\le \frac{2}{\sqrt{1+\delta^2}} \}$) to cover $T^\delta_0$ for any constant $K$ independent of $\delta$.\\
Indeed, as was seen in \eqref{Ax4m2>1}, $\tilde{T}^{K\delta}_v$ covers only a piece of $I_0$ of Euclidean length $< \delta^{1/2}$. Moreover, if $b_k \in I_0$ and $q=(0,\dots,0,a) \in I_0$ such that $|b_k-q| \ge \delta^{1/2}$, we can use a similar calculation as in \eqref{Ax4m2>1} to show that
\begin{equation*}
d_\infty(q,I_v(b_k)) \ge \delta^{3/4}.
\end{equation*}
Hence also any tube $\tilde{T}^{K\delta}_v(b_k)$ with $b_k \in I_0$ covers only a piece of $I_0$ of Euclidean length $< \delta^{1/2}$, which means that we need $\gtrsim \delta^{-1/2}$ of them to cover the whole $I_0$.
\end{rem}

\subsection{Case II: $m_2=1$}

We will prove the following theorem. 

\begin{theorem}\label{m21}
Let $\mathbb{G}$ be a Carnot group of step 2 whose second layer has dimension 1. The Hausdorff dimension of any bounded LT-Kakeya set in $\mathbb{G}$ with respect to any homogeneous metric is $\ge \frac{n+3}{2}$ and the Hausdorff dimension of any bounded Kakeya set with respect to any homogeneous metric is $\ge \frac{n+4}{2}$.
\end{theorem}

\textbf{Proof of Theorem \ref{m21} for LT-Kakeya sets}:
Let us first look at the case of LT-Kakeya sets and show that the Axioms 1-4 are satisfied hence Bourgain's method yields the required lower bound.

When $m_2=1$ we have $Q=m_1+2 m_2= n-1+2=n+1$. We denote the points by $p=[p^1,p_n]=(p_1, \dots, p_n)$, where $p^1 \in \mathbb{R}^{n-1}$. The group operation thus becomes
\begin{equation*}
p  \cdot q= [p^1 +q^1, p_n+q_n+P(p^1,q^1)],
\end{equation*}
where
\begin{equation*}
P(p^1,q^1)= \sum_{1 \le l< i \le n-1} b_{l,i} (p_l q_i-p_i q_l).
\end{equation*}
Here we let $Z=Y=B_{n-1}(0,r_R)$, where
\begin{equation}\label{rR1}
r_R < \sqrt{1+C_{R,n}^2} - C_{R,n},
\end{equation}
and
\begin{equation}\label{CRn1}
C_{R,n}=R \sqrt{(n-1)(n-2) \max_{1\le l<i\le n-1} |b_{l,i}|^2}.
\end{equation}
Note that $r_R$ is smaller than the one chosen in \eqref{rR}, whereas $C_{R,n}$ is essentially the same constant as in \eqref{CRn}. We let $d_Z$ be the Euclidean metric on $Y$, which is $(n-1)$- Ahlfors regular, so $S=n-1$. 

Observe that when $m_2=1$ the tubes $T^\delta_u(a)$ defined in \eqref{tubeTau} are essentially the same as the Euclidean tubes. The rough idea is that for every point $p$ there is only one direction (which depends on $p$) in which balls $B_\infty(p,r)$ do not behave like Euclidean balls and with the above choice of $r_R$ this direction is close to the direction of the segments $I_u(a)$, $u \in Y$ (when $p \in B_n(0,R)$). More precisely, the following holds.

\begin{lem}\label{lemang}
For every $q \in B_n(0,R)$ the angle between the horizontal hyperplane $H_q$ passing through $q$ (that is, the left translation by $q$ of the $x_1, \dots, x_{n-1}$-hyperplane) and the $x_1, \dots, x_{n-1}$-hyperplane is at most
\begin{equation*}
\theta_{R,n}:=\arccos \frac{1}{\sqrt{1+C_{R,n}^2}}.
\end{equation*}
\end{lem}

Note that by the angle between two hyperplanes we mean the angle between their normal vectors.

\begin{proof}
The horizontal hyperplane $H_q$ passing through $q$ is the set of points $p=[p^1,p_n]$ such that
\begin{equation}\label{Hq}
p_n=q_n+P(q^1,p^1)=q_n+ \langle B q^1,p^1 \rangle,
\end{equation}
where 
\begin{equation*}
Bq^1= \left(  \sum_{1 \le l \le n-1, \ l \neq i} b_{l,i} q_l \right)_{1 \le i \le n-1},
\end{equation*}
$q^1=(q_1, \dots, q_{n-1})$, $b_{l,i}=-b_{i,l}$.
Hence a normal vector to $H_q$ is
\begin{equation*}
n_1=[-Bq^1,1].
\end{equation*}
Since a normal vector to the $x_1, \dots, x_{n-1}$-hyperplane is $n_2=[0,1]$, the angle between them is
\begin{align*}
\theta=\arccos \frac{\langle n_1,n_2\rangle}{|n_1||n_2|}= \arccos \frac{ 1}{\sqrt{1+|Bq^1|^2}}.
\end{align*}
We have
\begin{align*}
|Bq^1|^2&= \sum_{i=1}^{n-1} \left(  \sum_{1 \le l \le n-1, \ l \neq i} b_{l,i} q_l \right)^2 \\
& \le( (n-1) (n-2) \max_{l,i} |b_{l,i}| |q^1|)^2\\
&\le C_{R,n}^2,
\end{align*}
where $C_{R,n}$ is defined in \eqref{CRn1}.
Hence
\begin{equation*}
\theta \le \arccos \frac{1}{\sqrt{1+C_{R,n}^2}}= \theta_{R,n}.
\end{equation*}
\end{proof}

Let us see more precisely how the tubes $T^\delta_u(a)$ compare with Euclidean tubes.

\begin{lem}\label{tubeessEucl}
There exist $0<c<C< \infty$ (depending only on $n$ and $R$) such that for every $a \in B_n(0,R)$ and every $u \in B_{n-1}(0,r_R)$
\begin{equation}\label{essEucl1}
T^{O,c \delta}(I_u(a)) \subset T^\delta_u(a) \subset T^{E,C \delta} (I_u(a)),
\end{equation}
where $T^{E,C \delta}(I_u(a))  $ denotes the $C \delta$ neighbourhood of $I_u(a)$ in the Euclidean metric, whereas $T^{O,c \delta}(I_u(a))$ denotes the set of points $q \in \mathbb{R}^n$ such that a line through $q$ orthogonal to $I_u(a)$ intersects $I_u(a)$ in a point $\bar{q}$ and $|q-\bar{q}| \le c \delta$.
\end{lem}

Observe that $T^{O,c \delta}(I_u(a))$ is not exactly the $c \delta$ neighbourhood of $I_u(a)$: it does not contain all the points whose distance from the extremal points of $I_u(a)$ is $\le c \delta$ (see Figure \ref{TEO}). On the other hand, these tubes $T^{O,c \delta}(I_u(a))$ are often used to define the Kakeya maximal function, see for example Definition 22.1 in \cite{Mattila}.

\begin{figure}[H]
    \centering
    \includegraphics[width=0.35\textwidth]{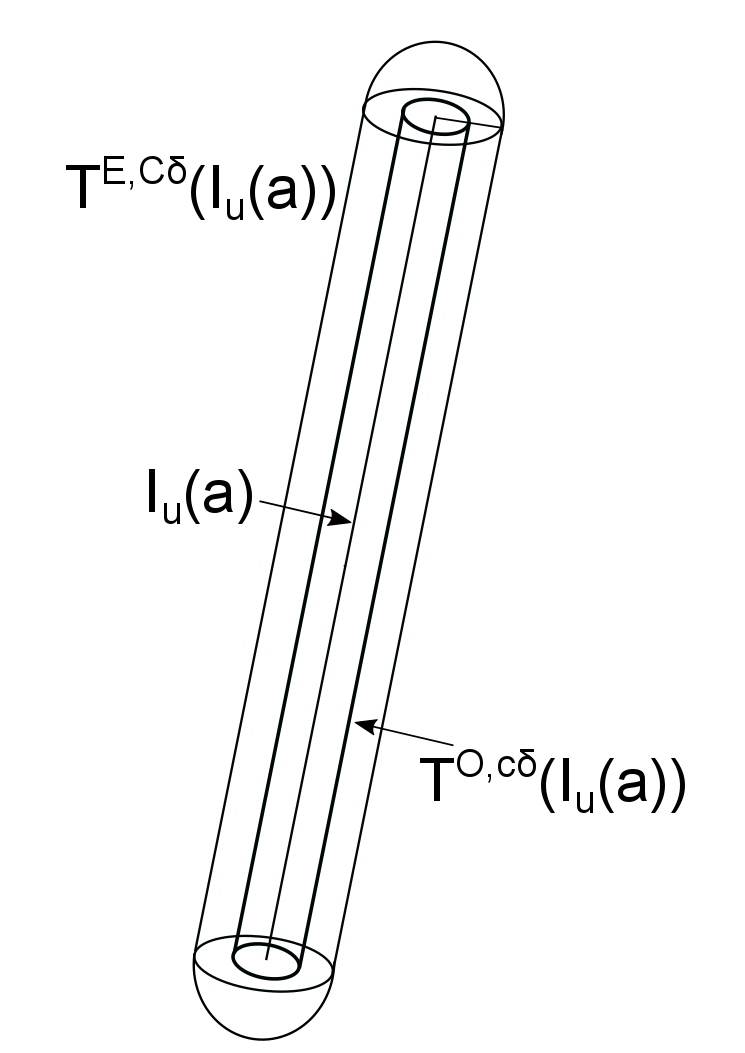}
    \caption{The tubes  $T^{O,c \delta}(I_u(a))$ and  $T^{E,C \delta}(I_u(a))$ in $\mathbb{R}^3$}
    \label{TEO}
\end{figure}

\begin{proof}
One inclusion is easy since by Lemma \ref{inftyECR} there exists $C=C_R$ such that for every $p,q \in B_n(0,R)$ we have
\begin{equation*}
|p-q| \le C d_\infty (p,q).
\end{equation*}
Thus $T^\delta_u(a) \subset T^{E,C \delta} (I_u(a)) $.

For the other implication, we want to find $c$ such that $T^{O,c \delta}(I_u(a)) \subset T^\delta_u(a)$. Let $q=[q^1,q_n] \in  T^{O,c \delta}(I_u(a))$. This means that $|q-\bar{q}|\le c \delta$, where $\bar{q} \in I_u(a)$ is such that the line containing $q$ and $\bar{q}$ is orthogonal to $I_u(a)$. The points in $I_u(a)$ are given by
\begin{equation}\label{gammat}
\gamma(t)= [a^1+tu, a_n+t+t P(a^1,u)], 
\end{equation} 
where $a=[a^1,a_n]$, $ 0 \le t \le \frac{1}{\sqrt{1+|u|^2}}$. Let $H_q$ be the horizontal hyperplane passing through $q$, that is the left translation by $q$ of the $x_1, \dots, x_{n-1}$-hyperplane. As seen in \eqref{Hq}, it has the form
\begin{equation*}
H_q= \{ p=[p^1, p_n]: p_n=q_n+P(q^1,p^1)] \}.
\end{equation*}
Let $\gamma(\bar{t})= H_q \cap I_u(a)$, that is
\begin{equation}\label{bart}
\bar{t}= \frac{q_n-a_n+P(q^1,a^1)}{1+P(a^1,u)-P(q^1,u)}.
\end{equation}
Let $\theta$ be the angle between $I_u(a)$ and the segment joining $q$ to $\gamma(\bar{t})$ (see Figure \ref{fig4}). If $\bar{q} \in H_q$ then $\gamma(\bar{t})=\bar{q}$ and $\theta=\pi/2$ is the angle between $I_u(a)$ and $H_q$. In general $\theta$ is greater or equal to the angle between $I_u(a)$ and $H_q$. By Lemma \ref{lemang} the angle between $H_q$ and the $x_1, \dots, x_{n-1} $-hyperplane is $\le \theta_{R,n}$. On the other hand, the angle between $I_u(a)$ and the $x_1, \dots, x_{n-1}$-hyperplane is 
\begin{align*}
&\arccos \frac{|u|}{\sqrt{|u|^2+(1+P(a^1,u))^2}}\\
& \ge \arccos \frac{r_R}{|1+P(a^1,u)|}\\
& \ge \arccos \frac{r_R}{1-|P(a^1,u)|}.
\end{align*} 
\begin{figure}[H]
    \centering
    \includegraphics[width=0.3\textwidth]{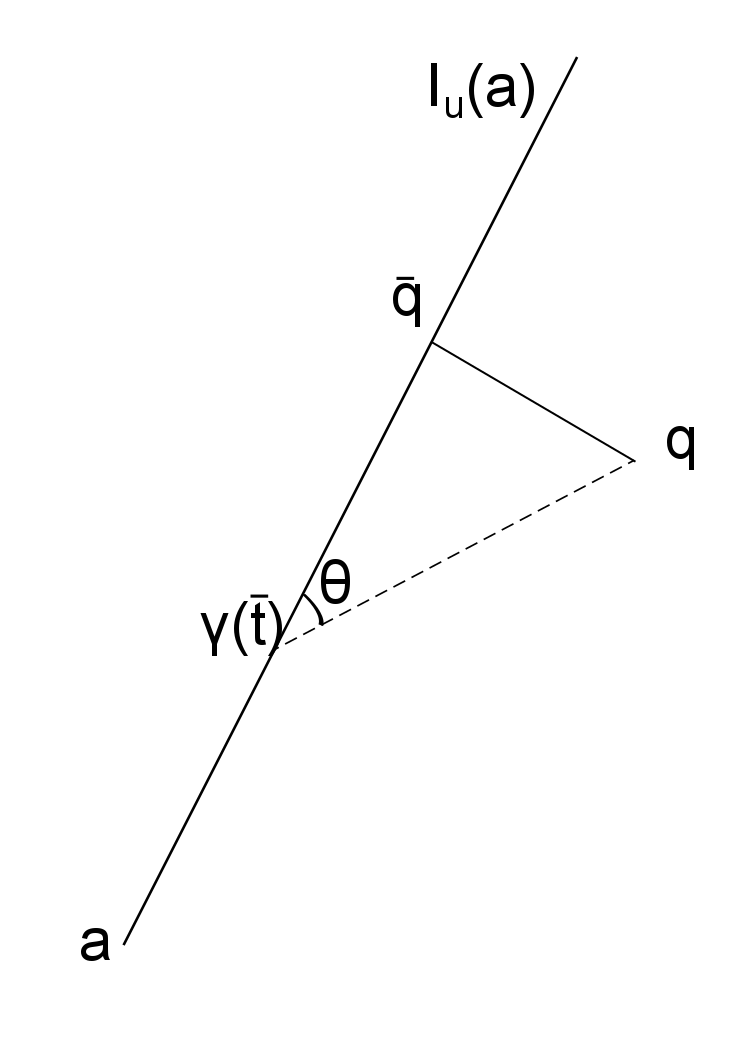}
    \caption{Geometric situation in the proof of Lemma \ref{tubeessEucl}}
    \label{fig4}
\end{figure}
 Since $|P(a^1,u)| \le C_{R,n}r_R$ by a similar calculation as in \eqref{Pau}, we have
\begin{equation}\label{thetabar1}
\arccos \frac{r_R}{1-|P(a^1,u)|} \ge \arccos \frac{r_R}{1-C_{R,n}r_R} =: \bar{\theta}_{R,n}.
\end{equation}
By the choice of $r_R$ made in \eqref{rR1}, $  \bar{\theta}_{R,n}> \theta_{R,n}$.
Thus we have $\theta \ge \bar{\theta}_{R,n} - \theta_{R,n} >0$. Hence
\begin{equation}\label{qgammat}
|q-\gamma(\bar{t}) | = \frac{|q-\bar{q}|}{\sin \theta} \le c_{R,n} |q-\bar{q}| \le c_{R,n} c \delta= \delta
\end{equation}
if we choose $c = 1/c_{R,n}$, where $c_{R,n}=1/\sin(\bar{\theta}_{R,n} -  \theta_{R,n})$. It follows that 
\begin{align*}
d_\infty(q, \gamma(\bar{t}))= \max \{ |q^1-\bar{t}u-a^1|, \epsilon |a_n+\bar{t}+ \bar{t}P(a^1,u)-q_n-P(q^1, \bar{t} u+a^1)|^{1/2} \} \le \delta
\end{align*}
because by \eqref{qgammat} we have
\begin{equation*}
|q^1-\bar{t}u-a^1| \le |q-\gamma(\bar{t}) | \le \delta
\end{equation*}
and from the expression \eqref{bart} of $\bar{t}$ we have
\begin{equation*}
 |a_n+\bar{t}+ \bar{t}P(a^1,u)-q_n-P(q^1, \bar{t} u+a^1)|^{1/2}=0.
\end{equation*}
Hence 
\begin{align*}
d_\infty(q, I_u(a))=\inf_{t \in \left[0,\frac{1}{\sqrt{1+|u|^2}}\right]} d_\infty(q, \gamma(t)) \le d_\infty (q, \gamma(\bar{t})) \le \delta,
\end{align*}
which means $q \in T^\delta_u(a)$. Thus $T^{O,c\delta}(I_u(a)) \subset T^\delta_u(a)$ and \eqref{essEucl1} is proved.
\end{proof}

\textbf{Axiom 1}: From \eqref{essEucl1} it follows that $\mathcal{L}^n(T^\delta_u(a)) \approx \delta^{n-1}$ for every $a \in B_n(0,R)$ and every $u \in B_{n-1}(0,r_R)$. Moreover, if $A \subset T^\delta_u(a)$ then $A \subset T^\delta_u(a) \cap B_n(q,\text{diam}_E(A))$ for every $q \in A$. Using again \eqref{essEucl1} we have $A \subset T^{E,C\delta}(I_u(a)) \cap B_n(q,\text{diam}_E(A))$, which implies $\mathcal{L}^n(A) \lesssim \text{diam}_E(A) \delta^{n-1}$.
Hence Axiom 1 holds with $T=n-1$.

\textbf{Axiom 3}: We can use the diameter estimate for Euclidean tubes and Lemma \ref{tubeessEucl} to prove the following lemma, which states that Axiom 3 is satisfied.

\begin{lem}\label{ax3m21}
There exists $b=b_{R,n} >0$ such that for every $a, a' \in B_n(0,R) $, $u,v \in B_{n-1}(0, r_R)$,
\begin{equation}\label{Ax3m21}
\text{diam}_E(T^\delta_u(a) \cap T^\delta_v(a')) \le b \frac{  \delta}{|u-v|}.
\end{equation}
\end{lem}

\begin{proof}
Suppose that $T^\delta_u(a) \cap T^\delta_v(a') \neq \emptyset$, otherwise \eqref{Ax3m21} holds trivially.
Then there exists $q \in I_u(a)$ such that $d_\infty(q,I_v(a')) \le 2 \delta$. This means that $d_\infty (q,p) \le 2 \delta$ for some $p \in I_v(a')$.
Then we have
\begin{equation*}
I_u(a) \subset \tilde{I}_u(q) ,
\end{equation*}
where
\begin{eqnarray*}
\tilde{I}_u(q)= \{ [tu+q^1,t+q_n+tP(q^1,u)]: -\frac{1}{\sqrt{|u|^2+1}} \le t \le \frac{1}{\sqrt{|u|^2+1}} \}.
\end{eqnarray*}
Indeed, if $q = [\bar{t} u+a^1, \bar{t}+a_n+\bar{t}P(a^1,u)] \in I_u(a)$, for some $0 \le \bar{t} \le \frac{1}{\sqrt{|u|^2+1}}$, then
\begin{eqnarray*}
\begin{split}
\tilde{I}_u(q)= \{ &[tu+\bar{t} u+a^1, t+\bar{t}+a_n+\bar{t}P(a^1,u)+tP(\bar{t} u+a^1,u)]:\\
& -\frac{1}{\sqrt{|u|^2+1}} \le t \le \frac{1}{\sqrt{|u|^2+1}} \}\\
= \{ & [tu+\bar{t} u+a^1, t+\bar{t}+a_n+\bar{t}P(a^1,u)+t P(a^1,u)]:\\
& -\frac{1}{\sqrt{|u|^2+1}} \le t \le \frac{1}{\sqrt{|u|^2+1}} \}.
\end{split}
\end{eqnarray*}
If $ z \in I_u(a)$ then $z=[s u+a^1,s+a_n+sP(a^1,u)]$ for some $0 \le s \le \frac{1}{\sqrt{|u|^2+1}}$. Thus $z \in \tilde{I}_u(q)$ since
\begin{equation*}
z=[(s-\bar{t})u+\bar{t}u+a^1,(s-\bar{t}) + \bar{t}+a_n+(s-\bar{t})P(a^1,u)+\bar{t} P(a^1,u)]
\end{equation*} 
and $- \frac{1}{\sqrt{|u|^2+1}} \le s-\bar{t} \le \frac{1}{\sqrt{|u|^2+1}}$. Hence $I_u(a) \subset \tilde{I}_u(q)$, which implies that 
\begin{equation}\label{inc1}
T^\delta_u(a) \subset \tilde{T}^\delta_u(q),
\end{equation}
where $\tilde{T}^\delta_u(q)$ is the $\delta$ neighbourhood of $\tilde{I}_u(a)$. Similarly, since $p \in I_v(a')$ we have $I_v(a') \subset \tilde{I}_v(p)$, hence
\begin{equation}\label{inc3}
T^\delta_v(a') \subset \tilde{T}^\delta_v(p).
\end{equation}
On the other hand, we now want to see that 
\begin{equation}\label{inc}
 \tilde{T}^\delta_u(q) \subset \tilde{T}^{3\delta}_u(p),
\end{equation}
where $\tilde{T}^{3\delta}_u(p)$ is the $3\delta$ neighbourhood of $\tilde{I}_u(p)$. Let $\bar{q} \in \tilde{T}^\delta_u(q)$. Then there exists $z=[\bar{s}u+q^1,\bar{s}+q_n+\bar{s}P(q^1,u)] \in \tilde{I}_u(q)$ such that $d_\infty(\bar{q}, z) \le \delta$. Let $z'=[\bar{s}u+p^1,\bar{s}+p_n+\bar{s}P(p^1,u)] \in \tilde{I}_u(p)$. Then
\begin{eqnarray*}
\begin{split}
d_\infty(z,z')= \max \{& |\bar{s}u+q^1-\bar{s}u-p^1|,\\
 &\epsilon | \bar{s}+q_n+\bar{s}P(q^1,u)-\bar{s}-p_n-\bar{s}P(p^1,u)-P(\bar{s}u+q^1,\bar{s}u+p^1)|^{1/2} \}\\
 =\max \{ & |q^1-p^1|, \epsilon |q_n-p_n-P(q^1,p^1)|^{1/2} \}\\
 = d_\infty (q,&p) \le 2 \delta.
 \end{split}
\end{eqnarray*}
Hence
\begin{eqnarray*}
d_\infty(\bar{q},z') \le d_\infty(\bar{q},z)+d_\infty(z,z') \le 3 \delta,
\end{eqnarray*}
which implies that $d_\infty(\bar{q},\tilde{I}_u(p)) \le 3 \delta$. Thus \eqref{inc} holds.
It follows from \eqref{inc1}, \eqref{inc3} and \eqref{inc} that
\begin{equation}\label{inc2}
T^\delta_u(a) \subset \tilde{T}^{ \delta}_u(q) \subset \tilde{T}^{3 \delta}_u(p) \quad \mbox{and} \quad T^\delta_v(a') \subset \tilde{T}^\delta_v(p).
\end{equation}
On the other hand,
\begin{eqnarray}\label{ltrans}
\tilde{T}^{3\delta}_u(p) \cap \tilde{T}^\delta_v(p)= \tau_p(\tilde{T}^{3\delta}_u \cap \tilde{T}^\delta_v).
\end{eqnarray}
If $A \subset \mathbb{G}$ then for every $p \in B_n(0,R)$
\begin{equation}\label{diamtransl}
\text{diam}_E(\tau_p(A)) \le \sqrt{2(1+C_{R,n}^2)} \text{diam}_E(A).
\end{equation}
Namely,
\begin{eqnarray*}
\text{diam}_E(\tau_p(A))=\sup_{q,\bar{q} \in A}|p \cdot q-p \cdot \bar{q}|
\end{eqnarray*}
and 
\begin{eqnarray*}
\begin{split}
|p \cdot q-p \cdot \bar{q}|^2&=|p^1+q^1-p^1-\bar{q}^1|^2+|p_n+q_n+P(p^1,q^1)-p_n-\bar{q}_n -P(p^1,\bar{q}^1)|^2\\
&\le |q^1-\bar{q}^1|^2+2|q_n-\bar{q}_n|^2+2|P(p^1,q^1-\bar{q}^1)|^2\\
&\le |q^1-\bar{q}^1|^2+2|q_n-\bar{q}_n|^2+2 C_{R,n}^2|q^1-\bar{q}^1|^2\\
&\le 2(1+ C_{R,n}^2)|q-\bar{q}|^2.
\end{split}
\end{eqnarray*}
Hence by \eqref{inc2}, \eqref{ltrans} and \eqref{diamtransl},
\begin{eqnarray*}
\begin{split}
\text{diam}_E(T^\delta_u(a) \cap T^\delta_v(a'))& \le \text{diam}_E(\tilde{T}^{3\delta}_u(p) \cap \tilde{T}^\delta_v(p)) \\
& = \text{diam}_E(\tau_p(\tilde{T}^{3\delta}_u \cap \tilde{T}^\delta_v))\\
& \lesssim \text{diam}_E(\tilde{T}^{3\delta}_u \cap \tilde{T}^\delta_v).
\end{split}
\end{eqnarray*}
Since by Lemma \ref{tubeessEucl} 
\begin{equation*}
\tilde{T}^{3\delta}_u\subset T^{E,3C\delta}(\tilde{I}_u) \quad  \mbox{and} \quad \tilde{T}^\delta_v \subset T^{E,C\delta}(\tilde{I}_v),
\end{equation*}
we have by the diameter estimate for Euclidean tubes,
\begin{equation*}
\text{diam}_E(\tilde{T}^{3\delta}_u \cap \tilde{T}^\delta_v) \le diam_E(T^{E,3C\delta}(\tilde{I}_u) \cap T^{E,C\delta}(\tilde{I}_v)) \le b_n \frac{C\delta}{|u-v|}.
\end{equation*}
Hence \eqref{Ax3m21} follows.
\end{proof}

\textbf{Axiom 4}: Let $u, v \in B_{n-1}(0,r_R)$ be such that $|u-v| \le \delta$. We want to show that for every $a \in B_n(0,R)$
\begin{equation}\label{Ax4m21}
T^\delta_u(a) \subset T^{K\delta,l}_v(a),
\end{equation}
where $ T^{K\delta,l}_v(a)$ is the $K\delta$ neighbourhood of 
\begin{equation*}
I^l_v(a)=\{[tv+a^1,t+a_n+tP(a^1,v)]: - \frac{1}{\sqrt{1+|v|^2}} \le t \le \frac{2}{\sqrt{1+|v|^2}} \}
\end{equation*}
and $K$ is a constant depending only on $n$ and $R$.

To prove \eqref{Ax4m21}, it suffices to show
\begin{equation}\label{Ax4m211}
T^\delta_u \subset T^{K \delta,l}_v
\end{equation}
since then \eqref{Ax4m21} follows by applying the left translation by $a$.
By Lemma \ref{tubeessEucl}, we know that
\begin{equation*}
T^\delta_u \subset T^{E,C\delta}(I_u). 
\end{equation*}
Since $|u-v| \le \delta$, we have
\begin{equation*}
T^{E,C\delta}(I_u)  \subset T^{E,C'\delta}(I_v),
\end{equation*}
where $C'$ is another constant depending only on $n$ and $R$.
On the other hand,
\begin{equation*}
T^{E,C'\delta}(I_v) \subset T^{O,C'\delta}(I^l_v)
\end{equation*}
and again by Lemma \ref{tubeessEucl},
\begin{equation*}
T^{O,C'\delta}(I^l_v) \subset T^{C'\delta/c,l}_v.
\end{equation*}
Hence \eqref{Ax4m211} holds with $K=C'/c$.

\textbf{Axiom 2}: In \cite{Venieri} we considered (bounded) Kakeya sets in the Heisenberg group, which we can identify with $\mathbb{R}^n=\mathbb{R}^{2N} \times \mathbb{R}$, equipped with the Kor\'anyi metric $d_H$ (which is bi-Lipschitz equivalent to the Carnot-Carath\'eodory metric). The proof of Theorem 1 in \cite{Venieri} contains the proof that Axiom 2 holds (with $\theta=0$) when the tubes are defined with respect to the Euclidean metric and the balls with respect to $d_H$. We could use essentially the same proof also for tubes defined with respect to $d_H$.
Actually in the proof of Theorem 1 we proved directly that \eqref{ax2union} holds, that is Axiom 2 with union of balls instead of one single ball. Proving only Axiom 2 would have been easier since we would have not needed to look at how the angle between horizontal segments through a point $x \in I_u(a)$ and $I_u(a)$ varies depending on $x$.
Essentially the same proof can be used in any Carnot group $\mathbb{G}$ of step $2$ with $m_2=1$ endowed with the metric $d_\infty$. We show it here. 

\begin{lem}\label{ax2m21}
There exist two constants (depending only on $n$ and $R$) $1 \le K < \infty$, $0 < K' < \infty$ such that the following holds. Let $a \in B_n(0,R)$, $u \in B_{n-1}(0,r_R)$, $p \in I_u(a)$ and $\delta \le r \le 2 \delta$. If
\begin{equation}\label{H1Iua}
\mathcal{H}^1_E(I_u(a) \cap B_\infty(p,r))= M
\end{equation}
for some $M>0$, then
\begin{equation}\label{Ax2m21}
\mathcal{L}^n(T^\delta_u(a) \cap B_\infty(p, K r)) \ge K' M \mathcal{L}^n(T^\delta_u(a)).
\end{equation}
\end{lem}

\begin{proof}
Let $H_p$ be the horizontal hyperplane passing through $p$, that is
\begin{equation*}
H_p=\{ q=[q^1,q_n]: q_n=p_n+P(p^1,q^1)\}.
\end{equation*} 
As was seen in the proof of Lemma \ref{tubeessEucl}, the angle between $I_u(a)$ and $H_p$ is $\ge \bar{\theta}_{R,n}  - \theta_{R,n} >0$ (since the angle between $I_u(a)$ and the $x_1, \dots, x_{n-1}$-hyperplane is $\ge \bar{\theta}_{R,n}$ and the angle between the  $x_1, \dots, x_{n-1}$-hyperplane and $H_p$ is $\le \theta_{R,n}$).

Let $b=[b^1,b_n] \in I_u(a) \cap B_\infty(p,r)$. Then we will show that any segment starting from $b$, with direction parallel to $H_p$ and contained in the tube $T^{E,C\delta}(I_u(a)) \supset T^\delta_u(a)$ is also contained in $B_\infty(p,Kr)$, where $K=1+C c_{R,n}$, $c_{R,n}=1/\sin(\bar{\theta}_{R,n}- \theta_{R,n})$.
Let $P$ be the hyperplane orthogonal to $I_u(a)$ passing through the origin and let $S_P$ be the unit sphere contained in $P$. For $e \in S_P$ and $s\ge 0$ let
\begin{equation*}
\sigma^b_e(s)= [se+b^1, b_n+sP(p^1,e)]
\end{equation*}
be a point in any segment starting from $b$ with direction parallel to $H_p$. It is contained in $T^{E,C\delta}(I_u(a))$ for those $s$ such that $d_E(\sigma^b_e(s),I_u(a)) \le C \delta$. Since the angle between $I_u(a)$ and $\{\sigma^b_e(s): s \ge 0\}$ is $\ge \bar{\theta}_{R,n}-\theta_{R,n}$, this implies
\begin{equation}\label{bsigma}
|b-\sigma^b_e(s)| \le \frac{d_E(\sigma^b_e(s),I_u(a))}{\sin(\bar{\theta}_{R,n}-\theta_{R,n})} \le C c_{R,n} \delta.
\end{equation}
Since $b \in B_\infty(p,r)$, we know that
\begin{align}\label{binBpr}
d_\infty(b,p)= \max \{ |b^1-p^1|, \epsilon |b_n-p_n-P(p^1,b^1)|^{1/2} \} \le r.
\end{align}
The $d_\infty$ distance from $p$ to $\sigma^b_e(s)$ is
\begin{align*}
d_\infty(\sigma^b_e(s),p)&= \max \{ |b^1+s e-p^1|, \epsilon | b_n+sP(p^1,e)-p_n-P(p^1,b^1+se)|^{1/2} \}\\
&=\max\{ |b^1+s e-p^1|, \epsilon |b_n-p_n-P(p^1,b^1)|^{1/2} \}.
\end{align*}
Since by \eqref{bsigma} $|b^1+se-b^1|=|se| \le |b-\sigma^b_e(s)| \le C c_{R,n}\delta$, and by \eqref{binBpr} $|b^1-p^1| \le r$, we have
\begin{align*}
|b^1+se-p^1| \le |se|+|b^1-p^1| \le C c_{R,n}  \delta+r \le K r.
\end{align*}
Moreover, by \eqref{binBpr} 
\begin{equation*}
 \epsilon |b_n-p_n-P(p^1,b^1)|^{1/2} \le r.
\end{equation*}
Thus it follows that $d_\infty(\sigma^b_e(s),p) \le Kr$, which means $\sigma^b_e(s) \in B_\infty(p,Kr)$.

Since for every $b \in I_u(a) \cap B_\infty(p,r)$, the segment $\{\sigma^b_e(s): d_E(\sigma^b_e(s), I_u(a)) \le C \delta \}$ is contained in $B_\infty(p,Kr)$, it follows from \eqref{H1Iua} that for every segment $I$ parallel $I_u(a)$ contained in $ T^{O,c\delta}(I_u(a)) \subset T^\delta_u(a)$ we have
\begin{equation*}
\mathcal{H}^1_E(I \cap T^{O,c\delta}(I_u(a)) \cap B_\infty(p, Kr)) \ge M.
\end{equation*}
To get
\begin{equation}\label{ax21}
\mathcal{L}^n(T^{O,c\delta}(I_u(a)) \cap B_\infty (p,Kr)) \gtrsim M \mathcal{L}^n(T^{O,c\delta}(I_u(a))),
\end{equation}
we use the following formula valid for any Lebesgue measurable and integrable function $f$ on $\mathbb{R}^n$,
\begin{equation}\label{fint}
\int_{\mathbb{R}^n} f(p) d \mathcal{L}^n p = \int_{S^{n-2}} \left( \int_0^\infty r^{n-2} \int_{-\infty}^\infty f(ry,p_n) d p_n dr \right) d \sigma^{n-2}(y),
\end{equation}
which is obtained by changing the order of integration in the formula giving integration in polar coordinates in translates of the $x_1, \dots, x_{n-1}$-hyperplane along the $x_n$-axis:
\begin{equation*}
\int_{\mathbb{R}^n} f(p) d \mathcal{L}^n p= \int_{- \infty}^\infty \left( \int_0^\infty r^{n-2} \int_{S^{n-2}} f(ry,p_n) d \sigma^{n-2}(y) dr \right) dp_n.
\end{equation*}
Indeed, consider now the coordinate system where $a$ is the origin and $I_u(a)$ is contained in the $x_n$-axis (we can reduce to this situation by translating in Euclidean sense and rotating $I_u(a)$). If we let $f= \chi_{T^{O,c\delta}(I_u(a)) \cap B_\infty (p,Kr)}$ in \eqref{fint} we obtain
\begin{align}\label{t1}
\mathcal{L}^n(T^{O,c\delta}(I_u(a)) \cap B_\infty (p,Kr))= \int_{S^{n-2}} \left( \int_0^\infty r^{n-2} \int_{-\infty}^\infty f(ry,z) dz dr \right) d \sigma^{n-2}(y).
\end{align}
Every segment $I$ parallel $I_u(a)$ contained in $ T^{O,c\delta}(I_u(a))$ is contained in
\begin{equation*}
I_{ry}=\{ (ry, z): z \in \mathbb{R} \}
\end{equation*}
for some $y \in S^{n-2}$ and $0 \le r \le c \delta$. Hence
\begin{equation*}
\int_{-\infty}^\infty f(ry,z) dz = \mathcal{H}^1_E(I_{ry} \cap T^{O,c\delta}(I_u(a)) \cap B_\infty(p, Kr)) \ge M
\end{equation*}
and in \eqref{t1} we need only to integrate over $r$ such that $0 \le r \le c \delta$. Thus we obtain
\begin{align*}
\mathcal{L}^n(T^{O,c\delta}(I_u(a)) \cap B_\infty (p,Kr)) \ge M \int_{S^{n-2}}  \int_0^{c \delta} r^{n-2} dr d \sigma^{n-2}(y) \approx M \delta^{n-1} \approx M \mathcal{L}^n(T^{O,c\delta}(I_u(a))),
\end{align*}
which proves \eqref{ax21}. From \eqref{ax21} we can then get \eqref{Ax2m21} since
\begin{eqnarray*}
\begin{split}
\mathcal{L}^n(T^\delta_u(a) \cap B_\infty(p,Kr)) &\ge \mathcal{L}^n(T^{O,c\delta}(I_u(a)) \cap B_\infty (p,Kr))\\
&\gtrsim  M \mathcal{L}^n(T^{O,c\delta}(I_u(a)))\\
& \approx M \mathcal{L}^n(T^\delta_u(a)).
\end{split}
\end{eqnarray*}
\end{proof}

Thus Axioms 1-4 hold and Bourgain's method yields the lower bound $\frac{n+3}{2}$ for the Hausdorff dimension (with respect to any metric that is bi-Lipschitz equivalent to the Carnot-Carath\'eodory metric) of any bounded LT-Kakeya set in $\mathbb{G}$. This completes the proof of Theorem \ref{m21} for LT-Kakeya sets.

\begin{rem}
We do not know if Axiom 5 holds or not since the direction of a segment $I_u(a)$ depends not only on $u$ but also on $a$. In Axiom 5 we consider tubes with central segments $I_u(a)$, $I_{u_j}(a_j)$ and $I_{u_i}(a_i)$ such that $|u-u_j| \ge \beta/8$, $|u-u_i| \ge \beta/8$ and $\delta<|u_j-u_i| \le \beta$. In general, however, the angle between the directions of $I_u(a)$ and $I_{u_j}(a_j)$ is not comparable to $|u-u_j|$ (and the same for the angle between the other directions).
\end{rem}

\textbf{Proof of Theorem \ref{m21} for Kakeya sets}:
To prove Theorem \ref{m21} for the classical bounded Kakeya sets, we define as in Remark \ref{classicKak}
\begin{equation*}
I_u(a)=I_u+a \quad \mbox{and} \quad T^\delta_u(a)=\{p: d_\infty(p,I_u(a)) \le \delta\}.
\end{equation*}
Here we can let
\begin{equation}\label{rR2}
r_R < \frac{1}{\sqrt{1+C^2_{R,n}}}.
\end{equation}
Note that for LT-Kakeya sets we needed a smaller $r_R$ as chosen in \eqref{rR1}. The reason is that here, given $a \in B_n(0,R)$ and $u \in B_{n-1}(0,r_R)$, the angle between the segment $I_u(a)$ and the $x_1, \dots, x_{n-1}$-hyperplane is
\begin{equation}\label{newbartheta}
\arccos \frac{|u|}{\sqrt{1+|u|^2}} \ge \arccos(r_R) =: \bar{\theta}_{R,n},
\end{equation}
which is greater than the angle $\bar{\theta}_{R,n}$ in \eqref{thetabar1}. With the choice of $r_R$ made above in \eqref{rR2} we have $\bar{\theta}_{R,n}> \theta_{R,n}$, where $\theta_{R,n}$ is as in Lemma \ref{lemang}.  

First observe that also with these tubes $T^\delta_u(a)$ we can prove Lemma \ref{tubeessEucl}, that is there exist $0<c<C<\infty$ such that for every $a \in B_n(0,R) $ and every $u \in B_{n-1}(0,r_R)$ we have
\begin{equation}\label{essEucl}
T^{O,c\delta}(I_u(a)) \subset T^\delta_u(a) \subset T^{E,C\delta}(I_u(a)).
\end{equation}
Indeed, there are only few changes in the proof of Lemma \ref{tubeessEucl}. The first change is in \eqref{gammat} since here
\begin{equation*}
\gamma(t)=[a^1+tu,a_n+t],
\end{equation*}
where $0 \le t \le \frac{1}{\sqrt{1+|u|^2}}$. Thus \eqref{bart} becomes
\begin{equation*}
\bar{t}=\frac{q_n-a_n+P(q^1,a^1)}{1-P(q^1,u)}.
\end{equation*}
Taking also here $c=\sin(\bar{\theta}_{R,n} -  \theta_{R,n})$ (with the new value of $\bar{\theta}_{R,n}$ given by \eqref{newbartheta}), we obtain the same result.

\textbf{Axiom 1}: It follows from \eqref{essEucl} that Axiom 1 holds with $T=n-1$.

\textbf{Axiom 2}: Since Lemma \ref{ax2m21} can be proved in the same way, Axiom 2 holds with $\theta=0$.

\textbf{Axiom 3}: Here Axiom 3 follows directly from \eqref{essEucl} and Axiom 3 for Euclidean tubes. Indeed here the direction of $I_u(a)$ is $[u,1]$, thus
\begin{eqnarray*}
\text{diam}_E(T^\delta_u(a) \cap T^\delta_v(a')) \le \text{diam}_E(T^{E,C\delta}(I_u(a)) \cap T^{E,C\delta}(I_v(a'))) \le b \frac{C\delta}{|u-v|}.
\end{eqnarray*}

\textbf{Axiom 4}: Also Axiom 4 follows directly from \eqref{essEucl} and Axiom 4 for Euclidean tubes. Indeed, if $|u-v|\le \delta$, then 
\begin{eqnarray*}
T^\delta_u(a) \subset T^{E,C\delta}(I_u(a)) \subset T^{E,2C\delta}(I_v(a)) \subset T^{O,2C\delta}(I^l_v(a)) \subset T^{2C\delta/c,l}_v(a),
\end{eqnarray*}
where $I^l_v(a)=\{ [tv+a^1,t+1] : -\frac{1}{\sqrt{1+|v|^2}} \le t \le \frac{2}{\sqrt{1+|v|^2}} \}$. 

\textbf{Axiom 5}: It holds with $\lambda=1$ and $\alpha=n-2$ since it follows again from \eqref{essEucl} and Axiom 5 for Euclidean tubes. Indeed, we can show the following.
\begin{lem}
Let $0<\delta, \beta, \gamma < 1$ and let $T=T^\delta_u(a)$ and $T_j=T^\delta_{u_j}(a_j)$, $j=1, \dots, N$, $a,a_j \in B_n(0,R)$, $u, u_j \in B_{n-1}(0,r_R)$, $T \cap T_j \neq \emptyset$ for every $j$. Suppose that $|u-u_j| \ge \beta/8$ and $|u_j-u_k| > \delta$ for every $j \neq k$. Then for all $j=1, \dots, N$,
\begin{equation}\label{ax5C}
\# \mathcal{I}_j=\# \{i: |u_i-u_j| \le \beta, T_i \cap T_j \neq \emptyset, d_E(T_i \cap T_j , T_j \cap T)  \ge \gamma \} \lesssim \beta \delta^{-1} \gamma^{2-n}.
\end{equation}
\end{lem}

\begin{proof}
As showed in the proof of Lemma \ref{ax5Fs}, we can assume that $\delta$ is much smaller than $\gamma$ and that $\beta> \frac{\delta}{\gamma}$.
We have by \eqref{essEucl} $T=T^\delta_u(a) \subset T^{E,C\delta}(I_u(a))=:T^E$ and $T_j=T^\delta_{u_j}(a_j) \subset T^{E,C\delta}(I_{u_j}(a_j))=:T^E_j$ for every $j$. Fix $j$ and let $ i \in \mathcal{I}_j$.
The directions of $T^E$, $T^E_j$ and $T^E_i$ are respectively $[u,1]$, $[u_j,1]$ and $[u_i,1]$. By assumption $|u-u_j| \ge \beta/8$, $|u-u_i| \ge \beta/8$ and $\delta<|u_i-u_j| \le \beta$. 
We can show that $d_E(T^E \cap T^E_j,T^E_j \cap T^E_i) \gtrsim \gamma$ as was done in the proof of Lemma \ref{ax5Fs}. Hence by axiom 5 for Euclidean tubes we have
\begin{align*}
\# \{i: |u_i-u_j| \le \beta, T^E_i \cap T^E_j \neq \emptyset, d_E(T^E_i \cap T^E_j , T^E_j \cap T^E)  \gtrsim \gamma \} \lesssim \beta \delta^{-1} \gamma^{2-n},
\end{align*}
which implies \eqref{ax5C}.
\end{proof}
 
Hence Theorem \ref{m21} for bounded Kakeya sets follows from Theorem \ref{Wolff}.

\begin{rem}
In \cite{Venieri} in the Heisenberg group we found the same lower bound $\frac{n+4}{2}$ for the Hausdorff dimension of bounded Kakeya sets when $n \le 8$. 
Moreover, we derived a better lower bound $\frac{4n+10}{7}$ for $n \ge 9$ from the Kakeya estimate obtained by Katz and Tao (\cite{Katz&Tao2}) using arithmetic methods (see also Remark \ref{Furstb}). 

Note that in the proof of Lemma \ref{ax2m21} we showed that Axiom 2 holds also with balls defined with respect to $d_\infty$ and tubes defined with respect to the Euclidean metric in the case of a Carnot group of step 2 and $m_2=1$ (see \eqref{ax21}).
Thus we could prove the following, which for Heisenberg groups is the same as Theorem 1 in \cite{Venieri} (and it can be proved as Theorem \ref{bound}). Let $\mathbb{G}$ be a Carnot group of step 2 with $m_2=1$, let $1 \le p <n$, $\beta >0$ such that $n+1-\beta p>0$. If
\begin{equation}\label{fin}
||f^*_\delta||_{L^p(S^{n-1})} \le C_{n,p,\beta} \delta^{- \beta} ||f||_{L^p(\mathbb{R}^n)}
\end{equation}
holds for any $f \in L^p(\mathbb{R}^n)$, then the Hausdorff dimension with respect to a homogeneous metric of any bounded Kakeya set in $\mathbb{G}$ is at least $n+1-\beta p$. Here $f^*_\delta$ denotes the classical Kakeya maximal function in which tubes are defined with respect to the Euclidean metric.

As seen in \eqref{fdp}, Katz and Tao's result shows that \eqref{fin} holds with $p=\frac{4n+3}{7}$ and $\beta=\frac{3n-3}{4n+3} $, hence it implies the lower bound $\frac{4n+10}{7}$ for the Hausdorff dimension (with respect to any homogeneous metric) of any bounded Kakeya set in any Carnot group of step 2 if the second layer has dimension 1. This improves the lower bound $\frac{n+4}{2}$ when $n \ge 9$.
\end{rem}

\bibliography{Venieri_arXiv}{}
\bibliographystyle{plain}

Department of Mathematics and Statistics, P.O. Box 68 (Gustaf H\"allstr\"omin katu 2b), FI-00014 University of Helsinki, Finland\\
E-mail: laura.venieri@helsinki.fi

\end{document}